\documentclass[11pt,reqno]{amsart}

\usepackage{amssymb}
\usepackage{amsthm}
\usepackage{mathrsfs, color}
\usepackage{graphicx}
\usepackage{graphics}
\usepackage{multirow}
\usepackage{fullpage}
\usepackage{amsmath}
\usepackage{enumitem}
\usepackage[normalem]{ulem}
\usepackage{bm}
\usepackage{natbib}

\newcommand{\frakg}{\mathfrak{g}}
\newcommand{\frakh}{\mathfrak{h}}

\newcommand{\Cov}[0]{\mathrm{Cov}}
\newcommand{\Var}[0]{\mathrm{Var}}

\newcommand{\sign}[0]{\mathrm{sign}}

\newcommand{\calA}[0]{\mathcal{A}}
\newcommand{\calB}[0]{\mathcal{B}}
\newcommand{\calC}[0]{\mathcal{C}}
\newcommand{\calD}[0]{\mathcal{D}}
\newcommand{\calF}[0]{\mathcal{F}}
\newcommand{\calG}[0]{\mathcal{G}}
\newcommand{\calH}[0]{\mathcal{H}}

\newcommand{\calS}[0]{\mathcal{S}}

\newcommand{\calV}[0]{\mathcal{V}}
\newcommand{\calX}[0]{\mathcal{X}}
\newcommand{\calY}[0]{\mathcal{Y}}

\newcommand{\E}[0]{\mathbb{E}}
\newcommand{\G}[0]{\mathbb{G}}
\newcommand{\R}[0]{\mathbb{R}}
\newcommand{\U}[0]{\mathbb{U}}

\newcommand{\Prob}[0]{\mathbb{P}}

\renewcommand{\le}{\leqslant}
\renewcommand{\ge}{\geqslant}

\newcommand{\ub}[0]{\underline{b}}

\newcommand{\vertiii}[1]{{\left\vert\kern-0.25ex\left\vert\kern-0.25ex\left\vert #1 
    \right\vert\kern-0.25ex\right\vert\kern-0.25ex\right\vert}}

\renewcommand{\tilde}{\widetilde}
\renewcommand{\hat}{\widehat}

\theoremstyle{plain}
\newtheorem{thm}{Theorem}[section]
\newtheorem{lem}[thm]{Lemma}
\newtheorem{prop}[thm]{Proposition}
\newtheorem{cor}[thm]{Corollary}

\theoremstyle{definition}
\newtheorem{defn}{Definition}[section]

\newtheorem{exmp}{Example}[section]
\newtheorem{rmk}{Remark}[section]

\begin{document}

\title[Jackknife multiplier bootstrap for $U$-processes]{Jackknife multiplier bootstrap: finite sample approximations to the $U$-process supremum with applications}
\thanks{X. Chen is supported by NSF DMS-1404891, NSF CAREER Award DMS-1752614, and UIUC Research Board Awards (RB17092,  RB18099)}

\author[X. Chen]{Xiaohui Chen}
\author[K. Kato]{Kengo Kato}

\date{First arXiv version: August 9, 2017. This version: \today}

\address[X. Chen]{
Department of Statistics, University of Illinois at Urbana-Champaign \\
725 S. Wright Street, Champaign, IL 61874 USA.
}
\email{xhchen@illinois.edu}

\address[K. Kato]{
Department of Statistical Science, Cornell University\\
1194 Comstock Hall, Ithaca, NY 14853 USA.
}
\email{kk976@cornell.edu}

\begin{abstract}
This paper is concerned with finite sample approximations to the supremum of a non-degenerate $U$-process of a general order indexed by a function class. We are primarily interested in situations where the function class as well as the underlying distribution change with the sample size, and the $U$-process itself is not weakly convergent as a process. Such situations arise in a variety of modern statistical problems. We first consider Gaussian approximations, namely, approximate the $U$-process supremum by the supremum of a Gaussian process, and derive coupling and Kolmogorov distance bounds. 
Such Gaussian approximations are, however, not often directly applicable in statistical problems since the covariance function of the approximating Gaussian process is unknown. This motivates us to study bootstrap-type approximations to the $U$-process supremum. We propose a novel jackknife multiplier bootstrap (JMB) tailored to the $U$-process, and derive coupling and Kolmogorov distance bounds for the proposed JMB method. All these results are non-asymptotic, and established under fairly general conditions on function classes and underlying distributions. Key technical tools in the proofs are new local maximal inequalities for $U$-processes, which may  be useful in other problems. 
We also discuss applications of the general approximation results to testing for qualitative features of nonparametric functions based on generalized local $U$-processes. 
\end{abstract}


\maketitle

\section{Introduction}
\label{sec:introduction}

This paper is concerned with finite sample approximations to the supremum of a $U$-process of a general order indexed by a function class. We begin with describing our setting. 
Let $X_1,\dots,X_n$ be independent and identically distributed (i.i.d.) random variables  defined on a probability space $(\Omega, \calA, \Prob)$ and taking values in a measurable space $(S, \calS)$ with common distribution $P$. For a given  integer $r \ge 2$, let $\calH$ be a class of  jointly measurable functions (kernels) $h:S^{r} \to \R$ equipped with a measurable envelope $H$ (i.e., $H$ is a nonnegative function on $S^{r}$ such that $H \ge \sup_{h \in \calH}|h|)$. Consider the associated $U$-process
\begin{equation}
\label{eqn:u-statistic}
U_n(h) := U_{n}^{(r)} (h) := \frac{1}{|I_{n,r}|} \sum_{(i_{1},\dots,i_{r}) \in I_{n,r}} h(X_{i_{1}},\dots,X_{i_{r}}), \ h \in \calH,
\end{equation}
where $I_{n,r} = \{ (i_{1},\dots,i_{r}) : 1 \le  i_{1},\dots,i_{r} \le  n, i_{j} \neq i_{k} \ \text{for} \ j \neq k \}$ and $| I_{n,r} | = n!/(n-r)!$ denotes the cardinality of $I_{n,r}$. 
Without loss of generality, we may assume that each $h \in \calH$ is symmetric, i.e., $h(x_{1},\dots,x_{r}) = h(x_{i_{1}},\dots,x_{i_{r}})$ for every permutation $i_{1},\dots,i_{r}$ of $1,\dots,r$, and the envelope $H$ is symmetric as well. 
Consider the normalized $U$-process
\begin{equation}
\label{eqn:u-process}
\U_n(h) = \sqrt{n} \{ U_n(h) - \E[U_n(h)] \}, \quad h \in \calH.
\end{equation}

The main focus of this paper is to derive finite sample  approximation results for the supremum of the normalized $U$-process, namely, $Z_{n} := \sup_{h \in \calH} \U_{n}(h)/r$, in the case where the $U$-process is \textit{non-degenerate}, i.e., $\Var(\E[h(X_1,\dots,X_{r}) \mid  X_1]) > 0$ for all $h \in \calH$.
The function class $\calH$ is allowed to depend on $n$, i.e., $\calH = \calH_{n}$, and we are primarily interested in situations where the normalized $U$-process $\U_{n}$ is not weakly convergent as a process (beyond finite dimensional convergence). For example, there are situations where $\calH_n$ depends on $n$ but $\calH_{n}$ is further indexed by a parameter set $\Theta$ independent of $n$. In such cases, one can think of  $\U_{n}$ as a $U$-process indexed by $\Theta$ and can consider weak convergence of the $U$-process in the space of bounded functions on $\Theta$, i.e., $\ell^{\infty}(\Theta)$. However, even in such cases, there are 
a variety of statistical problems where the $U$-process is not weakly convergent in $\ell^{\infty}(\Theta)$, even after a proper normalization. The present paper covers such ``difficult'' (and in fact yet more general) problems.

%
$U$-processes are powerful tools for a broad range of statistical applications such as testing for qualitative features of functions in nonparametric statistics \cite{leelintonwhang2009,ghosalsenvandervaart2000,abrevaya-wei2005_JBES}, cross-validation for density estimation \cite{nolanpollard1987}, and establishing limiting distributions of $M$-estimators \citep[see, e.g.,][]{arconesgine1993,sherman1993,sherman1994,delaPenaGine1999}. There are two perspectives on $U$-processes: 1) they are {\it infinite-dimensional} versions of $U$-statistics (with one kernel); 2) they are stochastic processes that are {\it nonlinear} generalizations of empirical processes. Both views are useful in that: 1) statistically, it is of greater interest to consider a rich class of statistics rather than a single statistic; 2) mathematically, we can borrow the insights from empirical process theory to derive limit or approximation theorems for $U$-processes. 
Importantly, however, 1) extending $U$-statistics to $U$-processes requires substantial efforts and different techniques; and 2) generalization from empirical processes to $U$-processes is highly nontrivial especially when $U$-processes are not weakly convergent as processes.  In classical settings where indexing function classes are fixed (i.e., independent of $n$), it is known that Uniform Central Limit Theorems (UCLTs) in the Hoffmann-J{\o}rgensen sense hold for $U$-processes under metric (or bracketing) entropy conditions, where  $U$-processes are  weakly convergent in spaces of bounded functions \citep{nolanpollard1988,arconesgine1993,borovskikh1996,delaPenaGine1999}  (these references also cover degenerate $U$-processes where limiting processes are Gaussian chaoses rather than Gaussian processes). Under such classical settings, \cite{arconesgine1994,zhang2001} study limit theorems for bootstrapping $U$-processes; see also
\cite{bickelfreedman1981,bretagnolle1983,arconesgine1992,huskovajanssen1993, huskovajanssen1993b,janssen1994, dehlingmikosch1994,wangjing2004} as references on bootstraps for $U$-statistics. \cite{ginemason2007} introduce a notion of the local $U$-process  motivated by a density estimator of a function of several variables proposed by \cite{frees1994} and establish a version of UCLTs for local $U$-processes.
More recently, \cite{chen2017a} studies Gaussian and bootstrap approximations for high-dimensional (order-two) $U$-statistics, which can be viewed as $U$-processes indexed by \textit{finite} function classes $\calH_n$ with  increasing cardinality in $n$. To the best of our knowledge, however, no existing work covers the case where the indexing function class $\calH = \calH_n$ 1) may change with $n$; 2) may have infinite cardinality for each $n$; and 3) need not verify UCLTs. This is indeed the situation for many of nonparametric specification testing problems \cite{leelintonwhang2009,ghosalsenvandervaart2000,abrevaya-wei2005_JBES}; see examples in Section \ref{sec:monotonicity_testing} for details.


In this paper, we develop a general non-asymptotic theory for directly approximating the supremum $Z_n = \sup_{h \in \calH} \U_n (h)/r$ without referring a weak limit of the underlying $U$-process $\{\U_n(h) : h \in \calH \}$.
Specifically, we first establish a general Gaussian coupling result to approximate $Z_n$ by the supremum of a Gaussian process $W_P$ in Section \ref{sec:gaussian_approx}. Our Gaussian approximation result builds upon recent development in modern empirical process theory \citep{cck2014_empirical_process,cck_honest_cb,cck2016_empirical_process_coupling} and high-dimensional $U$-statistics \citep{chen2017a}. As a significant departure from the existing literature \citep{ginemason2007,arconesgine1993,cck2014_empirical_process,cck2016_empirical_process_coupling}, our Gaussian approximation for $U$-processes has a multi-resolution nature, which is neither parallel with the theory of $U$-processes with fixed function classes nor that of empirical processes. In particular, unlike $U$-processes with fixed function classes, the higher-order degenerate components are not necessarily negligible compared with the H\'ajek  (empirical) process (in the sense of the Hoeffding projections \cite{hoeffding1948}) and they may impact error bounds of the Gaussian approximation.

However, the covariance function of the Gaussian process $W_P$ depends on the underlying distribution $P$ which is unknown and hence the Gaussian approximation developed in Section \ref{sec:gaussian_approx} is not directly applicable to statistical problems such as computing critical values of a test statistic defined by the supremum of a $U$-process. On the other hand, the (Gaussian) multiplier bootstrap developed in \cite{cck_honest_cb,cck2016_empirical_process_coupling} for empirical processes is not directly applicable to $U$-processes since the H\'ajek  process also depends on $P$ and hence is unknown. Our second main contribution is to develop a fully data-dependent procedure for approximating the distribution of $Z_n$. Specifically, we propose a novel {\it jackknife multiplier bootstrap} (JMB) tailored to $U$-processes in Section \ref{sec:bootstrap}. The key insight of the JMB is to replace the (unobserved) H\'ajek process  by its jackknife estimate \citep[cf.][]{callaertveraverbeke1981}. We  establish finite sample validity of the JMB (i.e., conditional multiplier CLT) with explicit error bounds. As a distinguished feature, our error bounds involve a delicate interplay among {\it all} levels of the Hoeffding projections. In particular, the key innovations are a collection of new powerful local maximal inequalities for level-dependent degenerate components associated with the $U$-process (see Section \ref{sec:local_maximal_inequalities}). To the best of our knowledge, there has been no theoretical guarantee on bootstrap consistency for $U$-processes whose function classes change with $n$ and which do not  converge weakly as processes. Our finite sample bootstrap validity results with explicit error bounds fill this important gap in literature, although we only focus on the supremum functional.

It should be emphasized that our approximation problem is different from the problem of approximating the \textit{whole} $U$-process $\{\U_n(h) : h \in \calH \}$. In testing monotonicity of nonparametric regression functions, \cite{ghosalsenvandervaart2000} consider a test statistic defined by the supremum of a bounded $U$-process of order-two and derive a Gaussian approximation result for the normalized $U$-process. Their idea is a two-step approximation procedure: first approximate the $U$-process by its H\'ajek process and then apply  Rio's coupling result \citep{rio1994a}, which is a Koml\'os-Major-Tusn\'ady (KMT) \citep{kmt1975} type strong approximation for empirical processes indexed by Vapnik-\v{C}ervonenkis (VC) type classes of functions. See also \citep{massart1989_AoP,koltchinskii1994} for extensions of the KMT construction to other function classes.
It is worth noting that the two-step approximation of $U$-processes based on KMT type approximations  in general requires more restrictive conditions on the function class and the underlying distribution in statistical applications. Our regularity conditions on the function class and the underlying distribution for the Gaussian and bootstrap approximations are easy to verify and  are  less restrictive than those required for KMT type approximations since we directly approximate the supremum of a $U$-process rather than the whole $U$-process; in fact, our approximation results can cover examples of statistical applications for which KMT type approximations are not applicable or difficult to apply; see Section \ref{sec:monotonicity_testing}  for details. In particular, both Gaussian and bootstrap approximation results of the present paper allow classes of functions with {\it unbounded} envelopes provided suitable moment conditions are satisfied.

To illustrate the general approximation results for suprema of $U$-processes, we consider the problem of testing qualitative features of the conditional distribution and regression functions in nonparametric statistics \cite{leelintonwhang2009,ghosalsenvandervaart2000,abrevaya-wei2005_JBES}. In Section \ref{sec:monotonicity_testing}, we propose a unified test statistic for specifications (such as monotonicity, linearity, convexity, concavity, etc.) of nonparametric functions based on the {\it generalized local $U$-process} (the name is inspired by \cite{ginemason2007}). Instead of attempting to establish a Gumbel type limiting distribution for the extreme-value test statistic (which is known to have slow rates of convergence; see \cite{hall1991,resnick1987a}), we apply the JMB to approximate the finite sample distribution of the proposed test statistic. Notably, the JMB is valid for a larger spectrum of bandwidths, allows for an unbounded envelope, and the size error of the JMB is decreasing polynomially fast in $n$, which should be contrasted with the fact that tests based on Gumbel approximations have size errors of order $1/\log n$. It is worth noting that \cite{leelintonwhang2009}, who develop a test for the stochastic monotonicity based on the supremum of a (second-order) $U$-process and derive a Gumbel limiting distribution for their test statistic under the null, state a conjecture that a bootstrap resampling method would yield the test whose size error is decreasing polynomially fast in $n$ \citep[][p.594]{leelintonwhang2009}. The results of the present paper formally solve this conjecture for a different version of bootstrap, namely, the JMB, in a more general setting. 
In addition, our general theory can be used to develop a version of the JMB test that is uniformly valid in compact bandwidth sets. Such ``uniform-in-bandwidth'' type results allow one to consider tests with data-dependent bandwidth selection procedures, which are not covered in \cite{ghosalsenvandervaart2000,leelintonwhang2009,abrevaya-wei2005_JBES}.

\subsection{Organization}

The rest of the paper is organized as follows.
In Section \ref{sec:gaussian_approx}, we derive non-asymptotic Gaussian approximation error bounds for the $U$-process supremum in the non-degenerate case.
In Section \ref{sec:bootstrap}, we develop and study a jackknife multiplier bootstrap (with Gaussian weights) tailored to the $U$-process to further approximate the distribution of the $U$-process supremum in a data-dependent manner. 
In Section \ref{sec:monotonicity_testing}, we discuss applications of the general results developed in Sections \ref{sec:gaussian_approx} and \ref{sec:bootstrap} to testing for qualitative features of nonparametric functions based on generalized local $U$-processes. 
In Section \ref{sec:local_maximal_inequalities}, we prove new  {\it multi-resolution} and {\it local} maximal inequalities for degenerate $U$-processes with respect to the degeneracy levels of their kernel. These inequalities are key technical tools in the proofs for the results in the previous sections. 
In Section \ref{sec:proof}, we present the proofs for Sections \ref{sec:gaussian_approx}--\ref{sec:bootstrap}. 
Appendix contains additional proofs, discussions, and auxiliary technical results. 

\subsection{Notation}

For a nonempty set $T$, let $\ell^{\infty}(T)$ denote the Banach space of bounded real-valued functions $f: T \to \R$ equipped with the sup norm $\| f \|_{T} := \sup_{t \in T}|f(t)|$. For a pseudometric space $(T,d)$, let $N(T,d,\varepsilon)$ denote the $\varepsilon$-covering number for $(T,d)$, i.e., the minimum number of closed $d$-balls with radius at most $\varepsilon$ that cover $T$. See \cite[Section 2.1]{vandervaartwellner1996} or \cite[Section 2.3]{ginenickl2016} for details.  For a probability space $(T,\mathcal{T},Q)$ and a measurable function $f: T \to \R$, we use the notation $Q f := \int f dQ$
whenever the integral is defined. For $q \in [1,\infty]$, let $\| \cdot \|_{Q,q}$ denote the $L^{q}(Q)$-seminorm, i.e., $\| f \|_{Q,q} := (Q|f|^{q})^{1/q} := (\int |f|^{q} dQ)^{1/q}$ for finite $q$ while $\| f \|_{Q,\infty}$ denotes the essential supremum of $|f|$ with respect to $Q$.  For a measurable space $(S,\mathcal{S})$ and a positive integer $r$, $S^{r} = S \times \cdots \times S$ ($r$ times) denotes the product space equipped with the product $\sigma$-field $\mathcal{S}^{r}$. 
For a generic random variable $Y$ (not necessarily real-valued), let $\mathcal{L}(Y)$ denote the law (distribution) of $Y$. 
For $a,b \in \R$, let $a \vee b = \max \{ a,b \}$ and $a \wedge b = \min \{ a,b \}$. 
Let $\lfloor a \rfloor$ denote the integer part of $a \in \R$. 
``Constants'' refer to finite, positive, and non-random numbers.

\section{Gaussian approximation for suprema of $U$-processes}
\label{sec:gaussian_approx}

In this section, we derive non-asymptotic Gaussian approximation error bounds for the $U$-process supremum in the non-degenerate case, which is essential for establishing the bootstrap validity in Section \ref{sec:bootstrap}. The goal  is to approximate the supremum of the normalized $U$-process, $\sup_{h \in \calH} \U_{n}(h)/r$, by the supremum of a suitable Gaussian process, and derive bounds on such approximations. 

We first recall the setting. Let $X_{1},\dots,X_{n}$ be i.i.d. random variables defined on a probability space $(\Omega,\calA,\Prob)$ and taking values in a measurable space $(S,\mathcal{S})$ with common distribution $P$. For a technical reason, we assume that $S$ is a separable metric space and $\calS$ is its Borel $\sigma$-field. For a given integer $r \ge 2$, let $\calH$ be a class of symmetric measurable functions $h: S^{r} \to \R$ equipped with a symmetric measurable envelope $H$. 
Recall the $U$-process $\{ U_{n}(h) : h \in \calH \}$ defined in (\ref{eqn:u-statistic}) and its normalized version $\{ \U_{n}(h) : h \in \calH \}$ defined in (\ref{eqn:u-process}).
In applications, the function class $\calH$ may depend on $n$, i.e., $\calH = \calH_{n}$. 
However, in Sections \ref{sec:gaussian_approx} and \ref{sec:bootstrap}, we will derive non-asymptotic results that are valid for each sample size $n$, and therefore suppress the possible dependence of $\calH = \calH_n$ on $n$ for the notational convenience.

We will use the following notation. For a symmetric measurable function $h:  S^{r} \to \R$ and $k=1,\dots,r$, let $P^{r-k}h$ denote the function on $S^{k}$ defined by 
\begin{align*}
P^{r-k}h (x_{1},\dots,x_{k}) &= \E[ h(x_{1},\dots,x_{k},X_{k+1},\dots,X_{r})] \\
&=\int \cdots \int h(x_{1},\dots,x_{k},x_{k+1},\dots,x_{r}) dP(x_{k+1}) \cdots dP(x_{r}), 
\end{align*}
whenever the latter integral exists and is finite for every $(x_{1},\dots,x_{k}) \in S^{k}$ ($P^{0}h = h$). Provided that $P^{r-k}h$ is well-defined, $P^{r-k}h$ is symmetric and measurable.

In this paper, we focus on the case where the function class $\calH$ is \textit{VC (Vapnik-\v{C}ervonenkis) type}, whose formal definition is stated as follows. 

\begin{defn}[VC type class]
\label{defn:vc-type-fns}
A function class $\calH$ on $S^{r}$ with  envelope $H$ is said to be \textit{VC type} with characteristics $(A,v)$ if $
\sup_Q N(\calH, \| \cdot \|_{Q,2}, \varepsilon \|H\|_{Q,2}) \le  (A / \varepsilon)^v$ for all $0 < \varepsilon \le  1$, where $\sup_{Q}$ is taken over all finitely discrete distributions on $S^{r}$. 
\end{defn}

We make the following assumptions on the function class $\calH$ and the distribution $P$. 

\begin{enumerate}
\item[(PM)] The function class $\calH$ is \textit{pointwise measurable}, i.e., there exists a countable subset $\calH' \subset \calH$ such that for every $h \in \calH$, there exists a sequence $h_k \in \calH'$ with $h_k \to h$ pointwise. 

\item[(VC)] The function class $\calH$ is VC type with characteristics $A \ge  (e^{2(r-1)}/16) \vee e$  and $v  \ge  1$ for envelope $H$.
The envelope $H$ satisfies that $H \in L^{q}(P^{r})$ for some $q \in [4,\infty]$ and $P^{r-k}H$ is everywhere finite for every $k=1,\dots,r$. 
\item[(MT)] Let $\calG := P^{r-1} \calH := \{ P^{r-1} h : h \in \calH \}$ and $G:=P^{r-1}H$. There exist (finite) constants 
\[
b_{\frakh} \ge  b_{\frakg} \vee \sigma_{\frakh} \ge  b_{\frakg} \wedge \sigma_{\frakh} \ge  \overline{\sigma}_{\frakg} > 0
\]
 such that the following hold: 
\[
\begin{split}
&\|G\|_{P,q} \le  b_{\frakg}, \qquad \sup_{g \in \calG} \| g \|_{P,\ell}^{\ell} \le \overline{\sigma}_{\frakg}^2 b_{\frakg}^{\ell-2}, \ \ell=2,3,4, \\
&\| P^{r-2} H \|_{P^{2},q} \le  b_{\frakh}, \ \text{and} \ \sup_{h \in \calH} \| P^{r-2}h \|_{P^{2},\ell}^{\ell} \le \sigma_{\frakh}^{2} b_{\frakh}^{\ell-2}, \ \ell=2,4,
\end{split}
\]
where $q$ appears in Condition (VC). 
\end{enumerate}

Some comments on the conditions are in order. Conditions (PM), (VC), and (MT) are inspired by Conditions (A)-(C) in \cite{cck2016_empirical_process_coupling}. 
Condition (PM) is made to avoid measurability difficulties. Our definition of ``pointwise measurability'' is borrowed from Example 2.3.4 in \cite{vandervaartwellner1996}; \cite[p.262]{ginenickl2016} calls a pointwise measurable function class a function class satisfying the \textit{pointwise countable approximation property}. 
Condition (PM) ensures that, e.g., $\sup_{h \in \calH} \U_{n}(h) = \sup_{h \in \calH'} \U_{n}(h)$, so that $\sup_{h \in \calH} \U_{n}(h)$ is a (proper) random variable. See \cite[Section 2.2]{vandervaartwellner1996} for details.

Condition (VC) ensures that $\calG$ is VC type as well with characteristics $4\sqrt{A}$ and $2v$ for envelope $G=P^{r-1}H$; see Lemma \ref{lem:uniform_entropy_numbers_between_GH} ahead.
Since $G \in L^{2}(P)$ by Condition (VC), it is seen from Dudley's criterion on sample continuity of Gaussian processes (see, e.g., \cite[][Theorem 2.3.7]{ginenickl2016}) that the function class $\calG$ is \textit{$P$-pre-Gaussian}, i.e., there exists a tight Gaussian random variable $W_P$ in $\ell^\infty(\calG)$ with mean zero and covariance function
\[
\E[W_P(g) W_P(g')] = \Cov(g(X_{1}), g'(X_{1})), \ g,g' \in \calG. 
\]
Recall that a Gaussian process $W= \{ W(g) : g \in \calG \}$ is a tight Gaussian random variable in $\ell^{\infty}(\calG)$ if and only if  $\calG$ is totally bounded for the intrinsic pseudometric $d_{W}(g,g') = (\E[ (W(g)-W(g'))^{2}])^{1/2}, g,g' \in \calG$, and $W$ has sample paths almost surely uniformly $d_{W}$-continuous \citep[][Section 1.5]{vandervaartwellner1996}. 
In applications, $\calG$ may depend on $n$ and so the Gaussian process $W_P$ (and its distribution) may depend on $n$ as well, although such dependences are suppressed in Sections \ref{sec:gaussian_approx} and \ref{sec:bootstrap}. The VC type assumption made in Condition (VC) covers many statistical applications. However, it is worth noting that in principle, we can derive corresponding results for Gaussian and bootstrap approximations under more general complexity assumptions on the function class beyond the VC type, as our local maximal inequalities for the $U$-process in Theorem \ref{lem:local_max_ineq_dengenerate_uprocess} ahead, which are key technical results in the proofs of the Gaussian and bootstrap approximation results, can cover more general function classes than VC type classes; but the resulting bounds would be more complicated and may not be clear enough. For the clarity of exposition, we focus on VC type function classes and present a Gaussian coupling bound for general function classes in Appendix~\ref{sec:Gaussian_approx_general}. 

Condition (MT) imposes suitable moment bounds on the kernel and its H\'ajek projection. Specifically, this moment condition contains interpolated parameters which control the lower moments (i.e., $L^{2}, L^{3}$, and $L^{4}$ sizes) and the envelopes  of $\calH$ and $\calG$.

Under these conditions on the function class $\calH$ and the distribution $P$, we will first construct a random variable, defined on the same probability space as $X_{1},\dots,X_{n}$, which is equal in distribution to $\sup_{g \in \calG} W_{P}(g)$ and ``close'' to $Z_{n}$ with high-probability. To ensure such constructions,  a common assumption is that the probability space is \textit{rich enough}. For the sake of clarity, we will  assume in Sections \ref{sec:gaussian_approx} and \ref{sec:bootstrap} that the probability space $(\Omega,\calA,\Prob)$ is such that 
\begin{equation}
\label{eq:probability_space}
(\Omega,\calA,\Prob) = (S^{n},\calS^{n},P^{n}) \times (\Xi,\calC,R) \times ((0,1),\calB(0,1), U(0,1)),
\end{equation}
where $X_{1},\dots,X_{n}$ are the coordinate projections of $(S^{n},\calS^{n},P^{n})$, multiplier random variables $\xi_{1},\dots,\xi_{n}$ to be introduced in Section \ref{sec:bootstrap} depend only on the ``second'' coordinate $(\Xi,\calC,R)$, and $U(0,1)$ denotes the uniform distribution (Lebesgue measure) on $(0,1)$ ($\calB(0,1)$ denotes the Borel $\sigma$-field on $(0,1)$). The augmentation of the last coordinate is reserved to generate a $U(0,1)$ random variable independent of $X_{1},\dots,X_{n}$ and $\xi_{1},\dots,\xi_{n}$, which is needed when applying the Strassen-Dudley theorem and its conditional version in the proofs of Proposition \ref{prop:entropy_bounds_vc} and Theorem \ref{thm:coupling_gaussian_mulitiplier_bootstrap}; see Appendix \ref{app:strassen-dudley} for the Strassen-Dudley theorem and its conditional version. We will also assume that the Gaussian process $W_{P}$ is defined on the same probability space (e.g. one can generate $W_{P}$ by the previous $U(0,1)$ random variable), but of course $\sup_{g \in \calG} W_{P}(g)$ is not what we want since there is no guarantee that $\sup_{g \in \calG}W_{P}(g)$ is close to $Z_{n}$.

Now, we are ready to state the first result of this paper. 
Recall the notation given in Condition (MT) and define
\begin{align*}
K_n = v  \log(A \vee n) \quad \text{and} \quad \chi_{n}  &=\sum_{k=3}^{r} n^{-(k-1)/2} \| P^{r-k}H \|_{P^{k},2} K_{n}^{k/2}
\end{align*}
with the convention that $\sum_{k=3}^{r} =0$ if $r=2$.
The following proposition  derives Gaussian coupling bounds for $Z_{n} = \sup_{h \in \calH}  \U_{n}(h)/r$.

\begin{prop}[Gaussian coupling bounds]
\label{prop:entropy_bounds_vc}
Let $Z_n = \sup_{h \in \calH}  \U_n(h)/r$. 
Suppose that Conditions (PM), (VC), and (MT) hold, and that $K_{n}^{3} \le n$. Then, for every $n \ge  r+1$ and $\gamma \in (0,1)$, one can construct a random variable $\tilde{Z}_{n,\gamma}$ such that $\mathcal{L}(\tilde{Z}_{n,\gamma}) = \mathcal{L}(\sup_{g \in \calG} W_P(g))$ and
\[
\Prob ( |Z_{n} - \tilde{Z}_{n,\gamma}| >  C \varpi_n ) \le  C' (\gamma + n^{-1}),
\]
where  $C,C'$ are constants depending only on $r$, and 
\begin{equation}
\varpi_n := \varpi_{n}(\gamma) :=  {(\overline{\sigma}_{\frakg}^2b_{\frakg} K_n^2)^{1/3} \over \gamma^{1/3} n^{1/6}} +\frac{1}{\gamma} \left ( \frac{b_{\frakg}K_{n}}{n^{1/2-1/q}} + \frac{\sigma_{\frakh}K_{n}}{n^{1/2}} + \frac{b_{\frakh} K_{n}^{2}}{n^{1-1/q}} + \chi_{n} \right ). 
\end{equation}
In the case of $q=\infty$, ``$1/q$'' is interpreted as $0$. 
\end{prop}

In statistical applications, bounds on the Kolmogorov distance are often more useful than coupling bounds. For two real-valued random variables $V,Y$, let $\rho (V,Y)$ denote the Kolmogorov distance between the distributions of $V$ and $Y$, i.e., $\rho (V,Y) := \sup_{t \in \R} | \Prob (V \le  t)  - \Prob (Y \le  t)|$. To derive a Kolomogorov distance bound, we will assume that there exists a constant $\underline{\sigma}_{\frakg} > 0$ such that
\begin{equation}
\label{eqn:lower_bound_variance_condition}
\inf_{g \in \calG} \| g - Pg \|_{P,2} \ge \underline{\sigma}_{\frakg}.
\end{equation}
Condition~\eqref{eqn:lower_bound_variance_condition} implies that the $U$-process is non-degenerate. For the notational convenience, let $\tilde{Z} = \sup_{g \in \calG} W_{P}(g)$. 
\begin{cor}[Bounds on the Kolmogorov distance between $Z_n$ and $\sup_{g \in \calG}W_{P}(g)$]
\label{cor:kolmogorov_distance_gaussian_coupling_vc}
Assume that all the conditions in Proposition \ref{prop:entropy_bounds_vc} and \eqref{eqn:lower_bound_variance_condition} hold. Then, there exists a constant $C$ depending only on $r, \overline{\sigma}_{\frakg}$ and $\underline{\sigma}_{\frakg}$ such that 
\[
\rho (Z_{n},\tilde{Z}) \le  C\Bigg \{ \left ( \frac{b_{\frakg}^{2}K_{n}^{7}}{n} \right )^{1/8} +\left ( \frac{b_{\frakg}^{2}K_{n}^{3}}{n^{1-2/q}} \right )^{1/4} +  \left ( \frac{\sigma_{\frakh}^{2}K_{n}^{3}}{n} \right )^{1/4} + \left ( \frac{b_{\frakh} K_{n}^{5/2}}{n^{1-1/q}} \right)^{1/2} +  \chi_{n}^{1/2}K_{n}^{1/4}  \Bigg \}. 
\]
\end{cor}

In particular, if the function class $\calH$ and the distribution $P$ are independent of $n$, then $\rho  (Z_n, \tilde{Z} )= O(\{ (\log n)^{7}/n \}^{1/8} )$.

Condition (\ref{eqn:lower_bound_variance_condition}) is used to apply the ``anti-concentration'' inequality for the Gaussian supremum (see Lemma \ref{lem:AC}), which is a key technical ingredient of the proof of Corollary \ref{cor:kolmogorov_distance_gaussian_coupling_vc}. 
The dependence of the constant $C$ on the variance parameters $\underline{\sigma}_{\frakg}$ and $\overline{\sigma}_{\frakg}$ is not a serious restriction in statistical applications. 
In statistical applications, the function class $\calH$ is often normalized in such a way that each function $g \in \calG$ has (approximately) unit variance. In such cases, we may take $\underline{\sigma}_{\frakg} = \overline{\sigma}_{\frakg} = 1$ or $(\underline{\sigma}_{\frakg},\overline{\sigma}_{\frakg})$ as positive constants independent of $n$; see Section \ref{sec:monotonicity_testing} for details.


\begin{rmk}[Comparisons with Gaussian approximations to suprema of empirical processes]
Our Gaussian coupling (Proposition \ref{prop:entropy_bounds_vc}) and approximation (Corollary \ref{cor:kolmogorov_distance_gaussian_coupling_vc}) results are level-dependent on the Hoeffding projections of the $U$-process $\U_n$ (cf. (\ref{eqn:hoeffding_projections}) and (\ref{eqn:hoeffding_decomp}) for formal definitions of the Hoeffding projections and decomposition). Specifically, we observe that: 1) $\underline{\sigma}_{\frakg}, \overline{\sigma}_{\frakg}, b_{\frakg}$ quantify the contribution from the H\'ajek (empirical) process associated with $\U_n$; 2)  $\sigma_{\frakh}, b_{\frakh}$ are related to the second-order degenerate component associated with $\U_n$; 3) $\chi_n$ contains the effect from all higher order projection terms of $\U_n$. For statistical applications in Section \ref{sec:monotonicity_testing} where the function class $\calH = \calH_n$ changes with $n$, the second and higher order projections terms are not necessarily negligible and we have to take into account the contributions of all higher order projection terms. Hence, the Gaussian approximation for the $U$-process supremum of a general order is not parallel with the approximation results for the empirical process supremum \cite{cck2014_empirical_process,cck2016_empirical_process_coupling}.
\end{rmk}

%

\section{Bootstrap approximation for suprema of $U$-processes}
\label{sec:bootstrap}

The Gaussian approximation results derived in the previous section are often not directly applicable in statistical applications such as computing critical values of a test statistic defined by the supremum of a $U$-process. This is because  the covariance function of the approximating Gaussian process $W_{P}(g), g \in \calG$, is often unknown. 
In this section, we study a Gaussian multiplier bootstrap, tailored to the $U$-process, to further approximate the distribution of the random variable $Z_{n} = \sup_{h \in \calH} \U_{n}(h)/r$ in a data-dependent manner. The Gaussian approximation results will be used as building blocks for establishing validity of the Gaussian multiplier bootstrap. 

We begin with noting that, in contrast to the empirical process case studied in \cite{cck_honest_cb} and \cite{cck2016_empirical_process_coupling}, devising (Gaussian) multiplier bootstraps for the $U$-process is not straightforward.
From the Gaussian approximation results, the distribution of $Z_{n}$ is well approximated by the Gaussian supremum $\sup_{g \in \calG} W_{P}(g)$. Hence, one might be tempted to approximate the distribution of $\sup_{g \in \calG}W_{P}(g)$ by the conditional distribution of the supremum of the the multiplier process 
\begin{equation}
\label{eq:naive_multiplier_process}
\calG \ni g \mapsto \frac{1}{\sqrt{n}} \sum_{i=1}^{n}\xi_{i} \{ g(X_{i}) - \overline{g} \},
\end{equation}
where $\xi_{1},\dots,\xi_{n}$ are i.i.d. $N(0,1)$ random variables independent of the data $X_{1}^{n} := \{ X_{1},\dots,X_{n} \}$ and $\overline{g} = n^{-1} \sum_{i=1}^{n} g(X_{i})$. 
However, a major problem of this approach is that, in statistical applications, functions in $\calG$ are unknown to us since functions in $\calG$ are of the form $P^{r-1}h$ for some $h \in \calH$ and depend on the (unknown) underlying distribution $P$. Therefore, we must devise a multiplier bootstrap properly tailored to the $U$-process.

Motivated by this fundamental challenge, we propose and study the following version of Gaussian multiplier bootstrap.
Let $\xi_1,\dots,\xi_n$ be i.i.d. $N(0,1)$ random variables  independent of the data $X_1^n$ (these multiplier variables will be assumed to depend only on the ``second'' coordinate in the probability space construction (\ref{eq:probability_space})). We introduce the following multiplier process:
\begin{equation}
\label{eq:multiplier_process}
\U_n^\sharp(h) = {1 \over \sqrt{n}} \sum_{i=1}^n \xi_{i} \left[ \frac{1}{|I_{n-1,r-1}|} \sum_{(i,i_{2},\dots,i_{r}) \in I_{n,r}} h(X_{i},X_{i_{2}},\dots,X_{i_{r}}) - U_n(h) \right], \ h \in \calH,
\end{equation}
where $\sum_{(i,i_{2},\dots,i_{r})}$ is taken with respect to $(i_{2},\dots,i_{r})$ while keeping $i$ fixed. 
The process $\{ \U_n^\sharp(h) : h \in \calH \}$ is a centered  Gaussian process conditionally on the data $X_{1}^{n}$ and  can be regarded as a version of the (infeasible) multiplier process (\ref{eq:naive_multiplier_process}) with each $g(X_{i})$ replaced by a jackknife estimate.
In fact, the multiplier process (\ref{eq:naive_multiplier_process}) can be alternatively represented as
\begin{equation}
\label{eq:alternative_naive_multiplier_process}
\calH \ni h \mapsto \frac{1}{\sqrt{n}} \sum_{i=1}^{n}\xi_{i} \{ (P^{r-1}h)(X_{i}) - \overline{P^{r-1}h} \},
\end{equation}
where $\overline{P^{r-1}h} = n^{-1} \sum_{i=1}^{n} P^{r-1}h(X_{i})$. For $x \in S$, denote by $\delta_{x}$ the Dirac measure at $x$ and denote by $\delta_{x} h$ the function on $S^{r-1}$ defined by  $(\delta_{x} h)(x_{2},\dots,x_{r}) =h(x,x_{2},\dots,x_{r})$ for $(x_{2},\dots,x_{r}) \in S^{r-1}$. 
For each $i=1,\dots,n$ and a function $f$ on $S^{r-1}$, let $U_{n-1,-i}^{(r-1)}(f)$ denote the $U$-statistic with kernel $f$ for the sample without the $i$-th observation, i.e.,
\[
U_{n-1,-i}^{(r-1)} (f) = \frac{1}{|I_{n-1,r-1}|} \sum_{(i,i_{2},\dots,i_{r}) \in I_{n,r}} f(X_{i_{2}},\dots,X_{i_{r}}).
\]
Then the proposed multiplier process (\ref{eq:multiplier_process}) can be alternatively written as 
\[
\U_n^\sharp(h) =\frac{1}{\sqrt{n}} \sum_{i=1}^n \xi_{i} \left[ U_{n-1,-i}^{(r-1)} (\delta_{X_{i}}h) - U_n(h) \right],
\]
that is, our multiplier process (\ref{eq:multiplier_process}) replaces each $(P^{r-1}h)(X_{i})$ in the infeasible multiplier process  (\ref{eq:alternative_naive_multiplier_process}) by its jackknife estimate $U_{n-1,-i}^{(r-1)} (\delta_{X_{i}}h)$. 

In practice, we approximate the distribution of $Z_{n}$ by the conditional distribution of the supremum of the multiplier process $Z_{n}^{\sharp}:=\sup_{h \in \calH} \U_{n}^{\sharp}(h)$ given $X_{1}^{n}$, which can be further approximated by Monte Carlo simulations on the multiplier variables. 

To the best of our knowledge, our multiplier bootstrap method for $U$-processes is new in the literature, at least in this generality; see Remark \ref{rem:other_bootstraps} for comparisons with other bootstraps for $U$-processes. We call the resulting bootstrap method  the {\it jackknife multiplier bootstrap} (JMB) for $U$-processes.

Now, we turn to proving validity of the proposed JMB. 
We will first construct couplings $Z_{n}^{\sharp}$ and $\tilde{Z}_{n}^{\sharp} := \tilde{Z}_{n,\gamma}^{\sharp}$ (a real-valued random variable that may depend on the coupling error $\gamma \in (0,1)$) such that: 1) $\mathcal{L}(\tilde{Z}_{n}^{\sharp} \mid X_{1}^{n} ) = \mathcal{L} ( \tilde{Z} )$,
where $\mathcal{L}(\cdot \mid X_{1}^{n})$ denotes the conditional law given $X_{1}^{n}$ (i.e., $\tilde{Z}_{n}^{\sharp}$ is independent of $X_{1}^{n}$ and has the same distribution as $\tilde{Z} = \sup_{g \in \calG}W_{P}(g)$); and at the same time 2) $Z_{n}^{\sharp}$ and $\tilde{Z}_{n}^{\sharp}$ are ``close'' to each other.  Construction of such couplings leads to validity of the JMB. To see this, suppose that  $Z_{n}^{\sharp}$ and $\tilde{Z}_{n}^{\sharp}$ are close to each other, namely, $\Prob (|Z_{n}^{\sharp} - \tilde{Z}_{n}^{\sharp}| > r_{1}) \le  r_{2}$ for some small $r_{1},r_{2} > 0$. To ease the notation, denote by $\Prob_{\mid X_{1}^{n}}$ and $\E_{\mid X_{1}^{n}}$ the conditional probability and expectation given $X_{1}^{n}$, respectively (i.e., the notation $\Prob_{\mid X_{1}^{n}}$ corresponds to taking probability with respect to the ``latter two'' coordinates in (\ref{eq:probability_space}) while fixing $X_{1}^{n}$). Then, 
\[
\Prob \left \{ \Prob_{\mid X_{1}^{n}} (|Z_{n}^{\sharp} - \tilde{Z}_{n}^{\sharp}| > r_{1}) > r_{2}^{1/2} \right \} \le  r_{2}^{1/2}
\]
by Markov's inequality, so that, on the event $\{ \Prob_{\mid X_{1}^{n}} (|Z_{n}^{\sharp} - \tilde{Z}_{n}^{\sharp}| > r_{1}) \le r_{2}^{1/2} \}$ whose probability is at least $1-r_{2}^{1/2}$, for every $t \in \R$,  
\[
\Prob_{\mid X_{1}^{n}} (Z_{n}^{\sharp} \le  t) \le \Prob_{\mid X_{1}^{n}} (\tilde{Z}_{n}^{\sharp} \le  t+r_{1} ) + r_{2}^{1/2} = \Prob  (\tilde{Z} \le  t+r_{1} ) + r_{2}^{1/2},
\]
and likewise $\Prob_{\mid X_{1}^{n}} (Z_{n}^{\sharp} \le  t) \ge \Prob (\tilde{Z} \le  t-r_{1} ) - r_{2}^{1/2}$. Hence, on that event, 
\begin{align*}
\sup_{t \in \R} \left | \Prob_{\mid X_{1}^{n}} (Z_{n}^{\sharp} \le  t) - \Prob (\tilde{Z} \le  t) \right | \le \sup_{t \in \R} \Prob (|\tilde{Z} - t | \le  r_{1}) + r_{2}^{1/2}.
\end{align*}
The first term on the right hand side can be bounded by using the anti-concentration inequality for the supremum of a Gaussian process (cf. \cite[Lemma A.1]{cck2014_empirical_process} which is stated in Lemma \ref{lem:AC} in Appendix \ref{app:supporting_lemmas}), and combining the Gaussian approximation results, we obtain a bound on the Kolmogorov distance between $\mathcal{L}(Z_{n}^{\sharp} \mid X_{1}^{n})$ and $\mathcal{L}(Z_{n})$ on an event with probability close to one, which leads to validity of the JMB.

The following theorem is the main result of this paper and derives bounds on such couplings. 
To state the next theorem, we need the additional notation. 
For a symmetric measurable function $f$ on $S^{2}$, define $f^{\odot 2} = f^{\odot 2}_{P}$ by 
\[
f^{\odot 2}(x_{1},x_{2}) := \int f(x_{1},x) f(x,x_{2}) dP(x).
\]
Let $\nu_{\frakh} := \| (P^{r-2}H)^{\odot 2} \|_{P^{2},q/2}^{1/2}$.

\begin{thm}[Bootstrap coupling bounds]
\label{thm:coupling_gaussian_mulitiplier_bootstrap}
Let $Z_n^\sharp = \sup_{h \in \calH} \U_n^\sharp(h)$. 
Suppose that Conditions (PM), (VC), and (MT) hold. Furthermore, suppose that 
\begin{equation}
\label{eq:growthcondition}
\begin{gathered}
\sigma_{\frakh} K_{n}^{1/2} \le \overline{\sigma}_{\frakg}n^{1/2}, \ \nu_{\frakh} K_{n} \le \overline{\sigma}_{\frakg}n^{3/4-1/q}, \ (\sigma_{\frakh} b_{\frakh})^{1/2} K_{n}^{3/4} \le  \overline{\sigma}_{\frakg}n^{3/4}, \\
b_{\frakh} K_{n}^{3/2} \le \overline{\sigma}_{\frakg} n^{1-1/q}, \ \text{and} \ \chi_{n} \le \overline{\sigma}_{\frakg}. 
\end{gathered}
\end{equation}
Then, for every $n \ge  r+1$ and  $\gamma \in (0,1)$, one can construct  a random variable $\tilde{Z}_{n,\gamma}^\sharp$ such that $\mathcal{L}(\tilde{Z}_{n,\gamma}^\sharp \mid  X_1^n) = \mathcal{L}(\sup_{g \in \calG} W_P(g))$ and 
\[
\Prob(|Z_n^\sharp - \tilde{Z}_{n,\gamma}^\sharp| > C \varpi_n^\sharp) \le  C' (\gamma + n^{-1}),
\]
where $C, C'$ are constants depending only on $r$, and 
\begin{equation}
\label{eq:bootstrap_rate}
\begin{split}
\varpi_n^\sharp := \varpi_{n}^{\sharp} (\gamma) := \frac{1}{\gamma^{3/2}} \Bigg \{ &\frac{ \{  (b_{\frakg} \vee \sigma_{\frakh} ) \overline{\sigma}_{\frakg}K_n^{3/2} \}^{1/2}}{n^{1/4}} + {b_{\frakg} K_n \over n^{1/2-1/q}} +  {(\overline{\sigma}_{\frakg} \nu_{\frakh})^{1/2}K_{n}  \over n^{3/8-1/(2q)}} \\
&\qquad + \frac{\overline{\sigma}_{\frakg}^{1/2}(\sigma_{\frakh} b_{\frakh})^{1/4}K_{n}^{7/8}}{n^{3/8}} 
+ \frac{(\overline{\sigma}_{\frakg}b_{\frakh})^{1/2}K_{n}^{5/4}}{n^{1/2-1/(2q)}} + \overline{\sigma}_{\frakg}^{1/2} \chi_{n}^{1/2}K_{n}^{1/2} \Bigg \}.
\end{split}
\end{equation}
In the case of $q=\infty$, ``$1/q$'' is interpreted as $0$. 
\end{thm}

We note that $\nu_{\frakh}^{q} \le \|P^{r-2}H\|_{P^{2},q}^{q} \le  b_{\frakh}^{q}$, but in our applications $\nu_{\frakh} \ll b_{\frakh}$ and this is why we introduced such a seemingly complicated definition for $\nu_{\frakh}$. To see that $\nu_{\frakh} \le b_{\frakh}$, observe that by the Cauchy-Schwarz and Jensen inequalities, 
\begin{align*}
\nu_{\frakh}^{q} &= \iint \left \{ \int (P^{r-2}H)(x_{1},x) (P^{r-2}H)(x,x_{2}) dP(x) \right \}^{q/2} dP(x_{1})dP(x_{2})\\
&\le \left \{ \iint (P^{r-2}H)^{q/2}(x_{1},x_{2}) dP(x_{1}) dP(x_{2}) \right \}^{2} \le \iint (P^{r-2}H)^{q}(x_{1},x_{2}) dP(x_{1}) dP(x_{2}) \le b_{\frakh}^{q}.
\end{align*}
Condition (\ref{eq:growthcondition}) is not restrictive. In applications, the function class $\calH$ is often normalized in such a way that $\overline{\sigma}_{\frakg}$ is of constant order, and under this  normalization, Condition (\ref{eq:growthcondition}) is a merely necessary condition for the coupling bound (\ref{eq:bootstrap_rate}) to tend to zero. 

The proof of Theorem \ref{thm:coupling_gaussian_mulitiplier_bootstrap} is lengthy and involved. A delicate part of the proof is to sharply bound the sup-norm distance between the conditional covariance function of the multiplier process $\U_{n}^{\sharp}$ and the covariance function of $W_{P}$, which boils down to bounding the  term
\[
\left \| \frac{1}{n} \sum_{i=1}^{n} \{ U_{n-1,-i}^{(r-1)}(\delta_{X_{i}} h) - P^{r-1}h(X_{i}) \}^{2} \right \|_{\calH}.
\]
To this end, we make use of the following observation: for a $P^{r-1}$-integrable function $f$ on $S^{r-1}$, $U_{n-1,-i}^{(r-1)}(f)$ is a $U$-statistic of order $(r-1)$, and denote by $S_{n-1,-i}(f)$  its first Hoeffding projection term.
Conditionally on $X_{i}$, $U_{n-1,-i}^{(r-1)} (\delta_{X_{i}}h) - P^{r-1}h(X_{i}) - S_{n-1,-i}(\delta_{X_{i}}h)$ is a degenerate $U$-process, and we will bound the expectation of the squared supremum of this term conditionally on $X_{i}$ using ``simpler'' maximal inequalities (Corollary \ref{thm:alternative_max_ineq} ahead). On the other hand, the term $n^{-1} \sum_{i=1}^{n} \{ S_{n-1,-i}(\delta_{X_{i}}h) \}^{2}$ is decomposed into 
\[
n^{-1} (\text{non-degenerate $U$-statistic of order 2}) + (\text{degenerate $U$-statistic of order 3}),
\]
where the order of degeneracy of the latter term is $1$, and we will apply ``sharper'' local maximal inequalities (Corollary \ref{cor:local_max_ineq_dengenerate_uprocess_VCtype} ahead) to bound the suprema of both terms. Such a delicate combination of different maximal inequalities turns out to be crucial to yield sharper regularity conditions for validity of the JMB in our applications. In particular, if we bound the sup-norm distance between the conditional
covariance function of $\U_{n}^{\sharp}$ and the covariance function of $W_{P}$ in a cruder way, then this will lead to more restrictive conditions on bandwidths in our applications, especially for  the ``uniform-in-bandwidth" results (cf. Condition (T5$'$) in Theorem \ref{thm:app_local_uproc_bootstrap_kolmogorov_distance_uniform_bandwidth}).

The following corollary derives a ``high-probability'' bound for the Kolmogorov distance between $\mathcal{L}( Z_{n}^{\sharp} \mid X_{1}^{n})$ and $\mathcal{L}(\tilde{Z})$ (here a high-probability bound refers to a bound holding with probability at least $1 - C n^{-c}$ for some constants $C,c$).

\begin{cor}[Validity of the JMB]
\label{cor:bootstrap_validity_vc}
Suppose that Conditions (PM), (VC), (MT), and \eqref{eqn:lower_bound_variance_condition} hold. Let
\[
\begin{split}
\eta_{n} := &\frac{ \{  (b_{\frakg} \vee \sigma_{\frakh})K_n^{5/2} \}^{1/2}}{n^{1/4}} + {b_{\frakg} K_n^{3/2} \over n^{1/2-1/q}} +  { \nu_{\frakh}^{1/2}K_{n}^{3/2}  \over n^{3/8-1/(2q)}} \\
&\qquad + \frac{(\sigma_{\frakh} b_{\frakh})^{1/4}K_{n}^{11/8}}{n^{3/8}} 
+ \frac{b_{\frakh}^{1/2}K_{n}^{7/4}}{n^{1/2-1/(2q)}} + \chi_{n}^{1/2}K_{n}
\end{split}
\]
with the convention that $1/q = 0$ when $q=\infty$. 
Then, there exist constants $C,C'$ depending only on $r, \overline{\sigma}_{\frakg}$, and $\underline{\sigma}_{\frakg}$ such that, with probability at least $1-C \eta_{n}^{1/4}$,
\[
\sup_{t \in \R} \left |  \Prob_{\mid X_{1}^{n}} (Z_{n}^{\sharp} \le  t) - \Prob (\tilde{Z} \le  t) \right | \le  C' \eta_{n}^{1/4}.
\]
\end{cor}

If the function class $\calH$ and the distribution $P$ are independent of $n$, then $\eta_{n}^{1/4}$ is of order $n^{-1/16}$, which is polynomially decreasing in $n$ but appears to be non-sharp. 
Sharper bounds could be derived by improving on $\gamma^{-3/2}$ in front of the $n^{-1/4}$ term in (\ref{eq:bootstrap_rate}). The proof of Theorem \ref{thm:coupling_gaussian_mulitiplier_bootstrap} consists of constructing a ``high-probability'' event on which, e.g., the sup-norm distance between the conditional covariance function of $\U_{n}^{\sharp}$ and the covariance function of $W_{P}$ is small.
To construct such a high-probability event, the current proof repeatedly relies on Markov's inequality, which could be replaced by more sophisticated deviation inequalities. However, this is at the cost of more technical difficulties and more restrictive moment conditions. In addition, we derive a conditional UCLT for the JMB in Appendix~\ref{sec:conditional_UCLT_JMB} when $\calH$ is fixed and $P$ does not depend on $n$.

\begin{rmk}[Connections to other bootstraps]
\label{rem:other_bootstraps}
There are several versions of bootstraps for {\it non-degenerate} $U$-processes. 
The most celebrated one is the {\it empirical bootstrap}
\[
\U_n^*(h) = {\sqrt{n} \over r |I_{n,r}|} \sum_{(i_{1},\dots,i_{r}) \in I_{n,r}} \left \{ h(X_{i_1}^*,\dots,X_{i_{r}}^*) - V_n(h) \right\}, \ h \in \calH,
\]
where $X_1^*,\dots,X_n^*$ are i.i.d. draws  from the empirical distribution $P_{n} = n^{-1} \sum_{i=1}^n \delta_{X_i}$ and $V_n(h) = n^{-r} \sum_{i_{1},\dots,i_{r}=1}^{n} h(X_{i_1},\dots,X_{i_{r}})$ is the $V$-statistic associated with kernel $h$ (cf. \cite{bickelfreedman1981,arconesgine1994,chen2017a}). A slightly different bootstrap procedure 
\[
\U_{n}^{\natural}(h) = n^{-r+1/2} \sum_{1 \le i_{1},\dots,i_{r} \le n} \left\{ h(X_{i_{1}}^{*}, X_{i_{2}},\dots, X_{i_{r}}) - h(X_{i_{1}}, X_{i_{2}},\dots, X_{i_{r}})) \right\}, \ h \in \calH,
\]
is proposed in \cite{arconesgine1992}; see Remark 2.7 therein. If $\calH = \{h\}$ is a singleton and the associated $U$-statistic $U_{n}(h)$ is non-degenerate, then $\U_{n}^{\natural}(h)$ and $\U_n^*(h)$ are asymptotically equivalent in the sense that they have the same weak limit that is given by the centered Gaussian random variable $W_{P}(P^{r-1}h)$; see Theorem 2.4 and Corollary 2.6 in \cite{arconesgine1992}. Since the bootstrap $\U_{n}^{\natural}(h)$ can be viewed as the empirical bootstrap applied to a $V$-statistic estimate of the H\'ajek projection, i.e., $\U_{n}^{\natural}(h) =n^{-1/2} \sum_{i=1}^{n} (\delta_{X_{i}^{*}} - P_{n}) P_{n}^{r-1} h$, 
our JMB is connected to (but still different from) $\U_{n}^{\natural}(h)$ in the sense that we
apply the multiplier bootstrap to a jackknife $U$-statistic estimate of the Hajek projection.
Another example is the {\it Bayesian bootstrap} (with Dirichlet weights)
\[
\U_n^\flat(h) = {\sqrt{n} \over r |I_{n,r}|} \sum_{(i_{1},\dots,i_{r}) \in I_{n,r}} (w_{i_1} \cdots w_{i_r} -1)h(X_{i_1},\dots,X_{i_{r}}), \ h \in \calH,
\]
where $w_{i} = \eta_{i} / (n^{-1} \sum_{j=1}^n \eta_{j})$ for $i=1,\dots,n$ and $\eta_1,\dots,\eta_n$ are i.i.d. exponential random variables with mean one (i.e., $(w_{1},\dots,w_{n})$ follows a scaled Dirichlet distribution) independent of $X_{1}^{n} = \{ X_{1},\dots,X_{n} \}$ \cite{rubin1981,lo1987,masonnewton1992,zhang2001}. 
%
If $\calH$ is a fixed VC type function class and the distribution $P$ is independent of $n$ (hence the distribution of the approximating Gaussian process $W_{P}$ is independent of  $n$), then the conditional distributions (given $X_1^n$) of the empirical bootstrap process $\{\U_n^*(h) : h \in \calH\}$ and the Bayesian bootstrap process $\{\U_n^\flat(h) : h \in \calH\}$ (with Dirichlet weights) are known to have the same weak limit as the $U$-process $\{r^{-1} \U_n(h) : h \in \calH\}$, where the weak limit is the Gaussian process $W_{P} \circ P^{r-1}$ in the non-degenerate case \cite{arconesgine1994,zhang2001}. The proposed multiplier process in (\ref{eq:multiplier_process}) is also connected to the empirical and Baysian bootstraps (or more general randomly reweighted bootstraps) in the sense that the latter two bootstraps also implicitly construct an empirical process whose conditional covariance function is close to that of $W_P$ under the supremum norm \citep[cf.][]{chen2017a}. Recall that the conditional covariance function of $\U_n^\sharp$ can be viewed as a jackknife estimate of the covariance function of $W_P$. For the special case where $r = 2$ and $\calH = \calH_n$ is such that $|\calH_n| < \infty$ and $|\calH_n|$ is allowed to increase with $n$, \cite{chen2017a} shows that the Gaussian multiplier, empirical and randomly reweighted bootstraps ($\U_n^\flat(h)$ with i.i.d. Gaussian weights $w_i \sim N(1,1)$) all achieve similar error bounds. In the $U$-process setting, it would be possible to establish  finite sample validity for the empirical and  more general randomly reweighted bootstraps, but this is  at the price of a much more involved technical analysis which we do not pursue in the present paper.
\end{rmk}

\section{Applications: Testing for qualitative features based on generalized local $U$-processes}
\label{sec:monotonicity_testing}

In this section, we discuss applications of the general results in the previous sections to \textit{generalized local $U$-processes}, which are motivated from testing for qualitative features of functions in nonparametric statistics (see below for concrete statistical problems).  

Let $m \ge  1, r \ge  2$ be fixed integers and let $\calV$ be a separable metric space. Suppose that $n \ge  r+1$, and let $D_{i} =  (X_{i},V_{i}), i=1,\dots,n$ be i.i.d. random variables taking values in $\R^m \times \calV$ with joint distribution $P$ defined on the product $\sigma$-field on $\R^{m} \times \calV$ (we equip $\R^{m}$ and $\calV$ with the Borel $\sigma$-fields). The variable $V_{i}$ may include some components of $X_{i}$. Let $\Phi$ be a class of \textit{symmetric} measurable functions $\varphi:\calV^r \to \R$, and let $L: \R^{m} \to \R$ be a (fixed) ``kernel function'', i.e., an integrable function on $\R^{m}$ (with respect to the Lebesgue measure)  such that $\int_{\R^{m}} L(x)dx = 1$.  For $b > 0$ (``bandwidth''), we use the notation $L_{b} (\cdot) = b^{-m}L(\cdot /b)$. For a given sequence of bandwidths $b_{n} \to 0$, let 
\[
h_{n,\vartheta}(d_{1},\dots,d_{r}):= \varphi(v_1,\dots,v_r) \prod_{k=1}^r L_{b_{n}}(x-x_{k}), \ \vartheta = (x ,\varphi) \in \Theta := \calX  \times \Phi, 
\]
where  $\calX \subset \R^{m}$ is a (nonempty) compact subset.
Consider the $U$-process
\[
U_{n}(h_{n,\vartheta}) := U_{n}^{(r)}(h_{n,\vartheta}) := \frac{1}{|I_{n,r}|} \sum_{(i_{1},\dots,i_{r}) \in I_{n,r}} h_{n,\vartheta}(D_{i_{1}},\dots,D_{i_{r}}),
\]
which we call, following \cite{ginemason2007}, the \textit{generalized local $U$-process}. The indexing function class is $\{ h_{n,\vartheta} : \vartheta \in \Theta \}$ which depends on the sample size  $n$. 
The $U$-process $U_{n}(h_{n,\vartheta})$ can be seen as a process indexed by $\Theta$, but in general is not weakly convergent in the space $\ell^{\infty}(\Theta)$, even after a suitable normalization (an exception is the case where $\calX$ and $\Phi$ are finite sets, and in that case, under regularity conditions,  the vector $\{\sqrt{nb_{n}^{m}} (U_{n}(h_{n,\vartheta}) - P^{r}h_{n,\vartheta}) \}_{\vartheta \in \Theta}$ converges weakly to a multivariate normal distribution). In addition, we will allow the set $\Theta$ to depend on $n$.

We are interested in approximating the distribution of the normalized version of this process
\[
S_n = \sup_{\vartheta \in \Theta} \frac{\sqrt{n b_{n}^{m}}  \{U_{n}(h_{n,\vartheta}) - P^{r}h_{n,\vartheta}\}}{rc_{n}(\vartheta)}, 
\]
where $c_{n}(\vartheta) > 0$ is a suitable normalizing constant.
The goal of this section is to characterize conditions under which the JMB developed in the previous section is consistent for approximating the distribution of $S_{n}$ (more generally we will allow the normalizing constant $c_{n}(\vartheta)$ to be data-dependent).
There are a number of statistical applications where we are interested in approximating distributions of such statistics. We provide a couple of examples. All the test statistics discussed in Examples in \ref{exmp:test_stoc_mono} and \ref{exmp:test_curv} are covered by our general framework. In Examples
\ref{exmp:test_stoc_mono} and \ref{exmp:test_curv}, $\alpha \in (0,1)$ is a nominal level. 

\begin{exmp}[Testing stochastic monotonicity]
\label{exmp:test_stoc_mono}
Let $X,Y$ be real-valued random variables and denote by $F_{Y \mid X}(y \mid x)$ the conditional distribution function of $Y$ given $X$. Consider the problem of testing the stochastic monotonicity
\[
H_0 : F_{Y \mid X}(y \mid x) \le  F_{Y \mid X} (y \mid x') \ \forall y \in \R \ \text{whenever $x \ge  x'$}.
\]
Testing for the stochastic monotonicity is an important topic in a variety of applied fields such as economics \cite{solon1992_AER,bgim2007_Econometrica,EllisonEllison2011_AEJ}. For this problem, \cite{leelintonwhang2009} consider a test for $H_0$ based on a local Kendall's tau statistic, inspired by \cite{ghosalsenvandervaart2000}. Let $(X_{i},Y_{i}), i=1,\dots,n$ be i.i.d. copies of $(X,Y)$. \cite{leelintonwhang2009} consider the $U$-process
\[
U_n(x,y) = {1 \over n (n-1)} \sum_{1 \le  i \neq j \le  n} \{ 1(Y_i \le  y) - 1(Y_j \le  y) \} \sign(X_{i}-X_{j}) L_{b_n}(x-X_{i}) L_{b_{n}}(x-X_{j}), 
\]
where $b_n \to 0$ is a sequence of bandwidths and  $\sign (x) = 1(x > 0) - 1(x < 0)$ is the sign function.
They propose to reject the null hypothesis if $S_n = \sup_{(x,y) \in \calX \times \calY} U_n(x, y)/c_n(x)$
is large, where $\calX,\calY$ are subsets of the supports of $X,Y$, respectively and $c_n(x) > 0$ is a suitable normalizing constant. \cite{leelintonwhang2009} argue that as far as the size control is concerned, it is enough to choose,  as a critical value,  the $(1-\alpha)$-quantile of $S_{n}$ when $X,Y$ are independent, under which $U_{n}(x,y)$ is centered. Under independence between $X$ and $Y$, and under regularity conditions, they derive a Gumbel limiting distribution  for a properly scaled version of $S_{n}$ using techniques from \citep{piterberg1996}, but do not consider bootstrap approximations to $S_{n}$. 
It should be noted that \cite{leelintonwhang2009} consider a slightly more general setup than that described above in the sense that they allow $X_{i}$ not to be directly observed but assume that estimated $X_{i}$ are available, and also cover the case where $X$ is multidimensional. 
\end{exmp}

\begin{exmp}[Testing curvature and monotonicity of nonparametric regression]
\label{exmp:test_curv}
Consider the nonparametric regression model $Y = f(X)  + \varepsilon$ with $\E[ \varepsilon \mid X] = 0$, 
where $Y$ is a scalar outcome variable, $X$ is an $m$-dimensional vector of regressors, $\varepsilon$ is an error term, and $f$ is the conditional mean function $f(x) = \E[ Y \mid X=x]$. We observe i.i.d. copies $V_{i} = (X_{i},Y_{i}), i=1,\dots,n$ of $V=(X,Y)$. We are interested in 
testing for qualitative features (e.g., curvature, monotonicity) of the regression function $f$.

\cite{abrevaya-wei2005_JBES} consider a simplex statistic to test linearity, concavity, convexity of $f$ under the assumption that the conditional distribution of $\varepsilon$ given $X$ is symmetric. 
To define their test statistics, for $x_{1},\dots,x_{m+1} \in \R^{m}$, let $\Delta^{\circ} (x_{1},\dots,x_{m+1})  = \{ \sum_{i=1}^{m+1} a_{i} x_{i} : 0 < a_{j} < 1, j=1,\dots,m+1, \ \sum_{i=1}^{m+1}a_{i} = 1\}$ denote the interior of the simplex spanned by $x_{1},\dots,x_{m+1}$, and define $\calD = \bigcup_{j=1}^{m+2} \calD_{j}$, where
\begin{align*}
\calD_{j} =  \Bigg \{ (x_{1},\dots,x_{m+2}) \in \R^{m \times (m+2)} : 
\begin{split} 
&x_{1},\dots,x_{j-1},x_{j+1},\dots,x_{m+2} \ \text{are affinely independent} \\ 
&\text{and} \ x_{j} \in \Delta^{\circ}(x_{1},\dots,x_{j-1},x_{j+1},\dots,x_{m+2})
\end{split}
\Bigg \}.
\end{align*}
The sets $\calD_{1},\dots,\calD_{m+2}$ are disjoint. For given $v_{i}=(x_{i},y_{i}) \in \R^{m} \times \R, i=1,\dots,m+2$, if $(x_{1},\dots,x_{m+2}) \in \calD$ then there exist a unique index  $j=1,\dots,m+2$  and a unique vector $(a_{i})_{1 \le i \le m+2,i \ne j}$ such that $0 < a_{i} < 1$ for all $i \neq j, \sum_{i \neq j} a_{i}=1$, and $x_{j}=\sum_{i \neq j} a_{i} x_{i}$; then, define $w(v_{1},\dots,v_{m+2}) = \sum_{i \neq j} a_{i} y_{i} - y_{j}$. 
The index $j$ and vector $(a_{i})_{1 \le i \le m+2,i \ne j}$ are functions of $x_{i}$'s. 
The set $\calD$ is symmetric (i.e., its indicator function is symmetric) and $w(v_{1},\dots,v_{m+2})$ is  symmetric in its arguments.


Under this notation, \cite{abrevaya-wei2005_JBES} consider the following {\it localized simplex statistic} 
\begin{equation}
\label{eq:localized_simplex_statistic}
U_n(x) = {1 \over |I_{n,m+2}|} \sum_{(i_{1},\dots,i_{m+2}) \in I_{n,m+2}} \varphi (V_{i_1}, \dots, V_{i_{m+2}}) \prod_{k=1}^{m+2} L_{b_n}(x-X_{i_{k}}),
\end{equation}
where $\varphi(v_1, \dots, v_{m+2}) = 1\{(x_1,\dots,x_{m+2}) \in \calD \} \sign(w(v_1,\dots,v_{m+2}))$, which is a $U$-process of order $(m+2)$.
 To test concavity and convexity of $f$, \cite{abrevaya-wei2005_JBES} propose to reject the hypotheses if $\overline{S}_n = \sup_{x \in \calX} U_n(x)/c_n(x)$ and $\underline{S}_n = \inf_{x \in \calX} U_n(x)/c_n(x)$
are large and small, respectively, 
where $\calX$ is a subset of the support of $X$ and $c_n(x) > 0$ is a suitable normalizing constant. The infimum statistic $\underline{S}_n$ can  be written as the supremum of a $U$-process by replacing $\varphi$ with $-\varphi$, so we will focus on $\overline{S}_{n}$. Precisely speaking, they consider to take discrete deign points $x_{1},\dots,x_{G}$ with $G = G_{n} \to \infty$, and take the supremum or infimum on the discrete grids $\{ x_{1},\dots,x_{G} \}$. \cite{abrevaya-wei2005_JBES} argue that as far as the size control is concerned, it is enough to choose, as a critical value, the $(1-\alpha)$-quantile of $\overline{S}_{n}$ when $f$ is linear, under which $U_{n}(x)$ is centered due to the symmetry assumption on the distribution of $\varepsilon$ conditionally on $X$. Under linearity of $f$, \cite[Theorem 6]{abrevaya-wei2005_JBES} claims to derive a Gumbel limiting distribution for a properly scaled version of $\overline{S}_{n}$, but the authors think that their proof needs a further justification. 
The proof of Theorem 6 in \cite{abrevaya-wei2005_JBES} proves that, in their notation, 
the \textit{marginal} distributions of $\tilde{U}_{n,h}(x_{g}^*)$ converge to $N(0,1)$ uniformly in $g =1,\dots,G$ (see their equation (A.1)), and the covariances between $\tilde{U}_{n,h}(x_{g}^*)$ and $\tilde{U}_{n,h}(x_{g'}^*)$ for $g \neq g'$ are approaching zero faster than the variances, but what they need to show is that the \textit{joint} distribution of $(\tilde{U}_{n,h}(x_{1}^*),\dots,\tilde{U}_{n,h}(x_{G}^*))$ is approximated by $N(0,I_{G})$ in a suitable sense, which is lacking in their proof. An alternative proof strategy is to apply Rio's coupling \cite{rio1994a} to the H\'ajek process associated to $U_{n}$, but it seems non-trivial to apply Rio's coupling since it is non-trivial to verify that the function $\varphi$ is of bounded variation.

On the other hand, \cite{ghosalsenvandervaart2000} study testing monotonicity of $f$ when $m=1$ and $\varepsilon$ is independent of $X$. Specifically, they consider testing whether $f$ is increasing, and propose to reject the hypothesis if $S_{n} = \sup_{x \in \calX} \check{U}_{n}(x)/c_{n}(x)$
is large, where $\calX$ is a subset of the support of $X$, 
\begin{equation}
\check{U}_{n}(x) = \frac{1}{n(n-1)} \sum_{1 \le  i \neq j \le  n} \sign (Y_{j}-Y_{i})\sign (X_{i}-X_{j}) L_{b_{n}}(x-X_{i})L_{b_{n}}(x-X_{j}),
\label{eq:gsv_statistic}
\end{equation}
and $c_{n}(x) > 0$ is a suitable normalizing constant. \cite{ghosalsenvandervaart2000} argue that as far as the size control is concerned, it is enough to choose, as a critical value, the $(1-\alpha)$-quantile of $S_{n}$ when $f \equiv 0$, under which $U_{n}(x)$ is centered. Under $f \equiv 0$ and under regularity conditions, \cite{ghosalsenvandervaart2000} derive a Gumbel limiting distribution for a properly scaled version of $S_{n}$ but do not study bootstrap approximations to $S_{n}$. 

In Appendix~\ref{sec:comments_on_alternative_tests}, we discuss some alternative tests in the literature for concavity/convexity and monotonicity of regression functions. 
\end{exmp}

Now, we go back to the general case. In applications, a typical choice of the normalizing constant $c_{n}(\vartheta)$ is 
$c_n(\vartheta) = b_{n}^{m/2} \sqrt{\Var_{P}(P^{r-1}h_{n,\vartheta})}$ where $\Var_{P}(\cdot)$ denotes the variance under $P$, so that each $b_{n}^{m/2} c_{n}(\vartheta)^{-1}P^{r-1}h_{n,\vartheta}$ is normalized to have unit variance, but other choices (such as $c_{n}(\vartheta) \equiv 1$) are also possible. The choice $c_n(\vartheta) = b_{n}^{m/2} \sqrt{\Var_{P}(P^{r-1}h_{n,\vartheta})}$ depends on the unknown distribution $P$ and needs to be estimated in practice. Suppose in general (i.e., $c_n(\vartheta)$ need not to be $b_{n}^{m/2} \sqrt{\Var_{P}(P^{r-1}h_{n,\vartheta})}$) that there is an estimator $\hat{c}_{n}(\vartheta) = \hat{c}_{n}(\vartheta ; D_{1}^{n}) > 0$ for $c_{n}(\vartheta)$ for each $\vartheta \in \Theta$, and instead of original $S_{n}$, consider
\[
\hat{S}_{n} := \sup_{\vartheta \in \Theta} \frac{\sqrt{nb_{n}^{m}}\{ U_{n}(h_{n,\vartheta}) - P^{r}h_{n,\vartheta}\}}{r \hat{c}_{n}(\vartheta)}.
\]
We consider to approximate the distribution of $\hat{S}_{n}$ by the conditional distribution of the JMB analogue of $\hat{S}_{n}$: $\hat{S}_{n}^{\sharp} := \sup_{\vartheta \in \Theta} b_{n}^{m/2}\U_{n}^{\sharp}(h_{n,\vartheta})/\hat{c}_{n}(\vartheta)$, where 
\[
\U_{n}^{\sharp}(h_{n,\vartheta}) = \frac{1}{\sqrt{n}} \sum_{i=1}^n \xi_{i} \left [
U_{n-1,-i}^{(r-1)} (\delta_{D_{i}} h_{n,\vartheta}) - U_n(h_{n,\vartheta}) \right ],
\ \vartheta \in \Theta,
\]
and $\xi_{1},\dots,\xi_{n}$ are i.i.d. $N(0,1)$ random variables independent of $D_{1}^{n} = \{ D_{i} \}_{i=1}^{n}$. 
Recall that for a function  $f$ on $(\R^{m} \times \mathcal{V})^{r-1}$, $U_{n-1,-i}^{(r-1)}(f)$ denotes the $U$-statistic with kernel $f$ for the sample without the $i$-th observation, i.e., $U_{n-1,-i}^{(r-1)} (f) = |I_{n-1,r-1}|^{-1} \sum_{(i,i_{2},\dots,i_{r}) \in I_{n,r}} f(D_{i_{2}},\dots,D_{i_{r}})$.

Let $\zeta, c_{1},c_{2}$, and $C_{1}$ be given positive constants such that $C_{1} >1$ and $c_{2} \in (0,1)$, and let $q \in [4,\infty]$. Denote by $\calX^{\zeta}$ the $\zeta$-enlargement of $\calX$, i.e., $\calX^{\zeta} := \{ x \in \R^{m} : \inf_{x' \in \calX} | x - x' | \le \zeta \}$ where $|\cdot |$ denotes the Euclidean norm. Let $\Cov_{P}(\cdot,\cdot)$ and $\Var_{P}(\cdot)$ denote the covariance and variance under $P$, respectively. For the notational convenience, for arbitrary $r$ variables $d_{1},\dots,d_{r}$, we use the notation $d_{k:\ell} = (d_{k},d_{k+1},\dots,d_{\ell})$ for $1 \le k \le \ell \le r$. We make the following assumptions. 

\begin{enumerate}[label=(T\arabic*)]
\setcounter{enumi}{0}
\item Let $\calX$ be a non-empty compact subset of $\R^{m}$ such that its diameter is bounded by $C_{1}$. 
\item The random vector $X$ has a  Lebesgue density $p(\cdot)$  such that $\| p \|_{\calX^{\zeta}} \le  C_{1}$. 

\item  Let $L:\R^{m} \to \R$ be a continuous kernel function supported in $[-1,1]^m$ such that the function class $\mathfrak{L} := \{ x \mapsto  L(ax+b) : a \in \R, b \in \R^{m} \}$ is VC type for envelope $\| L \|_{\R^{m}} = \sup_{x \in \R^{m}}|L(x)|$. 

\item Let $\Phi$ be  a pointwise measurable class of symmetric functions $\calV^r \to \R$ that  is VC type with characteristics $(A,v)$ for a finite and symmetric  envelope $\overline{\varphi} \in L^{q}(P^{r})$ such that $\log A \le  C_{1}\log n$ and $v \le  C_{1}$. In addition, the envelope $\overline{\varphi}$ satisfies that 
$( \E[ \overline{\varphi}^{q} (V_{1:r}) \mid X_{1:r}=x_{1:r}]  )^{1/q}  \le  C_{1}$ for all $x_{1:r} \in \calX^{\zeta} \times \cdots \times \calX^{\zeta}$ if $q$ is finite, and $\| \overline{\varphi} \|_{P^{r},\infty} \le  C_{1}$ if $q=\infty$

\item $nb_{n}^{3mq/[2(q-1)]} \ge  C_{1}n^{c_{2}}$ with the convention that $q/(q-1) = 1$ when $q=\infty$, and $2m(r-1)b_{n} \le \zeta/2$. 

\item $b_{n}^{m/2}\sqrt{\Var_{P} (P^{r-1}h_{n,\vartheta})} \ge c_{1}$ for all $n$ and $\vartheta \in \Theta$. 
\item $c_{1} \le c_{n}(\vartheta) \le  C_{1}$ for all $n$ and $\vartheta \in \Theta$.
For each fixed $n$, if $x_{k} \to x$ in $\calX$ and $\varphi_{k} \to \varphi$ pointwise in $\Phi$,  then $c_{n}(x_{k},\varphi_{k}) \to c_{n}(x,\varphi)$.

\item With probability at least $1-C_{1}n^{-c_{2}}$, $\sup_{\vartheta \in \Theta} \left| \frac{\hat{c}_n(\vartheta)}{c_n(\vartheta)}- 1\right| \le  C_1 n^{-c_2}$. 
\end{enumerate}

Some comments on the conditions are in order. Condition (T1) allows the set $\calX$ to depend on $n$, i.e., $\calX= \calX_{n}$, but its diameter is bounded (by $C_{1}$). For example, $\calX$ can be discrete grids whose cardinality increases with $n$ but its diameter must be bounded (an implicit assumption here is that the dimension $m$ is fixed; in fact the constants appearing in the following  results depend on the dimension $m$, so that $m$ should be considered as fixed). Condition (T2) is a mild restriction on the density of $X$. It is worth mentioning that $V$ may take values in a generic measurable space, and even if $V$ takes values in a Euclidean space, $V$ need not be absolutely continuous with respect to the Lebesgue measure (we will often omit  the qualification ``with respect to the Lebesgue measure'').  In Examples \ref{exmp:test_stoc_mono} and \ref{exmp:test_curv},  the variable $V$ consists of the pair of regressor vector and outcome variable, i.e., $V=(X, Y)$ with $Y$ being real-valued, and our conditions allow the distribution of $Y$ to be generic. 
In contrast,  \cite{ghosalsenvandervaart2000,leelintonwhang2009} assume that the \textit{joint} distribution of  $X$ and $Y$ have a continuous density (or at least they require the distribution function of $Y$ to be continuous) and thereby ruling out the case where the distribution of $Y$ has a discrete component. This is essentially because they rely on Rio's coupling  \cite{rio1994a} when deriving limiting null distributions of their test statistics.
Rio's coupling is a powerful KMT \cite{kmt1975} type strong approximation result for general empirical processes, but requires the underlying distribution to be defined on a hyper-cube and to have a density bounded away from zero on the hyper-cube. In contrast, our analysis is conditional on $X$ and we only require some moment conditions and VC type conditions on the function class. Thus our JMB does not require $Y$ to have a density for its validity and thereby having a wider applicability in this respect. 

Condition (T3) is a standard regularity condition on kernel functions $L$.
Sufficient conditions under which $\mathfrak{L}$ is VC type are found in \cite{nolanpollard1987,GineNickl2009_AoP, ginenickl2016}.
Condition (T4) allows the envelope $\overline{\varphi}$ to be unbounded. Condition (T4) allows the function class $\Phi$ to depend on $n$, as long as the VC characteristics $A$ and $v$ satisfy that $\log A \le C_{1}\log n$ and $v \le C_{1}$. For example, $\Phi$ can be a discrete set whose cardinality is bounded by $Cn^{c}$ for some constants $c,C>0$. 
Condition (T5) relaxes bandwidth requirements in \cite{ghosalsenvandervaart2000,leelintonwhang2009} where  $m = 1$ and $q = \infty$. For example, \cite{ghosalsenvandervaart2000} assume $n b_n^2 /(\log n)^{4} \to \infty$ and $b_n  \log n \to 0$ for size control. For the problem of testing for regression/stochastic monotonicity of univariate functions, our test statistic is of order $r=2$. If we choose a bounded kernel (such as the sign kernel), then we only need $n^{-2/3+c} \lesssim b_{n} \lesssim 1$ for some small constant $c > 0$. 
Further,  our general theory allows us to develop a version of the JMB that is uniformly valid in  compact bandwidth sets, which can be used to develop versions of tests that are valid with data-dependent bandwidths in Examples \ref{exmp:test_stoc_mono} and \ref{exmp:test_curv}; see Section \ref{subsec:uniformly_valid_JMB_bandwidth} ahead for details. 


Condition (T6) is a high-level condition and implies the $U$-process to be non-degenerate. Let  $\varphi_{[r-1]}(v_{1},x_{2:r}) := \E[\varphi (v_{1},V_{2:r}) \mid X_{2:r}=x_{2:r}] \prod_{j=2}^{r}  p(x_{j})$, and observe that
\[
(P^{r-1}h_{n,\vartheta}) (x_{1},v_{1}) =L_{b_{n}}(x-x_{1}) \int \varphi_{[r-1]}(v_{1},x-b_{n}x_{2:r})\prod_{j=2}^{r} L(x_{j})  dx_{2:r}
\]
for $\vartheta = (x,\varphi)$, where $x-b_{n}x_{2:r}= (x-b_{n}x_{2},\dots,x-b_{n}x_{r})$. From this expression, 
in applications, it is not difficult to find primitive regularity conditions that guarantee Condition (T6). To keep the presentation concise, however, we assume Condition (T6). 

Condition (T7) is concerned with the normalizing constant $c_{n}(\vartheta)$. 
For the special case where $c_n(\vartheta) = b_{n}^{m/2} \sqrt{\Var_{P}(P^{r-1}h_{n,\vartheta})}$, Condition (T7) is implied by Conditions (T4) and (T6). 
Condition (T8) is also a high-level condition, which together with (T7) implies that there is a uniformly consistent estimate $\hat{c}_{n}(\vartheta)$ of $c_{n}(\vartheta)$ in $\Theta$ with polynomial error rates. Construction of $\hat{c}_{n}(\vartheta)$ is quite flexible: for $c_{n}(\vartheta) = b_{n}^{m/2} \sqrt{\Var_{P}(P^{r-1}h_{n,\vartheta})}$, one natural example is the jackknife estimate 
\begin{equation}
\label{eqn:normalizing_constant_jackknife_estimate}
\hat{c}_{n}(\vartheta)=\sqrt{\frac{b_{n}^{m}}{n} \sum_{i=1}^{n} \left\{ U_{n-1,-i}^{(r-1)}(\delta_{D_{i}}h_{n,\vartheta}) - U_{n}(h_{n,\vartheta}) \right\}^{2}}, \ \vartheta \in \Theta.
\end{equation}
The following lemma verifies that the jackknife estimate (\ref{eqn:normalizing_constant_jackknife_estimate}) obeys Condition (T8) for $c_{n}(\vartheta) = b_{n}^{m/2} \sqrt{\Var_{P}(P^{r-1}h_{n,\vartheta})}$. However, it should be noted that  other estimates for this normalizing constant are possible depending on  applications of interest; see \cite{ghosalsenvandervaart2000,leelintonwhang2009,abrevaya-wei2005_JBES}.

\begin{lem}[Estimation error of the normalizing constant]
\label{lem:rate_normalizing_constant_jackknife_estimate}
Suppose that Conditions (T1)-(T7) hold. Let $c_{n}(\vartheta) = b_{n}^{m/2} \sqrt{\Var_{P}(P^{r-1}h_{n,\vartheta})}, \vartheta \in \Theta$ and $\hat{c}_{n}(\vartheta)$ be defined in (\ref{eqn:normalizing_constant_jackknife_estimate}). Then there exist constants $c,C$ depending only on  $r, m, \zeta, c_{1},c_{2}, C_{1}, L$ such that 
\[
\Prob \left \{ \sup_{\vartheta \in \Theta}\left | \frac{\hat{c}_{n}(\vartheta)}{c_{n}(\vartheta)} - 1 \right | > Cn^{-c} \right \} \le Cn^{-c}.
\]
\end{lem}

Now, we are ready to state finite sample validity of the JMB for approximating the distribution of the supremum of the generalized local $U$-process.

\begin{thm}[JMB validity for the supremum of a generalized local $U$-process]
\label{thm:app_local_uproc_bootstrap_kolmogorov_distance}
Suppose that Conditions (T1)--(T8) hold. Then there exist constants $c,C$  depending only on $r, m, \zeta, c_{1},c_{2}, C_{1}, L$ such that the following holds: for every $n$, there exists a tight Gaussian random variable $W_{P,n}(\vartheta), \vartheta \in \Theta$ in $\ell^{\infty}(\Theta)$ with mean zero and covariance function
\begin{equation}
\E[ W_{P,n} (\vartheta)W_{P,n}(\vartheta') ] = b_{n}^{m}\Cov_{P}(P^{r-1}h_{n,\vartheta},P^{r-1}h_{n,\vartheta'})/\{ c_{n}(\vartheta)c_{n}(\vartheta') \} \label{eq:app_covariance_function}
\end{equation}
for $\vartheta, \vartheta' \in \Theta$, and it follows that
\begin{equation}
\label{eqn:app_local_uproc_bootstrap_kolmogorov_distance}
\begin{split}
&\sup_{t \in \R} \left | \Prob (\hat{S}_{n} \le  t) - \Prob (\tilde{S}_{n} \le  t) \right | \le  Cn^{-c}, \\
&\Prob \left \{ \sup_{t \in \R} \left | \Prob_{\mid D_{1}^{n}} (\hat{S}_{n}^\sharp \le  t) - \Prob (\tilde{S}_{n} \le  t) \right | >  Cn^{-c} \right \} \le Cn^{-c},
\end{split}
\end{equation}
where  $\tilde{S}_{n}:=\sup_{\vartheta \in \Theta}W_{P,n}(\vartheta)$.
\end{thm}

Theorem \ref{thm:app_local_uproc_bootstrap_kolmogorov_distance} leads to the following corollary, which is another form of validity of the JMB.
For $\alpha \in (0,1)$, let $q_{\hat{S}_{n}^\sharp}(\alpha) = q_{\hat{S}_{n}^\sharp}(\alpha ; D_{1}^{n})$ denote the conditional $\alpha$-quantile of $\hat{S}_{n}^\sharp$ given $D_{1}^{n}$, i.e., $q_{\hat{S}_{n}^\sharp} (\alpha) = \inf \left \{ t \in \R : \Prob_{\mid D_{1}^{n}} (\hat{S}_{n}^\sharp \le  t) \ge \alpha \right \}$.

\begin{cor}[Size validity of the JMB test]
\label{cor:app_local_uproc_bootstrap}
Suppose that Conditions (T1)--(T8) hold. Then there exist constants $c,C$  depending only on $r, m, \zeta, c_{1},c_{2}, C_{1}, L$ such that 
\[
\sup_{\alpha \in (0,1)} \left|\Prob \left \{ \hat{S}_n \le  q_{\hat{S}_{n}^\sharp}(\alpha) \right \}  - \alpha \right | \le  C n^{-c}.
\]
\end{cor}

\subsection{Uniformly valid JMB test in bandwidth}
\label{subsec:uniformly_valid_JMB_bandwidth}
A version of Theorem \ref{thm:app_local_uproc_bootstrap_kolmogorov_distance} continues to hold even if we additionally take the supremum over a set of possible bandwidths.
For a given bandwidth $b \in (0,1)$, let 
\[
h_{\vartheta,b} (d_{1},\dots,d_{r})  = \varphi (v_{1},\dots,v_{r}) \prod_{k=1}^{r} L_{b}(x-x_{k}),
\]
and  for a given candidate set of bandwidths $\calB_n \subset [\ub_n, \overline{b}_n]$ with $0 < \ub_n \le  \overline{b}_n < 1$, consider
\[
\begin{split}
&S_{n} := \sup_{(\vartheta,b) \in \Theta \times \calB_{n}} \frac{\sqrt{nb^{m}} \{ U_{n}(h_{\vartheta,b}) - P^{r}h_{\vartheta,b} \}}{r c(\vartheta,b)} \quad \text{and} \\
&\hat{S}_{n} := \sup_{(\vartheta,b) \in \Theta \times \calB_{n}} \frac{\sqrt{nb^{m}} \{ U_{n}(h_{\vartheta,b}) - P^{r}h_{\vartheta,b} \}}{r \hat{c}(\vartheta,b)},
\end{split}
\]
where $c_{n}(\vartheta,b) > 0$ is a suitable normalizing constant and $\hat{c}(\vartheta,b) > 0$ is an estimate of $c(\vartheta,b)$. 
Following a similar argument used in the proof of Theorem \ref{thm:app_local_uproc_bootstrap_kolmogorov_distance}, we are able to derive a version of the JMB test that is also valid uniformly in bandwidth, which opens new possibilities to develop tests that are valid with data-dependent bandwidths in Examples \ref{exmp:test_stoc_mono} and \ref{exmp:test_curv}. For related discussions, we refer the readers to Remark 3.2 in \cite{leelintonwhang2009} for testing  stochastic monotonicity and \cite{einmahlmason2005_AoS} for kernel type estimators.

Consider the JMB analogue of $\hat{S}_{n}$: 
\[
\hat{S}_{n}^\sharp = \sup_{(\vartheta,b) \in \Theta \times \calB_{n}} \frac{b^{m/2}}{\hat{c}_{n}(\vartheta,b)\sqrt{n}}\sum_{i=1}^n \xi_{i} \left[ U_{n-1,-i}^{(r-1)} (\delta_{D_{i}} h_{\vartheta,b}) - U_n(h_{\vartheta,b}) \right].
\]
Let $\kappa_{n} = \overline{b}_n / \ub_n$ denote the ratio of the largest and smallest possible values in the bandwidth set $\calB_{n}$, which intuitively quantifies the size of $\calB_{n}$. To ease the notation and to facilitate comparisons, we only consider $q = \infty$. We make the following assumptions instead of Conditions (T5)--(T8). 

\begin{enumerate}
\item[(T5$'$)] $n \ub_{n}^{3m/2} \ge C_1 n^{c_2} \kappa_{n}^{m(r-2)}$, $\kappa_{n} \le C_1 \ub_{n}^{-1/(2r)}$, and $2m(r-1)\overline{b}_{n} \le \zeta/2$. 
\item[(T6$'$)] $b^{m/2}\sqrt{\Var_{P} (P^{r-1}h_{\vartheta,b})} \ge c_{1}$ for all $n$ and $(\vartheta, b) \in \Theta \times \calB_{n}$. 
\item[(T7$'$)] $c_{1} \le c_{n}(\vartheta,b) \le  C_{1}$ for all $n$ and $(\vartheta, b) \in \Theta \times \calB_{n}$. 
For each fixed $n$, if $x_{k} \to x$ in $\calX$, $\varphi_{k} \to \varphi$ pointwise in $\Phi$, and $b_k \to b$ in $\calB_n$, then $c_{n}(x_{k},\varphi_{k}, b_{k}) \to c_{n}(x,\varphi,b)$.

\item[(T8$'$)] With probability at least $1 - C_1 n^{-c_2}$, $\sup_{(\vartheta, b) \in \Theta \times \calB_{n}} \left| \frac{\hat{c}_n(\vartheta,b)}{c_n(\vartheta,b)} - 1\right| \le  C_1 n^{-c_2}$.
\end{enumerate}

\begin{thm}[Bootstrap validity for the supremum of a generalized local $U$-process: uniform-in-bandwidth result]
\label{thm:app_local_uproc_bootstrap_kolmogorov_distance_uniform_bandwidth}
Suppose that Conditions (T1)-(T4) with $q=\infty$, and Conditions  (T5$'$)--(T8$'$) hold. Then there exist constants $c,C$  depending only on $r, m, \zeta, c_{1},c_{2}, C_{1}, L$ such that the following holds: for every $n$, there exists a tight Gaussian random variable $W_{P,n}(\vartheta,b), (\vartheta, b) \in \Theta \times \calB_n$ in $\ell^{\infty}(\Theta \times \calB_n)$ with mean zero and covariance function
\[
\begin{split}
&\E[ W_{P,n} (\vartheta,b)W_{P,n}(\vartheta',b') ] \\
&\quad =b^{m/2} (b')^{m/2}\Cov_{P}(P^{r-1}h_{\vartheta,b},P^{r-1}h_{\vartheta',b'})/\{ c_{n}(\vartheta,b) c_{n}(\vartheta',b') \}
\end{split}
\]
for $(\vartheta,b), (\vartheta',b') \in \Theta \times \calB_{n}$, 
and the result (\ref{eqn:app_local_uproc_bootstrap_kolmogorov_distance}) continues to hold with $\tilde{S}_{n}:=\sup_{(\vartheta, b) \in \Theta \times \calB_n}W_{P,n}(\vartheta,b)$. 
\end{thm}

If $\ub_{n} = \overline{b}_{n} = b_n$ (i.e., $\calB_{n} = \{b_n\}$ is a singleton set), then Conditions (T5$'$)--(T8$'$) reduce to (T5)--(T8) and Theorem \ref{thm:app_local_uproc_bootstrap_kolmogorov_distance_uniform_bandwidth} covers Theorem \ref{thm:app_local_uproc_bootstrap_kolmogorov_distance} with $q = \infty$ as a special case. Condition (T5$'$) states that the size of the bandwidth set $\calB_{n}$ cannot be too large. Conditions (T6$'$)--(T8$'$) are completely parallel with Conditions (T6)--(T8). Such ``uniform-in-bandwidth'' type results are not covered in \cite{ghosalsenvandervaart2000,leelintonwhang2009,abrevaya-wei2005_JBES}. 

\subsection{A simulation study on testing for monotonicity of regression}
\label{subsec:simulation_test_reg_mono}

We provide a numerical example to verify the size validity of the JMB test for monotonicity of regression in Example~\ref{exmp:test_curv}. We generate i.i.d. univariate covariates $X_{1},\dots,X_{n}$ from the uniform distribution on $[0,1]$ and consider the zero regression function $f \equiv 0$ (which implies that the covariate $X$ and the response $Y$ are stochastically independent). As argued in \cite{ghosalsenvandervaart2000}, $f \equiv 0$ is the hardest case in terms of size control under the null hypothesis $H_{0} : f \mbox{ is increasing on } [0,1]$.
We consider two error distributions: (i) Gaussian distribution $\varepsilon_{i} \sim N(0, 0.1^{2})$; (ii) (scaled) Rademacher distribution $\Prob(\varepsilon_{i} = \pm 0.1) = 1/2$. For both error distributions, the (unnormalized) $U$-process $\check{U}_{n}(x)$ defined in \eqref{eq:gsv_statistic} has mean zero (i.e., $\E[\check{U}_{n}(x)] = 0$ for all $x \in [0,1]$). The Rademacher distribution is not covered in \cite{ghosalsenvandervaart2000}. We use the Epanechnikov kernel $L(x) = 0.75(1-x^{2})$ for $x \in [-1,1]$ and $L(x) = 0$ otherwise, together with bandwidth parameter $b_{n} = n^{-1/5}$. We consider three  sample sizes $n=100,200,500$. For each setup, we generate 2,000 bootstrap samples. We consider test of the form 
\[
\sup_{x \in [0.05, 0.95]} {\check{U}_{n}(x) \over \hat{c}_{n}(x)} > q \Rightarrow \mbox{reject } H_{0}, 
\]
where $\hat{c}_{n}(x)$ is given in \eqref{eqn:normalizing_constant_jackknife_estimate} and the critical value $q$ is calibrated by the JMB. In particular, for any nominal size $\alpha \in (0, 1)$, the value of $q := q(\alpha)$ is chosen as the $(1-\alpha)$-th conditional quantile of the JBM. Empirical rejection probability of the JMB test is obtained by averaging over 5,000 simulations. We observe that the empirical rejection probability is close to the nominal size of the JMB test. Table~\ref{tab:empirical_rejection_prob_test_reg_mono} shows the proportion of rejections at the nominal sizes $\alpha=0.05, 0.10$, and Figure~\ref{fig:test_reg_mono_JMB} shows the JMB approximation of the proportion of rejections uniformly in $\alpha \in (0, 1)$. 

\begin{table}[t]
\caption{Empirical rejection probability of the JMB test for regression monotonicity at the nominal sizes $0.05$ and $0.10$ with Gaussian and Rademacher error distributions.}
\label{tab:empirical_rejection_prob_test_reg_mono}
\begin{tabular}{c|c|cc}
\hline
Nominal size & Sample size & Gaussian & Rademacher \\
\hline
\multirow{3}{*}{$\alpha=0.05$} & $n=100$ & 0.0374 & 0.0372 \\
& $n=200$ & 0.0362 & 0.0408 \\
& $n=500$ & 0.0412 &  0.0430 \\
\hline
\multirow{3}{*}{$\alpha=0.10$} & $n=100$ & 0.0846 & 0.0796 \\
& $n=200$ & 0.0860 & 0.0872 \\
& $n=500$ & 0.0886 &  0.0844 \\
\hline
\end{tabular}
\end{table}

\begin{figure}[h!] 
   \centering
       \includegraphics[scale=0.28]{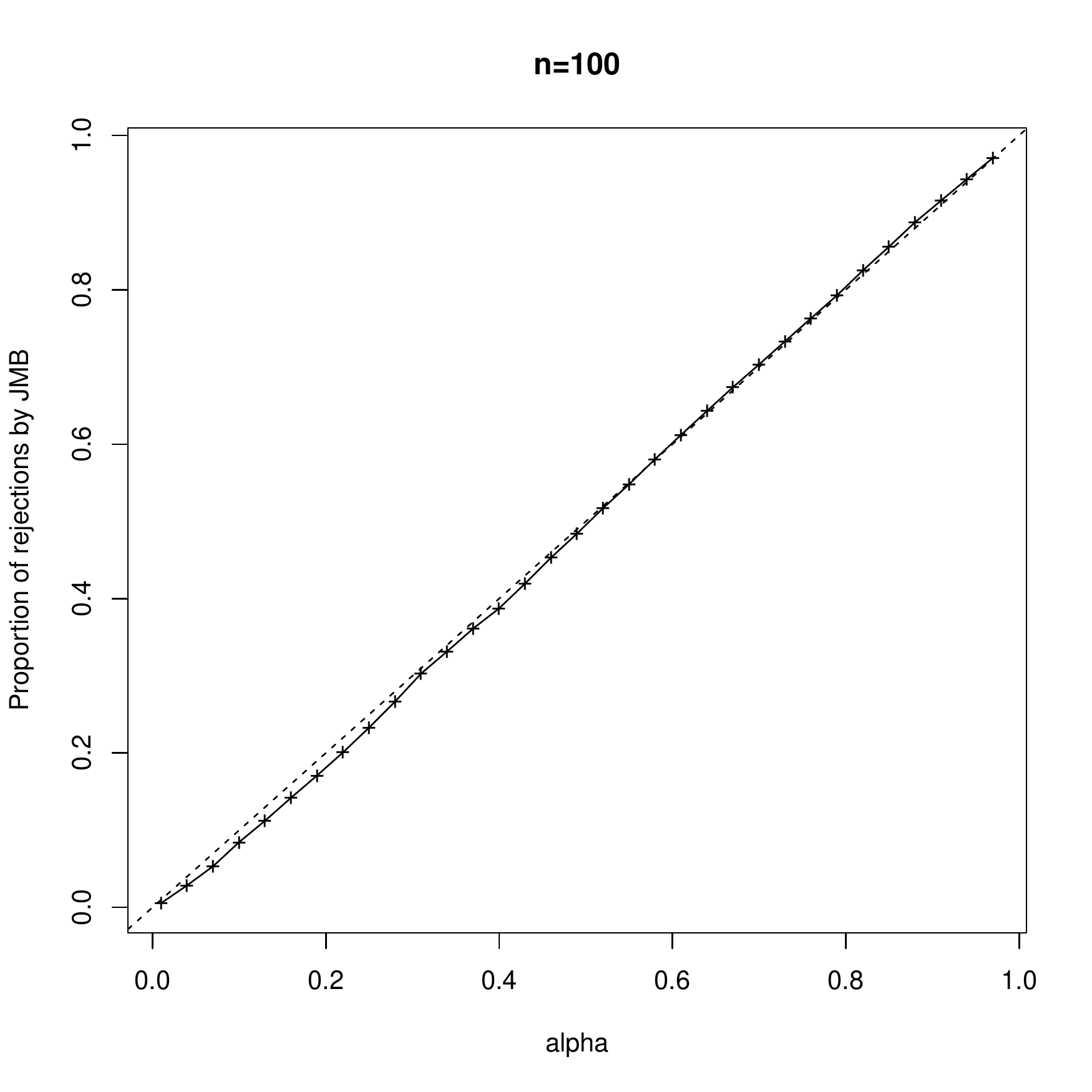}
       \includegraphics[scale=0.28]{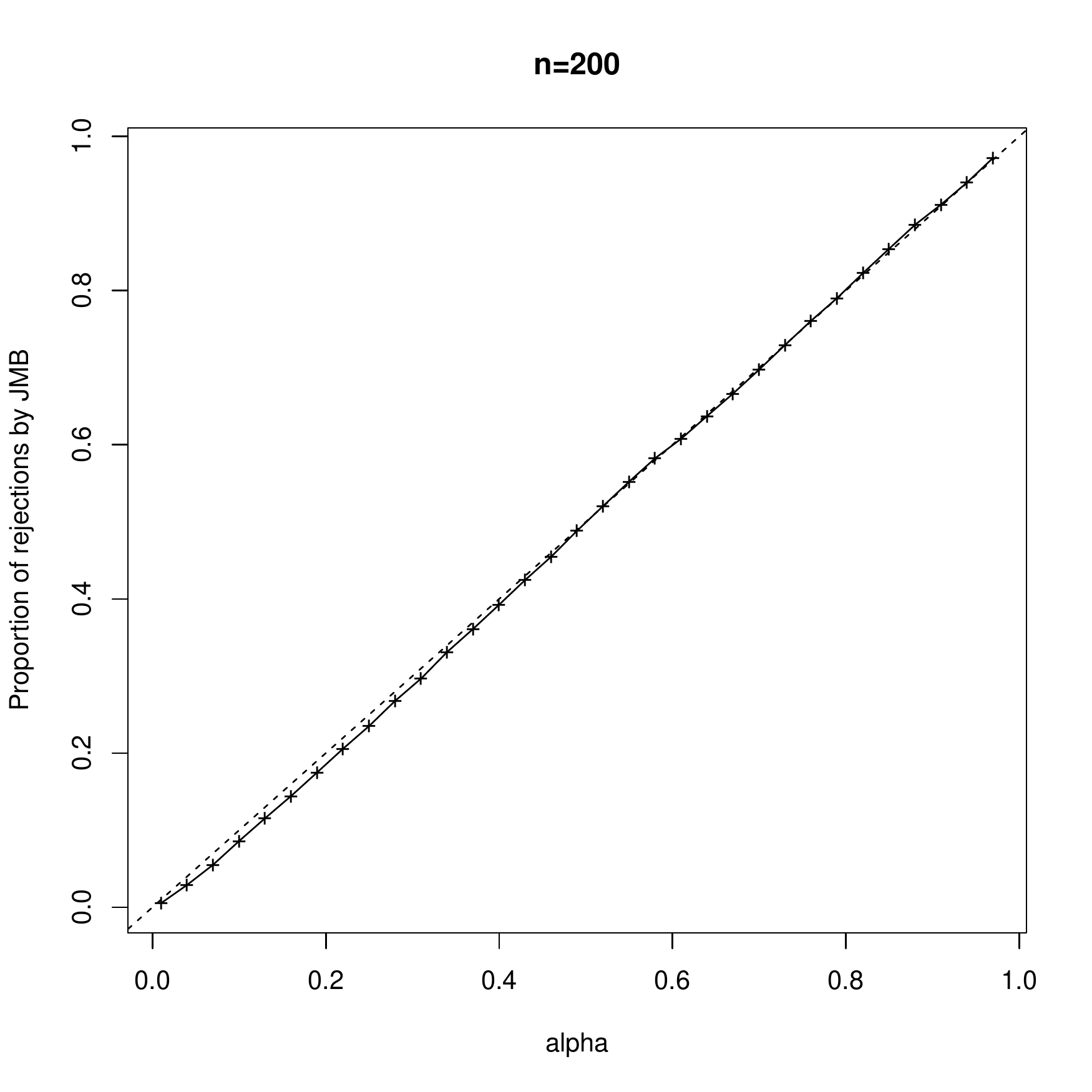}
	  \includegraphics[scale=0.28]{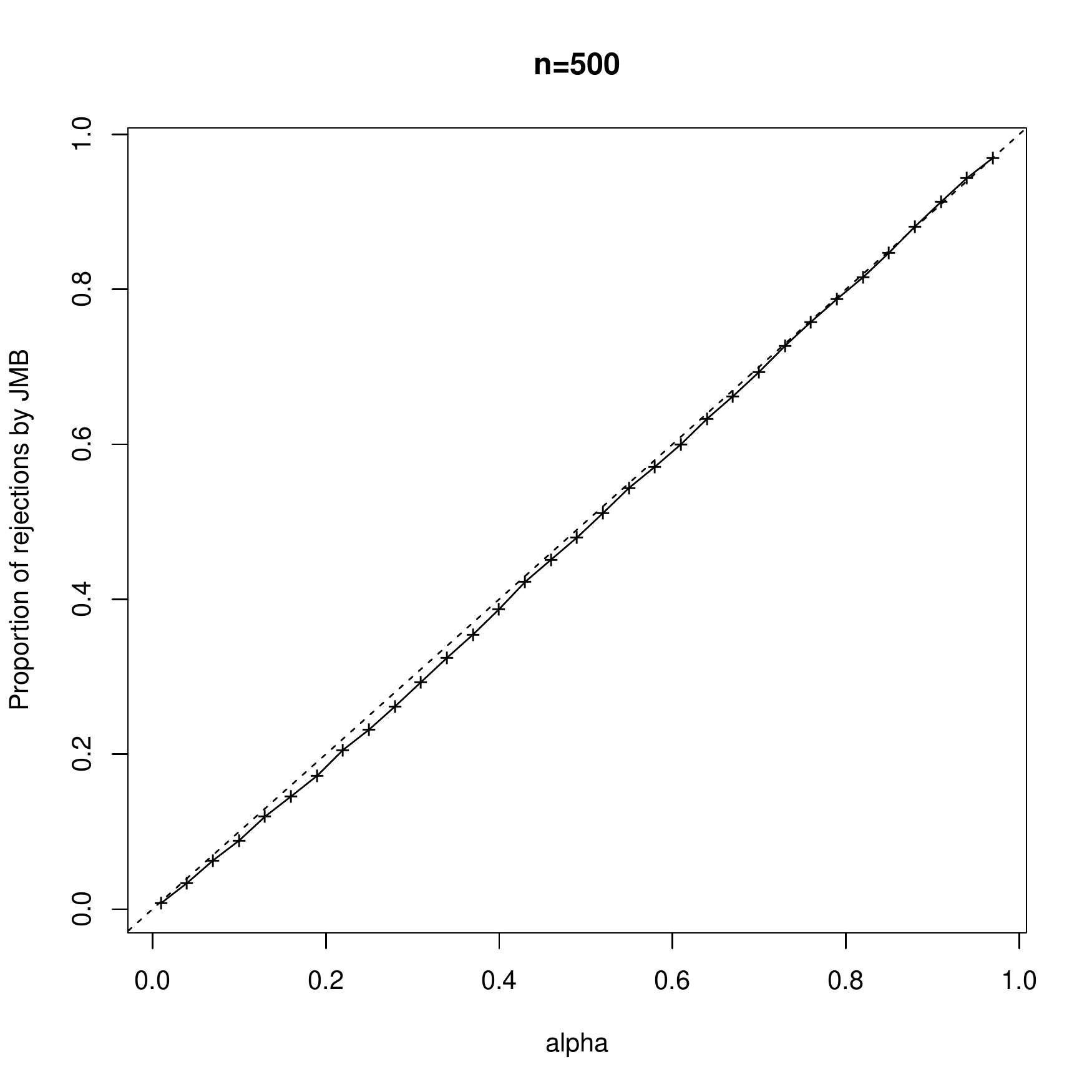}\\
	  \includegraphics[scale=0.28]{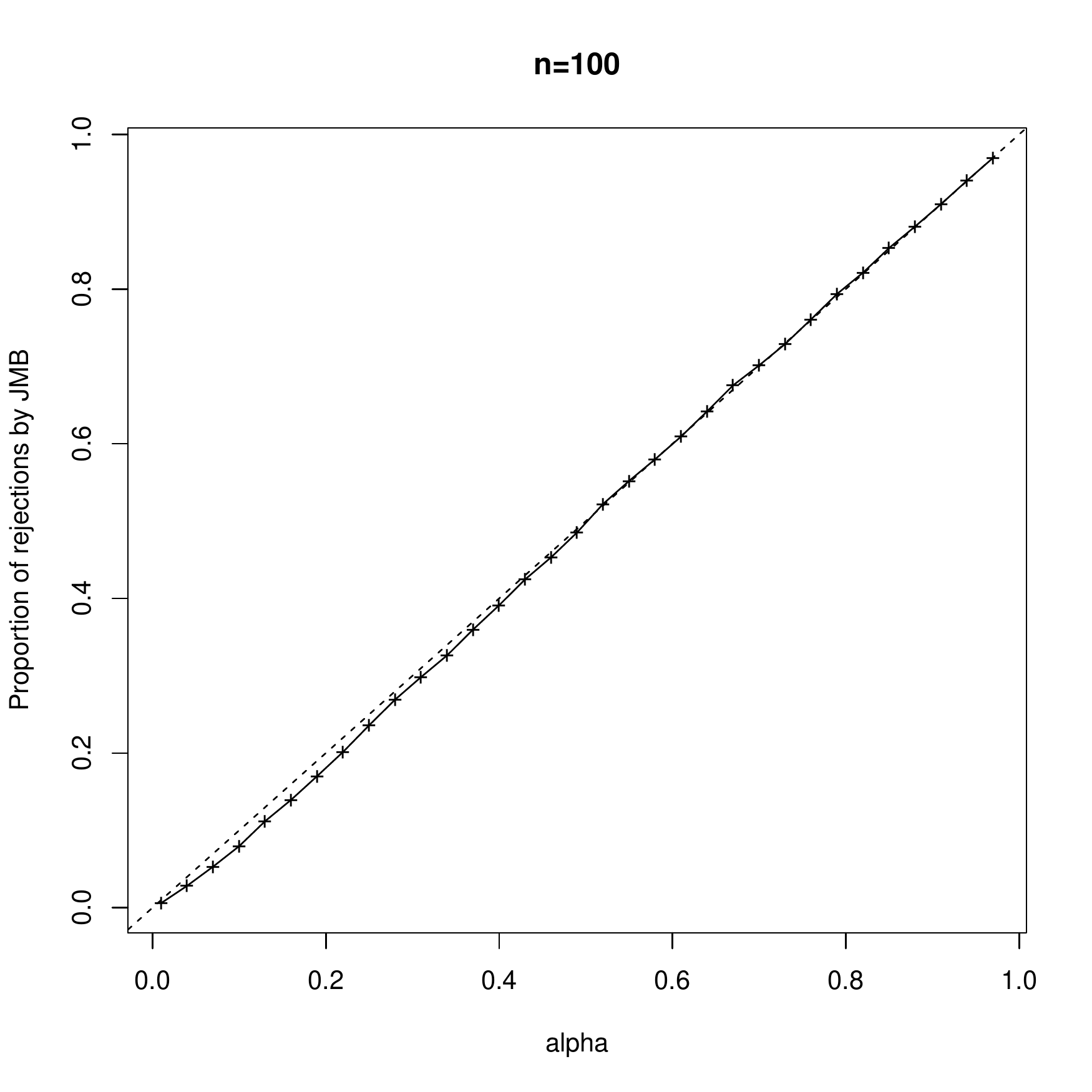}
       \includegraphics[scale=0.28]{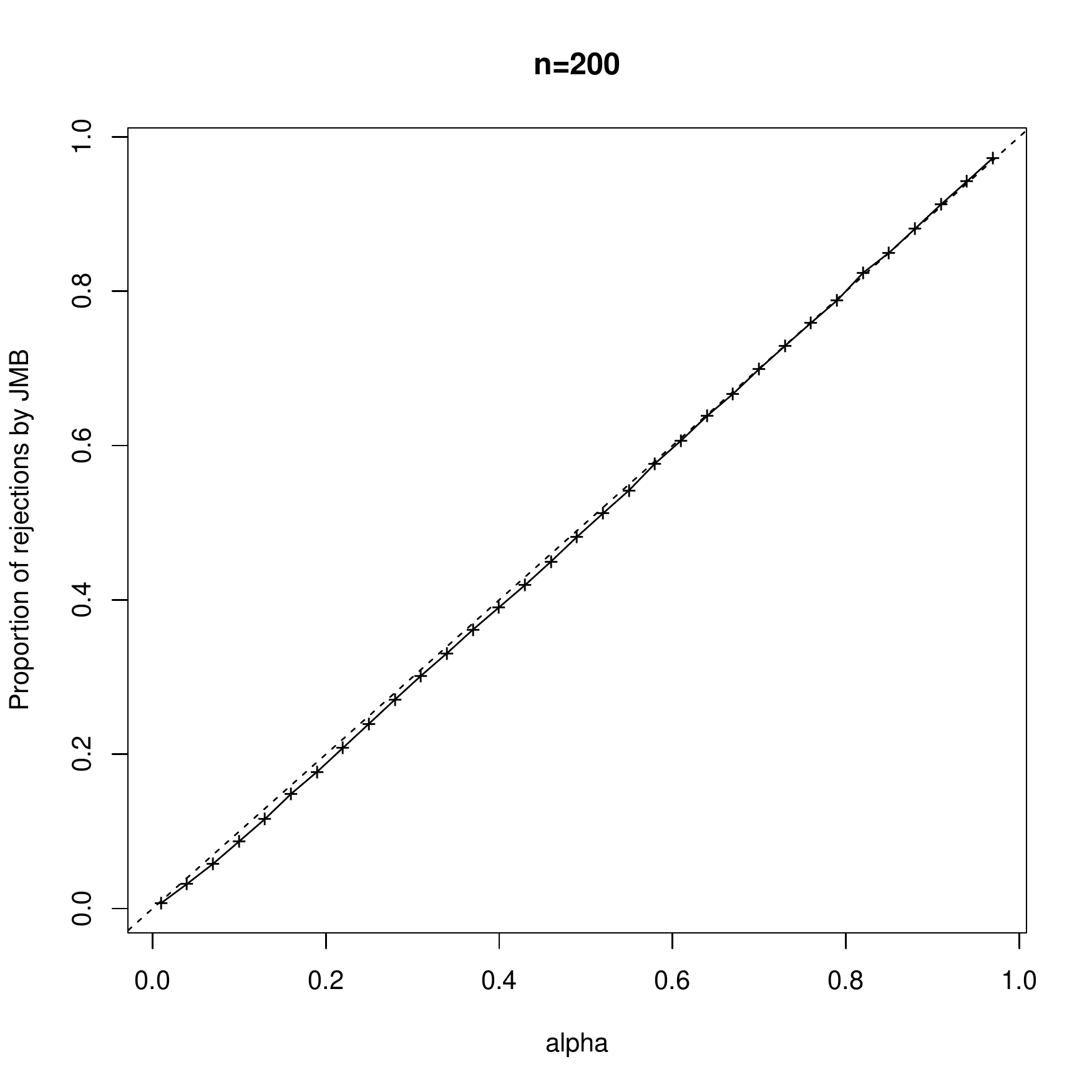}
	  \includegraphics[scale=0.28]{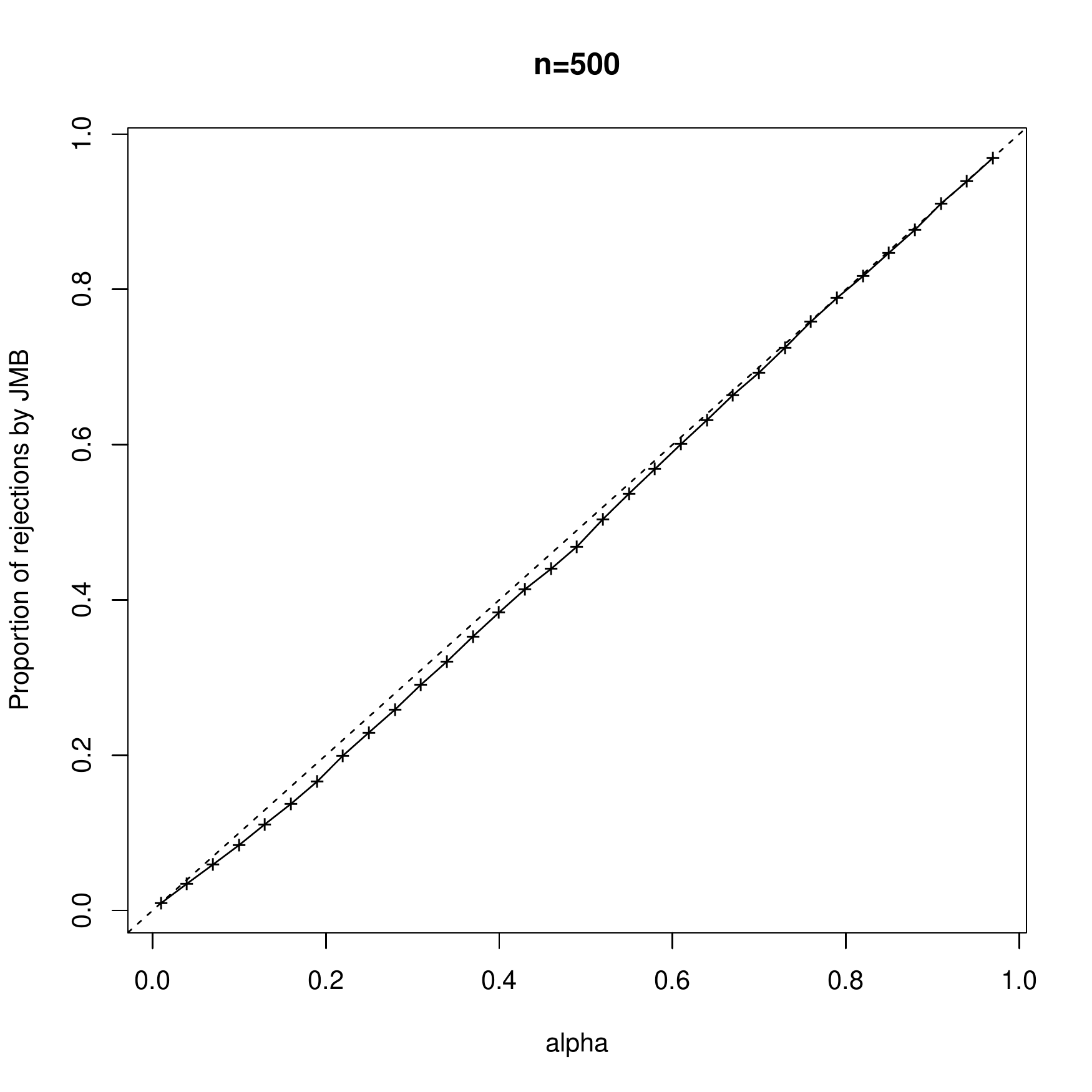}\\
    \caption{JMB approximation of sizes of the regression monotonicity test. Top row: Gaussian errors. Bottom row: Rademacher errors.}
   \label{fig:test_reg_mono_JMB}
\end{figure}

\section{Local maximal inequalities for $U$-processes}
\label{sec:local_maximal_inequalities}

In this section, we prove \textit{local maximal inequalities} for $U$-processes, which are of independent interest and can be useful for other applications. These multi-resolution local maximal inequalities are key technical tools in proving the results stated in the previous sections.

We first review some basic terminologies and facts about $U$-processes. 
For a textbook treatment on $U$-processes, we refer to \cite{delaPenaGine1999}. 
Let $r \ge  1$ be a fixed integer and let $X_{1},\dots,X_{n}$ be i.i.d. random variables taking values in a measurable space $(S,\mathcal{S})$ with  common distribution $P$. 

\begin{defn}[Kernel degeneracy; Definition 3.5.1 in \cite{delaPenaGine1999}]
\label{def:degenerate_kernel}
A  symmetric measurable function $f: S^{r} \to \R$ with $P^{r}f=0$  is said to be {\it degenerate of order $k$} with respect to $P$ if $P^{r-k}f(x_{1},\dots,x_{k}) = 0$ for all $x_{1},\dots,x_{k} \in S$. In particular, $f$ is said to be {\it completely degenerate} if $f$ is degenerate of order $r-1$, and $f$ is said to be {\it non-degenerate} if $f$ is not degenerate of any positive order.
\end{defn}

Let $\calF$ be a class of symmetric measurable functions $f: S^{r} \to \R$. We assume that there is a symmetric measurable envelope $F$ for $\calF$ such that $P^{r}F^{2} < \infty$. Furthermore, we assume that each $P^{r-k}F$ is everywhere finite. Consider the associated $U$-process
\begin{equation}
\label{eqn:uprocess}
U_{n}^{(r)} (f) = {1 \over |I_{n,r}|} \sum_{(i_{1},\dots,i_{r}) \in I_{n,r}} f(X_{i_{1}},\dots,X_{i_{r}}), \ f \in \calF.
\end{equation}
For each $k=1,\dots,r$, the \textit{Hoeffding projection} (with respect to $P$) is defined by
\begin{equation}
\label{eqn:hoeffding_projections}
(\pi_{k} f) (x_{1},\dots,x_{k}) := (\delta_{x_{1}}-P) \cdots (\delta_{x_{k}}-P) P^{r-k}f.
\end{equation}
The Hoeffding projection $\pi_k f$ is a completely degenerate kernel of $k$ variables. Then, the \textit{Hoeffding decomposition} of $U_{n}^{(r)}(f)$ is given by 
\begin{equation}
\label{eqn:hoeffding_decomp}
U_{n}^{(r)} (f) - P^{r} f = \sum_{k=1}^{r} \binom{r}{k} U_{n}^{(k)} (\pi_{k}f).
\end{equation}

In what follows, let $\sigma_{k}$ be any positive constant such that $\sup_{f \in \calF} \| P^{r-k}f \|_{P^{k},2} \le \sigma_{k} \le \| P^{r-k}F \|_{P^{k},2}$ whenever $\| PF^{r-k} \|_{P^{k},2} > 0$ (take $\sigma_{k} = 0$ when $\| P^{r-k} F \|_{P^{k},2} = 0$), and let 
\[
M_{k} = \max_{1 \le  i \le \lfloor n/k \rfloor} (P^{r-k} F)(X_{(i-1)k+1}^{ik}),
\]
where $X_{(i-1)k+1}^{ik} = (X_{(i-1)k+1},\dots,X_{ik})$.

We will assume certain uniform covering number conditions for the function class $\calF$.
For $k=1,\dots,r$, define the uniform entropy integral 
\begin{equation}
\label{eqn:uniform_entropy_integral}
J_k(\delta) := J_k(\delta, \calF, F) := \int_0^\delta  \sup_Q \left [1 + \log N(P^{r-k}\calF, \| \cdot \|_{Q,2}, \tau \| P^{r-k}F \|_{Q,2}) \right ]^{k/2} d \tau,
\end{equation}
where $P^{r-k}\calF = \{ P^{r-k}f : f \in \calF \}$ and $\sup_{Q}$ is taken over all finitely discrete distributions on $S^k$. We note that $P^{r-k}F$ is an envelope for $P^{r-k}\calF$. To avoid measurablity difficulties, we will assume that $\calF$ is pointwise measurable. 
If $\calF$ is pointwise measurable and $P^{r} F < \infty$ (which we have assumed) then $\pi_{k}\calF := \{ \pi_{k} f : f \in \calF \}$ and $P^{r-k}\calF$ for $k=1,\dots,r$ are all pointwise measurable by the dominated convergence theorem. 

Let $\varepsilon_{1},\dots,\varepsilon_{n}$ be i.i.d. Rademacher random variables such that $\Prob(\varepsilon_{i}=\pm1)=1/2$. A real-valued Rademacher chaos variable of order $k$, $X$, is a polynomial of order $k$ in the Rademacher random variables $\varepsilon_{i}$ with real coefficients, i.e., 
\[
X = a + \sum_{i=1}^{n} a_{i} \varepsilon_{i} + \sum_{(i_{1},i_{2}) \in I_{n,2}} a_{i_{1} i_{2}} \varepsilon_{i_{1}} \varepsilon_{i_{2}} + \cdots + \sum_{(i_{1},\dots,i_{k}) \in I_{n,k}} a_{i_{1} \dots i_{k}} \varepsilon_{i_{1}} \cdots \varepsilon_{i_{k}}, 
\]
where $a, a_{i}, a_{i_{1} i_{2}}, \dots, a_{i_{1} \dots i_{k}} \in \R$. If only the monomials of degree $k$ in the variables $\varepsilon_{i}$ in $X$ are not zero, then $X$ is a homogeneous Rademacher chaos of order $k$; see Section 3.2 in \cite{delaPenaGine1999}.  

\begin{defn}[Rademacher chaos process of order $k$; page 220 in \cite{delaPenaGine1999}]
A stochastic process $X(t), t \in T$ is said to be a {\it Rademacher chaos process of order $k$} if for all $s, t \in T$, the joint law of $(X(s), X(t))$ coincides with the joint law of two (not necessarily homogeneous) Rademacher chaos variables of order $k$. 
\end{defn}

In the remainder of this section, the notation $\lesssim$ signifies that the left hand side is bounded by the right hand side up to a constant that depends only on $r$. Recall that $\| \cdot \|_{\calF} = \sup_{f \in \calF} | \cdot |$.

\begin{thm}[Local maximal inequalities for $U$-processes]
\label{lem:local_max_ineq_dengenerate_uprocess}
Suppose that $\calF$ is poinwise measurable
and that  $J_{k}(1) < \infty$ for $k=1,\dots,r$. Let $\delta_{k} =\sigma_{k}/ \| P^{r-k} F \|_{P^{k},2}$ for $k=1,\dots,r$.  Then
\begin{equation}
\label{eq:local_max_ineq_dengenerate_uprocess}
n^{k/2} \E[\| U_{n}^{(k)} (\pi_{k}f) \|_{\calF}] \lesssim J_{k}(\delta_{k}) \| P^{r-k} F \|_{P^{k},2} + \frac{J_{k}^{2}(\delta_{k}) \|M_{k}\|_{\Prob,2}}{\delta_{k}^{2}\sqrt{n}}
\end{equation}
for every $k=1,\dots,r$. 
If $\| P^{r-k} F \|_{P^{k},2} = 0$, then the right hand side is interpreted as $0$. 
\end{thm}

The proof of Theorem \ref{lem:local_max_ineq_dengenerate_uprocess} relies on the following lemma on the uniform entropy integrals.

\begin{lem}[Properties of the maps $\delta \mapsto J_k(\delta)$]
\label{lem:properties_Jprime}
Assume that $J_{k} (1) < \infty$ for $k=1,\dots,r$. 
 Then, the  following properties hold for every $k=1,\dots,r$.
(i) The map $\delta \mapsto J_k(\delta)$ is non-decreasing and concave.
(ii) For $c \ge  1$, $J_k(c \delta) \le  c J_k(\delta)$.
(iii) The map $\delta \mapsto J_k(\delta) / \delta$ is non-increasing.
(iv) The map $(x,y) \mapsto J_k(\sqrt{x/y}) \sqrt{y}$ is jointly concave in $(x,y) \in [0,\infty) \times (0,\infty)$.
\end{lem}

\begin{proof}[Proof of Lemma \ref{lem:properties_Jprime}]
The proof is almost identical to \cite[Lemma A.2]{cck2014_empirical_process} and hence omitted. 
\end{proof}

\begin{proof}[Proof of Theorem \ref{lem:local_max_ineq_dengenerate_uprocess}]
Pick any $k=1,\dots,r$.
It suffices to prove (\ref{eq:local_max_ineq_dengenerate_uprocess}) when $\| P^{r-k} F \|_{P^{k},2}  > 0$ since otherwise there is nothing to prove (recall that we have assumed that $P^{r} F^{2} < \infty$, which ensures that $\| P^{r-k} F\|_{P^{k},2} < \infty$). 
Let $\varepsilon_{1},\dots,\varepsilon_n$ be  i.i.d. Rademacher random variables independent of $X_1^n$. In addition, let $\{ X_{i}^{j} \}$ and $\{ \varepsilon_{i}^{j} \}$ be independent copies of $\{ X_{i} \}$ and $\{ \varepsilon_{i} \}$. From the randomization theorem for $U$-processes \citep[][Theorem 3.5.3]{delaPenaGine1999} and Jensen's inequality, we have 
\begin{align*}
\E[ \|  U_{n}^{(k)} (\pi_{k}f) \|_{\calF} ] 
&\lesssim \E \left [ \left \| \frac{1}{|I_{n,k}|} \sum_{(i_1,\dots,i_k) \in I_{n,k}} \varepsilon_{i_{1}}^{1}\cdots \varepsilon_{i_{k}}^{k} (\pi_{k}f)(X_{i_{1}}^{1},\dots,X_{i_{k}}^{k}) \right \|_{\calF} \right ] \\
&\lesssim  \E \left [ \left \| \frac{1}{|I_{n,k}|} \sum_{(i_1,\dots,i_k) \in I_{n,k}} \varepsilon_{i_{1}}^{1}\cdots \varepsilon_{i_{k}}^{k} (P^{r-k} f) (X_{i_{1}}^{1},\dots,X_{i_{k}}^{k}) \right \|_{\calF} \right ] \\
&\lesssim  \E \left [ \left \| \frac{1}{|I_{n,k}|} \sum_{(i_1,\dots,i_k) \in I_{n,k}} \varepsilon_{i_{1}}\cdots \varepsilon_{i_{k}} (P^{r-k}f)(X_{i_{1}},\dots,X_{i_{k}}) \right \|_{\calF} \right ].
\end{align*}
Conditionally on $X_{1}^{n}$, 
\[
R_{n,k}(f) := \frac{1}{\sqrt{|I_{n,k}|}} \sum_{(i_1,\dots,i_k) \in I_{n,k}} \varepsilon_{i_{1}}\cdots \varepsilon_{i_{k}} (P^{r-k}f)(X_{i_{1}},\dots,X_{i_{k}}), \ f \in \calF
\]
is a (homogeneous) Rademacher chaos process of order $k$. Denote by  $\Prob_{I_{n,k}} = |I_{n,k}|^{-1} \sum_{(i_1,\dots,i_k) \in I_{n,k}} \delta_{(X_{i_{1}},\dots,X_{i_{k}})}$  the empirical distribution  on all possible $k$-tuples of $X_1^n$; then Corollary 3.2.6 in \cite{delaPenaGine1999} yields
\[
\| R_{n,k}(f) - R_{n,k}(f') \|_{\psi_{2/k} \mid X_{1}^{n}} \lesssim \| P^{r-k}f - P^{r-k}f' \|_{\Prob_{I_{n,k}},2}, \ \forall f,f' \in \calF,
\]
where $\| \cdot \|_{\psi_{2/k} \mid X_{1}^{n}}$ denotes the Orlicz (quasi-)norm associated with $\psi_{2/k}(u) = e^{u^{2/k}} - 1$ evaluated conditionally on $X_{1}^{n}$. The $\| \cdot \|_{\psi_{2/k} \mid X_{1}^{n}}$-diameter of the function class $ \calF$ is at most $2\sigma_{I_{n,k}}$ with $\sigma_{I_{n,k}}^{2} := \sup_{f \in \calF} \| P^{r-k} f \|_{\mathbb{P}_{I_{n,k}},2}^{2}$. 
So, since the first moment is bounded by the $\psi_{2/k}$-(quasi)norm up to a constant that depends only on $k$ (and hence $r$), by Corollary 5.1.8 in \cite{delaPenaGine1999} together with Fubini's theorem and a change of variables, we have 
\begin{align*}
&\E \left [ \left \| \frac{1}{\sqrt{|I_{n,k}|}} \sum_{(i_1,\dots,i_k) \in I_{n,k}} \varepsilon_{i_{1}}\cdots \varepsilon_{i_{k}} (P^{r-k}f)(X_{i_{1}},\dots,X_{i_{k}}) \right \|_{\calF} \right ] \\
&\lesssim \E \left [ \left \| \left \| \frac{1}{\sqrt{|I_{n,k}|}} \sum_{(i_1,\dots,i_k) \in I_{n,k}} \varepsilon_{i_{1}}\cdots \varepsilon_{i_{k}} (P^{r-k}f)(X_{i_{1}},\dots,X_{i_{k}}) \right \|_{\calF} \right \|_{\psi_{2/k} \mid X_{1}^{n}} \right ] \\
&\lesssim \E\left [ \int_{0}^{\sigma_{I_{n,k}}} \left [ 1 + \log N(P^{r-k}\calF, \| \cdot \|_{\mathbb{P}_{I_{n,k}},2}, \tau) \right ]^{k/2} d\tau \right ]  \\
&= \E \left [ \| P^{r-k}F  \|_{\mathbb{P}_{I_{n,k}},2} \int_{0}^{\sigma_{I_{n,k}}/\| P^{r-k}F \|_{\mathbb{P}_{I_{n,k}},2}} \left [  1+\log N(P^{r-k} \calF, \| \cdot \|_{\mathbb{P}_{I_{n,k}},2}, \tau \| P^{r-k}F  \|_{\mathbb{P}_{I_{n,k}},2}) \right ]^{k/2}d\tau \right ] \\
&\le \E \left [ \| P^{r-k}F  \|_{\mathbb{P}_{I_{n,k}},2} J_{k}(\sigma_{I_{n,k}}/\| P^{r-k}F \|_{\mathbb{P}_{I_{n,k}},2}) \right].
\end{align*}
The last inequality follows from the definition of $J_{k}$. 
Since $J_k(\sqrt{x/y}) \sqrt{y}$ is jointly concave in $(x,y) \in [0,\infty) \times (0,\infty)$ by Lemma \ref{lem:properties_Jprime} (iv), Jensen's inequality yields 
\begin{equation}
\label{eqn:intermediate_bound_on_Vn}
n^{k/2} \E[ \| U_{n}^{(k)} (\pi_{k}f) \|_{\calF} ] \lesssim \| P^{r-k} F \|_{P^{k},2} J_{k}(z), \quad \text{where} \ z:=\sqrt{\E[\sigma_{I_{n,k}}^{2}] / \| P^{r-k}F \|_{P^{k},2}^2}.
\end{equation}
We shall bound $\E[\sigma_{I_{n,k}}^{2}]$.  
To this end, we will use Hoeffding's averaging \cite[Section 5.1.6]{serfling1980}. Let 
\[
S_{f,k}(x_{1},\dots,x_{n}) = \frac{1}{m} \sum_{i=1}^{m} (P^{r-k}f)^{2}(x_{(i-1)k+1},\dots,x_{ik}), \ m=\lfloor n/k \rfloor.
\]
Then, the $U$-statistic $\| P^{r-k} f \|_{\mathbb{P}_{I_{n,k}},2}^{2} = |I_{n,k}|^{-1} \sum_{I_{n,k}} (P^{r-k}f)^{2}(X_{i_{1}},\dots,X_{i_{k}})$ is the average of the variables $S_{f,k}(X_{j_{1}},\dots,X_{j_{n}})$ taken over all the permutations $j_{1},\dots,j_{n}$ of $1,\dots,n$. Hence, 
\[
\E[ \sigma_{I_{n,k}}^{2}] \le \E \left [ \sup_{f \in \calF} S_{f,k}(X_{1}^{n}) \right ] = \E \left [ \left \| \frac{1}{m}\sum_{i=1}^{m} (P^{r-k}f)^{2}(X_{(i-1)k+1}^{ik}) \right \|_{\calF} \right ] =: B_{n,k}
\]
by Jensen's inequality, 
so that $z \le  \tilde{z} := \sqrt{B_{n,k} / \| P^{r-k}F \|_{P^{k},2}^2}$. Since the blocks $X_{(i-1)k+1}^{ik}, i=1,\dots,m$ are i.i.d., 
\begin{eqnarray*}
B_{n,k} &\le _{(1)}& \sigma_k^2 + \E \left [\left \| {1 \over m} \sum_{i=1}^m \left \{ (P^{r-k} f)^2(X_{(i-1)k+1}^{ik}) - \E [(P^{r-k} f)^2(X_{(i-1)k+1}^{ik}) ] \right \} \right \|_{\calF}\right ] \\
&\le _{(2)}& \sigma_k^2 + 2 \E\left [ \left\| {1 \over m} \sum_{i=1}^m \varepsilon_i (P^{r-k} f)^2(X_{(i-1)k+1}^{ik}) \right \|_{\calF} \right ] \\
&\le _{(3)}& \sigma_k^2 + 8 \E\left [ M_k \left \| {1 \over m} \sum_{i=1}^m \varepsilon_i (P^{r-k} f)(X_{(i-1)k+1}^{ik}) \right \|_{\calF} \right ] \\
&\le _{(4)}& \sigma_k^2 + 8 \|M_k\|_{\Prob,2}  \sqrt{\E \left [\left \| {1 \over m} \sum_{i=1}^m \varepsilon_i (P^{r-k} f)(X_{(i-1)k+1}^{ik}) \right\|_{\calF}^2 \right]},
\end{eqnarray*}
where $(1)$ follows from the triangle inequality, $(2)$ follows from the symmetrization inequality \cite[Lemma 2.3.1]{vandervaartwellner1996}, $(3)$ follows from the contraction principle \cite[Corollary 3.2.2]{ginenickl2016}, and $(4)$ follows from the Cauchy-Schwarz inequality. By (a version of) the Hoffmann-J{\o}rgensen inequality to the empirical process \cite[Proposition A.1.6]{vandervaartwellner1996},
\begin{align*}
&\sqrt{\E \left [\left \| {1 \over m} \sum_{i=1}^m \varepsilon_i (P^{r-k} f)(X_{(i-1)k+1}^{ik}) \right\|_{\calF}^2 \right]} \\\
&\quad \lesssim \E \left [ \left \| {1 \over m} \sum_{i=1}^m \varepsilon_i (P^{r-k} f)(X_{(i-1)k+1}^{ik}) \right \|_{\calF} \right] + m^{-1} \|M_k\|_{\Prob,2}.
\end{align*}
The analysis of the expectation on the right hand side is rather standard. 
From the first half of the proof of Theorem 5.2 in \cite{cck2014_empirical_process} (or repeating the first half of this proof with $r=k=1$), we have
\begin{multline*}
\E \left [ \left \|\frac{1}{\sqrt{m}} \sum_{i=1}^m   \varepsilon_i (P^{r-k} f)(X_{(i-1)k+1}^{ik}) \right \|_{\calF} \right]  \\
\lesssim \| P^{r-k} F \|_{P^{k},2} \int_{0}^{\tilde{z}} \sup_{Q} \sqrt{1+\log N(P^{r-k} \calF, \| \cdot \|_{Q,2}, \tau \| P^{r-k} F \|_{Q,2})} d\tau.
\end{multline*}
%
Since the integral on the right hand side is bounded by $J_{k}(\tilde{z})$, we have
\[
B_{n,k} \lesssim \sigma_k^2 + n^{-1} \|M_k\|_{\Prob,2}^2 + n^{-1/2} \|M_k\|_{\Prob,2} \| P^{r-k} F \|_{P^{k},2} J_k ( \tilde{z} ).
\]
Therefore, we conclude that
\[
\tilde{z}^2 \lesssim \Delta^2 + {\|M_k\|_{\Prob,2} \over \sqrt{n} \| P^{r-k} F \|_{P^{k},2}} J_k (\tilde{z}), \quad \text{where} \ \Delta^2 := {\sigma_k^2 \vee n^{-1} \|M_k\|_{\Prob,2}^2 \over \| P^{r-k} F \|_{P^{k},2}^{2}}.
\]
By Lemma \ref{lem:properties_Jprime} (i) and applying \cite[Lemma 2.1]{vandervaartwellner2011} with  $J(\cdot)= J_k(\cdot), r=1, A^2 = \Delta^2$, and $B^2 = \|M_k\|_{\Prob,2} / (\sqrt{n} \| P^{r-k} F \|_{P^{k},2})$, we have
\begin{equation}
\label{eqn:intermediate_bound_on_B}
J_k(z) \le  J_k(\tilde{z}) \lesssim J_k(\Delta) \left[ 1 + J_k(\Delta) {\|M_k\|_{\Prob,2} \over \sqrt{n} \| P^{r-k} F \|_{P^{k},2}\Delta^{2} } \right].
\end{equation}
Combining (\ref{eqn:intermediate_bound_on_Vn}) and (\ref{eqn:intermediate_bound_on_B}), we arrive at
\begin{equation}
\label{eqn:intermediate_bound_on_C}
n^{k/2} \E[\| U_{n}^{(k)} (\pi_{k}f) \|_{\calF}] \lesssim J_k(\Delta) \| P^{r-k} F \|_{P^{k},2}  + {J_k^{2}(\Delta) \|M_k\|_{\Prob,2} \over \sqrt{n} \Delta^2}.
\end{equation}
We note that $\Delta \ge  \delta_k$ and recall that $\delta_{k}  =\sigma_{k}/\| P^{r-k} F \|_{P^{k},2}$. Since the map $\delta \mapsto J_k(\delta)/\delta$ is non-increasing by Lemma \ref{lem:properties_Jprime} (iii), we have
\[
J_k(\Delta) \le  \Delta {J_k(\delta_k) \over \delta_k} = \max \left\{ J_k(\delta_k), {\|M_k\|_{\Prob,2} J_r(\delta_k) \over \sqrt{n} \| P^{r-k} F \|_{P^{k},2} \delta_k} \right\}.
\]
In addition, since  $J_k(\delta_k) / \delta_k \ge  J_k(1) \ge  1$, we have
$$
J_k(\Delta) \le  \max \left\{ J_k(\delta_k), {\|M_k\|_{\Prob,2} J_k^{2}(\delta_k) \over \sqrt{n} \| P^{r-k} F \|_{P^{k},2} \delta_k^2} \right\}.
$$
Finally, since
$$
{J_k^{2}(\Delta) \|M_k\|_{\Prob,2} \over \sqrt{n} \Delta^2} \le  {J_k^{2}(\delta_k) \|M_k\|_{\Prob,2} \over \sqrt{n} \delta_k^2},
$$
the desired inequality (\ref{eq:local_max_ineq_dengenerate_uprocess}) follows from (\ref{eqn:intermediate_bound_on_C}).
\end{proof}

When the function class $\calF$ is VC type, we may derive a more explicit bound on $n^{k/2}\E[ \| U_{n}^{(k)} (\pi_{k} f) \|_{\calF}]$.  

\begin{cor}[Local maximal inequalities for $U$-processes indexed by VC type classes]
\label{lem:local_max_ineq_dengenerate_uprocess_VCtype}
If $\calF$ is pointwise measurable and VC type with characteristics $A \ge (e^{2(r-1)}/16) \vee e$ and $v \ge 1$, then
\begin{align}
\label{eq:local_max_ineq_dengenerate_uprocess_VCtype}
&n^{k/2} \E[\| U_{n}^{(k)} (\pi_{k}f)  \|_{\calF}] \notag \\
&\quad \lesssim \sigma_{k} \left \{ v \log (A\| P^{r-k} F \|_{P^{k},2}/\sigma_{k}) \right \}^{k/2} + \frac{\|M_{k}\|_{\Prob,2}}{\sqrt{n}} \left \{ v \log (A\| P^{r-k}F \|_{P^{k},2}/\sigma_{k}) \right \}^{k}
\end{align}
for every $k=1,\dots,r$. 
\end{cor}

\begin{rmk}
\label{rmk:comparison_local_max_ineq}
(i). Our maximal inequality (\ref{eq:local_max_ineq_dengenerate_uprocess}) scales correctly with the order of degeneracy, namely, the bound on $\E[\| U_{n}^{(k)} (\pi_k f) \|_{\calF}]$ scales as $n^{-k/2}$ if $\calF$ is fixed with $n$; recall that the functions $\pi_{k} f, f \in \calF$ are completely degenerate functions of $k$ variables. In addition, our maximal inequality is ``local'' in the sense that the bound is able take into account the $L^{2}$-bound on functions $P^{r-k}f, f \in \calF$, namely, the bound will yield a better estimate if we have an additional information that such an $L^{2}$-bound is small.

(ii).
\cite[Theorem 8]{ginemason2007} establishes a different local maximal inequality for a $U$-process indexed by a VC type class with a bounded envelope. To be precise, they prove the following bound under the assumption that  the envelope $F$ is bounded by a constant $M$: there exist constants $C_1$ and $C_2$ depending only on $r,A,v$, and $M$ such that
\begin{equation}
n^{k/2} \E[\| U_{n}^{(k)}(\pi_{k}f) \|_{\calF} ] \le C_1 \sigma_r \left ( \log \frac{A \| F \|_{P^{r},2}}{\sigma_r} \right )^{k/2}, \ k=1,\dots,r
\label{eq: Gine Mason bound}
\end{equation}
whenever 
\[
n\sigma_r^{2} \ge C_2 \log \left ( \frac{2\| F \|_{P^{r},2}}{\sigma_r} \right ),
\]
where $\sigma_r$ is a positive  constant satisfying $\sup_{f \in \calF} \|  f \|_{P^{r},2} \le \sigma_r \le \| F \|_{P^{r},2}$. 
Our Corollary \ref{lem:local_max_ineq_dengenerate_uprocess_VCtype} improves upon the bound (\ref{eq: Gine Mason bound}) in several directions: 1) First, our bound (\ref{eq:local_max_ineq_dengenerate_uprocess_VCtype}) allows for an unbounded envelope while the bound (\ref{eq: Gine Mason bound}) requires the envelope to be bounded. 2) Second, the constants $C_1$ and $C_2$ appearing in the bound (\ref{eq: Gine Mason bound}) implicitly depend on the VC characteristics $(A,v)$ and the $L^{\infty}$-bound $M$ on the envelope $F$, in addition to the order $r$, and so is not applicable to cases where the VC characteristics $(A,v)$ and/or the $L^{\infty}$-bound $M$ change with $n$. On the other hand, the constant involved in our bound (\ref{eq:local_max_ineq_dengenerate_uprocess_VCtype}) depends only on $r$ (recall that the notation $\lesssim$ in present section signifies that the left hand side is bounded by the right hand side up to a constant that depends only on $r$), and so is applicable to such cases. 3) Finally, our bound (\ref{eq:local_max_ineq_dengenerate_uprocess_VCtype}) is of the multi-resolution nature in the sense that it depends on the $L^{2}$-bound on $P^{r-k}f$ for $f \in \calF$ (i.e., $\sigma_k$) for each projection level $k=1,\dots,r$ rather than that on $f \in \calF$  (i.e., $\sigma_r$), which allows us to obtain better rates of convergence for kernel type statistics than (\ref{eq: Gine Mason bound}). In particular, $\sigma_{k}$ for $k < r$ can be potentially much smaller than $\sigma_{r}$, which is indeed the case in the applications considered in Section \ref{sec:monotonicity_testing}. To be precise, for the function class $\{ b_{n}^{m/2} c_{n}(\vartheta)^{-1} h_{n,\vartheta} : \vartheta \in \Theta \}$ appearing in Section \ref{sec:monotonicity_testing}, $\sigma_{k}$ would be of order $b_n^{-m(k-1)/2}$ and so $\sigma_k \ll \sigma_r$ for $k < r$; see the proof of Theorem \ref{thm:app_local_uproc_bootstrap_kolmogorov_distance}. 

We also note that \cite{adamczak2006_AoP,ginelatalazinn2000} derive sophisticated moment inequalities for $U$-statistics in Banach spaces. However, we find that their inequalities are difficult to apply in  our setting. 
%


(iii). Theorem \ref{lem:local_max_ineq_dengenerate_uprocess} and Corollary \ref{lem:local_max_ineq_dengenerate_uprocess_VCtype} generalize Theorem  5.2 and Corollary 5.1 in \cite{cck2014_empirical_process} to $U$-processes. In fact, Theorem \ref{lem:local_max_ineq_dengenerate_uprocess} and Corollary \ref{lem:local_max_ineq_dengenerate_uprocess_VCtype} reduce to Theorem  5.2 and Corollary 5.1 in \cite{cck2014_empirical_process} when $r=k=1$, respectively. 
\end{rmk}

Before proving Corollary \ref{lem:local_max_ineq_dengenerate_uprocess_VCtype}, we first verify the following fact about VC type properties. 

\begin{lem}
\label{lem:uniform_entropy_numbers_between_GH}
If $\calF$ is VC type with characteristics $(A,v)$, then for every $k=1,\dots,r-1$, $P^{r-k}\calF$ is also VC type with characteristics $4\sqrt{A}$ and $2v$ for envelope $P^{r-k}F$, i.e.,
\[
\sup_{Q} N(P^{r-k}\calF, \| \cdot \|_{Q,2}, \tau \| P^{r-k} F \|_{Q,2}) \le  (4\sqrt{A}/\tau)^{2v}, \ 0 < \forall \tau \le 1.
\]
\end{lem}

\begin{proof}[Proof of Lemma \ref{lem:uniform_entropy_numbers_between_GH}]
This follows from Lemma \ref{lem:entropy_conditional} in Appendix \ref{app:supporting_lemmas} with $r=s=2$. 
\end{proof}

\begin{proof}[Proof of Corollary \ref{lem:local_max_ineq_dengenerate_uprocess_VCtype}]
For the notational convenience, put $A'=4\sqrt{A}$ and $v'=2v$. Then,
\[
J_{k} (\delta) \le \int_{0}^{\delta} (1+v'\log (A'/\tau))^{k/2} d\tau \le  A' (v')^{k/2} \int_{A'/\delta}^{\infty} \frac{(1+\log \tau)^{k/2}}{\tau^{2}} d\tau.
\]
Integration by parts yields that for $c \ge  e^{k-1}$, 
\begin{align*}
\int_{c}^{\infty} \frac{(1+\log \tau)^{k/2}}{\tau^{2}} d\tau &= \left [ -\frac{(1+\log \tau)^{k/2}}{\tau} \right]_{c}^{\infty} + \frac{k}{2} \int_{c}^{\infty} \frac{(1+\log \tau)^{k/2}}{\tau^{2} (1+\log \tau)} d\tau \\
&\le \frac{(1+\log c)^{k/2}}{c} + \frac{1}{2} \int_{c}^{\infty} \frac{(1+\log \tau)^{k/2}}{\tau^{2}} d\tau.
\end{align*}
Since $A'/\delta \ge  A' \ge  e^{r-1} \ge  e^{k-1}$ for $0 < \delta \le  1$, we conclude that 
\[
\int_{A/\delta'}^{\infty} \frac{(1+\log \tau)^{k/2}}{\tau^{2}} d\tau \le \frac{2\delta (1+\log (A'/\delta))^{k/2}}{A'} \lesssim \frac{\delta (\log (A/\delta))^{k/2}}{A'}.
\]
Combining Theorem \ref{lem:local_max_ineq_dengenerate_uprocess}, we obtain the desired inequality (\ref{eq:local_max_ineq_dengenerate_uprocess_VCtype}).
\end{proof}

The appearance of $\| P^{r-k} F \|_{P^{k},2}/\sigma_{k}$ inside the log may be annoying in applications but there is a clever way to delete this term. Namely, choose $\sigma_{k}' = \sigma_{k} \vee (n^{-1/2} \| P^{r-k} F \|_{P^{k},2})$ and apply Corollary \ref{lem:uniform_entropy_numbers_between_GH} with $\sigma_{k}$ replaced by $\sigma_{k}'$; then the  bound for $n^{k/2} \E[ \| U_{n}^{(k)}(\pi_{k}f) \|_{\calF} ]$ is 
\[
\lesssim \sigma_{k} \left \{ v \log (A \vee n) \right \}^{k/2} + \frac{\| P^{r-k} F \|_{P^{k},2}}{\sqrt{n}} \left \{ v \log (A \vee n) \right \}^{k/2} + \frac{\| M_{k} \|_{\Prob,2}}{\sqrt{n}}  \left \{ v \log (A \vee n) \right \}^{k}.
\]
Since $v \log (A \vee n) \ge 1$ by our assumption, the second term is bounded by the third term. We state the resulting bound as a separate corollary since this form would be most useful in (at least our) applications.

\begin{cor}
\label{cor:local_max_ineq_dengenerate_uprocess_VCtype}
If $\calF$ is pointwise measurable and VC type with characteristics $A \ge (e^{2(r-1)}/16) \vee e$ and $v \ge 1$,  then, 
\[
n^{k/2} \E[\| U_{n}^{(k)} (\pi_{k}f)  \|_{\calF}]  \lesssim \sigma_{k} \left \{ v \log (A \vee n) \right \}^{k/2} + \frac{\| M_{k} \|_{\Prob,2}}{\sqrt{n}}  \left \{ v \log (A\vee n) \right \}^{k}
\]
for every $k=1,\dots,r$. Furthermore, $\| M_{k} \|_{\Prob,2} \le  n^{1/q} \| P^{r-k}F \|_{P^{k},q}$ for every $k=1,\dots,r$ and $q \in [2,\infty]$, where ``$1/q$'' for the $q=\infty$ case is interpreted as $0$. 
\end{cor}

\begin{proof}[Proof of Corollary \ref{cor:local_max_ineq_dengenerate_uprocess_VCtype}]
The first half of the corollary is already proved. The latter half is trivial. 
\end{proof}

If one is interested in bounding $\E[ \| U_{n}^{(r)}(f) - P^{r}f \|_{\calF} ]$, then it suffices to apply  (\ref{eq:local_max_ineq_dengenerate_uprocess}) or (\ref{eq:local_max_ineq_dengenerate_uprocess_VCtype}) repeatedly for $k=1,\dots,r$. 
However, it is often the case that lower order Hoeffding projection terms are dominant, and for bounding higher order Hoeffding projection terms, it would suffice to apply the following simpler (but less sharp) maximal inequalities. 


\begin{cor}[Alternative maximal inequalities for $U$-processes]
\label{thm:alternative_max_ineq}
Let $p \in [1,\infty)$. Suppose that $\calF$ is pointwise measurable and that $J_{k}(1)< \infty$ for $k=1,\dots,r$. Then, there exists a constant $C_{r,p}$ depending only on $r,p$ such that 
\[
n^{k/2} (\E [\| U_{n}^{(k)} (\pi_{k}f) \|_{\calF}^{p}])^{1/p} \le  C_{r,p} J_{k}(1)\| P^{r-k}F \|_{P^{k},2 \vee p} 
\]
for every $k=1,\dots,r$. If $\calF$ is VC type with characteristics $A \ge (e^{2(r-1)}/16) \vee e$ and $v \ge 1$, then $J_{k}(1) \lesssim ( v \log A )^{k/2}$ for every $k=1,\dots,r$. 
\end{cor}

\begin{proof}[Proof of Corollary \ref{thm:alternative_max_ineq}]
The last assertion follows from a similar computation to that in the proof of Corollary \ref{lem:local_max_ineq_dengenerate_uprocess_VCtype}. Hence we focus here on the first assertion. 
The proof is a modification to the proof of Theorem \ref{lem:local_max_ineq_dengenerate_uprocess} and we shall use the notation used in the proof. 
The randomization theorem and Jensen's inequality yield that $n^{pk/2}\E[\| U_{n}^{(k)}(\pi_{k}f) \|_{\calF}^{p}]$ is bounded by
\[
\E \left [ \left \| \frac{1}{\sqrt{|I_{n,k}|}} \sum_{I_{n,k}} \varepsilon_{i_{1}} \cdots \varepsilon_{i_{k}} (P^{r-k} f)(X_{i_{1}},\dots,X_{i_{k}}) \right \|_{\calF}^{p} \right],
\]
up to a constant depending only on $r,p$, 
where $\varepsilon_{1},\dots,\varepsilon_{n}$ are i.i.d. Rademacher random variables independent of $X_{1}^{n}$. Denote by $\E_{\mid X_{1}^{n}}$ the conditional expectation given $X_{1}^{n}$. 
Since the $L^{p}$-norm is bounded from above by the $\psi_{2/k}$-(quasi-)norm up to a constant that depends only on $k$ (and hence $r$) and $p$, we have
\begin{align*}
&\E_{\mid X_{1}^{n}} \left [ \left \| \frac{1}{\sqrt{|I_{n,k}|}} \sum_{I_{n,k}} \varepsilon_{i_{1}} \cdots \varepsilon_{i_{k}} (P^{r-k} f)(X_{i_{1}},\dots,X_{i_{k}}) \right \|_{\calF}^{p} \right ] \\
&\le  C \left \| \left \| \frac{1}{\sqrt{|I_{n,k}|}} \sum_{I_{n,k}} \varepsilon_{i_{1}} \cdots \varepsilon_{i_{k}} (P^{r-k} f)(X_{i_{1}},\dots,X_{i_{k}}) \right \|_{\calF}\right \|_{\psi_{k/2} \mid X_{1}^{n}}^{p}
\end{align*}
for some constant $C$ depending only on $r$ and $p$. The entropy integral bound for Rademacher chaoses (see the proof of Theorem \ref{lem:local_max_ineq_dengenerate_uprocess}) yields that the right hand side is bounded by, after changing  variables,
\[
 \| P^{r-k}F  \|_{\mathbb{P}_{I_{n,k}},2}^{p} J_{k}^{p}\left (\sigma_{I_{n,k}}/\| P^{r-k}F \|_{\mathbb{P}_{I_{n,k}},2} \right )
\]
up to a constant depending only on $r,p$. The desired result follows from bounding $\sigma_{I_{n,k}}/\| P^{r-k}F \|_{\mathbb{P}_{I_{n,k}},2}$ by $1$ and observation that $\E[\| P^{r-k}F  \|_{\mathbb{P}_{I_{n,k}},2}^{p}] \le \| P^{r-k} F \|_{P^{k},2 \vee p}^{p}$ by Jensen's inequality. 
\end{proof}

\begin{rmk}
Corollary \ref{thm:alternative_max_ineq} is an extension of Theorem 2.14.1 in \cite{vandervaartwellner1996}. 
For $p=1$, Corollary \ref{thm:alternative_max_ineq} is often less sharp than Theorem \ref{lem:local_max_ineq_dengenerate_uprocess}  since $\sigma_{k} \le \| P^{r-k}F \|_{P^{k},2}$ and in some cases $\sigma_{k} \ll \| P^{r-k} F \|_{P^{k},2}$. 
However, Corollary \ref{thm:alternative_max_ineq} is useful for directly  bounding higher order moments of $\| U_{n}^{(k)} (\pi_{k}f) \|_{\calF}$. For the empirical process case (i.e., $k=1$), bounding higher order moments of the supremum is essentially reduced to bounding the first moment by the Hoffmann-J\o rgensen inequality \cite[Proposition A.1.6]{vandervaartwellner1996}. There is an analogous Hoffmann-J\o rgensen type inequality for $U$-processes \citep[see][Theorem 4.1.2]{delaPenaGine1999}, but for $k \ge  2$, bounding higher order moments of $\| U_{n}^{(k)}(\pi_{k}f) \|_{\calF}$ using this Hoffmann-J\o rgensen inequality combined with the local maximal inequality in Theorem \ref{lem:local_max_ineq_dengenerate_uprocess} would be more involved. 
\end{rmk}

%
%

\section{Proofs for Sections \ref{sec:gaussian_approx} and \ref{sec:bootstrap}}
\label{sec:proof}

In what follows, let $\calB(\R)$ denote the Borel $\sigma$-field on $\R$. For a set $B \subset \R$ and $\delta > 0$, let $B^{\delta}$ denote the $\delta$-enlargement of $B$, i.e., $B^{\delta} =\{ x \in \R : \inf_{y \in B} |x-y| \le \delta \}$. 

\subsection{Proofs for Section \ref{sec:gaussian_approx}}

We begin with stating the following lemma.

\begin{lem}
\label{lem:empirical_process}
Work with the setup described in Section \ref{sec:gaussian_approx}. Suppose that Conditions (PM), (VC), and (MT) hold. 
Let $L_{n} := \sup_{g \in \calG} n^{-1/2}\sum_{i=1}^{n} (g(X_{i}) - Pg)$ and
$\tilde{Z} := \sup_{g \in \calG} W_{P}(g)$. Then, there exist universal constants $C,C'> 0$ such that $\Prob(L_n \in B) \le  \Prob (\tilde{Z} \in B^{C \delta_n}) + C'(\gamma + n^{-1})$
for every $B \in \calB(\R)$, where 
\begin{equation}
\label{eq:error_rate}
\delta_n = {(\overline{\sigma}_{\frakg}^2b_{\frakg}  K_n^2)^{1/3} \over \gamma^{1/3} n^{1/6}} + {b_{\frakg} K_n \over \gamma n^{1/2-1/q}}.
\end{equation}
In the case of $q=\infty$, ``$1/q$'' is interpreted as $0$. 
\end{lem}

The proof is a minor modification to that of Theorem 2.1 in \cite{cck2016_empirical_process_coupling}. Differences are 1)  Lemma \ref{lem:empirical_process} allows $q=\infty$, and constants $C,C'$ to be independent of $q$; 2) the error bound $\delta_{n}$ contains $b_{\frakg} K_n/ (\gamma n^{1/2-1/q})$ instead of $b_{\frakg} K_n/(\gamma^{1/q} n^{1/2-1/q})$; and 3) our definition of $K_{n}$ is slightly different from theirs.
For completeness, in Appendix \ref{app:complement}, we provide a sketch of the proof for Lemma \ref{lem:empirical_process}, which points out required modifications to the proof of Theorem 2.1 in \cite{cck2016_empirical_process_coupling}. 

\begin{proof}[Proof of Proposition \ref{prop:entropy_bounds_vc}]
In view of the Strassen-Dudley theorem (see Theorem \ref{thm:strassen-dudley}), it suffices to verify that there exist constants $C,C'$ depending only  $r$ such that 
\[
\Prob (Z_{n} \in B) \le \Prob (\tilde{Z} \in B^{C\varpi_{n}}) + C'(\gamma+n^{-1})
\]
for every $B \in \calB (\R)$. In what follows, $C,C'$ denote generic  constants that depend only on $r$; their values may vary from place to place. 

We shall follow the notation used in Section \ref{sec:local_maximal_inequalities}.
Consider the Hoeffding decomposition for $U_{n}(h) = U_{n}^{(r)}(h)$: $
U_{n}^{(r)}(h) - P^{r}h = r U_{n}^{(1)}(\pi_{1} h) + \sum_{k=2}^{r} \binom{r}{k} U_{n}^{(k)}(\pi_{k}h)$, or
\[
\U_{n} (h) = \sqrt{n} ( U_{n}^{(r)}(h) - P^{r}h) = r \G_{n} (P^{r-1}h) + \sqrt{n} \sum_{k=2}^{r} \binom{r}{k} U_{n}^{(k)}(\pi_{k}h),
\]
where $\G_n(P^{r-1}h) := n^{-1/2}\sum_{i=1}^{n} (P^{r-1}h (X_{i}) - P^{r}h)$ is the H\'ajek (empirical) process associated with $\U_n$.
Recall that $\calG = P^{r-1}\calH = \{ P^{r-1}h : h \in \calH \}$, and let  
$L_n = \sup_{g \in \calG} \G_n(g)$ and $R_n = \| \sqrt{n} \sum_{k=2}^r \binom{r}{k} U_{n}^{(k)}(\pi_{k}h)/r \|_\calH$. Then, since $|Z_n - L_n| \le  R_n$,  Markov's inequality and Lemma \ref{lem:empirical_process} yield that for every $B \in \calB(\R)$, 
\begin{align}
\Prob (Z_{n} \in B) &\le \Prob (\{ Z_{n} \in B \} \cap \{ R_{n} \le \gamma^{-1} \E[R_{n}] \}) + \Prob (R_{n} > \gamma^{-1} \E[ R_{n} ] ) \notag \\
&\le \Prob (L_{n} \in B^{\gamma^{-1} \E[R_{n}]}) + \gamma \notag \\
&\le  \Prob (\tilde{Z} \in B^{C \delta_n + \gamma^{-1} \E[R_{n}]}  ) + C'(\gamma + n^{-1}), \label{eqn:gaussian_coupling_bound_step1}
\end{align}
where $\delta_{n}$ is given in (\ref{eq:error_rate}). 

It remains to bound $\E[R_{n}]$. To this end, we shall separately apply Corollary \ref{cor:local_max_ineq_dengenerate_uprocess_VCtype} for $k=2$ and Corollary \ref{thm:alternative_max_ineq} for $k=3,\dots,r$. 
First,  applying Corollary \ref{cor:local_max_ineq_dengenerate_uprocess_VCtype} to $\calF=\calH$  for $k=2$ yields 
\[
n \E[\| U_{n}^{(2)}(\pi_{2}h) \|_{\calH}] \le  C \left ( \sigma_{\frakh} K_{n} + b_{\frakh} K_{n}^{2} n^{-1/2+1/q} \right ).
\]
Likewise, applying Corollary \ref{thm:alternative_max_ineq} to $\calF= \calH$ for $k=3,\dots,r$ yields 
\[
\sum_{k=3}^{r}\E[\| U_{n}^{(k)}(\pi_{k}h) \|_{\calH}] \le  C \sum_{k=3}^{r} n^{-k/2}\| P^{r-k} H \|_{P^{k},2} K_{n}^{k/2} = Cn^{-1/2} \chi_{n}.  
\]
Therefore, we conclude that 
\begin{align}
\E[R_n]  \le  C\sum_{k=2}^r n^{1/2} \E[\| U_{n}^{(k)}(\pi_{k}h) \|_{\calH} ] \le  C'\left ( \sigma_{\frakh} K_{n} n^{-1/2} + b_{\frakh} K_{n}^{2} n^{-1+1/q} + \chi_{n}\right ). \label{eqn:gaussian_coupling_bound_step2}
\end{align}
Combining (\ref{eqn:gaussian_coupling_bound_step1}) with (\ref{eqn:gaussian_coupling_bound_step2}) leads to the conclusion of the proposition.
\end{proof}

\begin{proof}[Proof of Corollary \ref{cor:kolmogorov_distance_gaussian_coupling_vc}]
We begin with noting that we may assume that $b_{\frakg} \le n^{1/2}$, since otherwise the conclusion is trivial by taking $C \ge 1$. 
In this proof, the notation $\lesssim$ signifies that the left hand side is bounded by the right hand side up to a constant that depends only on $r, \overline{\sigma}_{\frakg}$, and $\underline{\sigma}_{\frakg}$. 
Let $\gamma \in (0,1)$ and pick a version $\tilde{Z}_{n,\gamma}$ of $\tilde{Z}$ as in Proposition \ref{prop:entropy_bounds_vc} ($\tilde{Z}_{n,\gamma}$ may depend on $\gamma$).  Proposition \ref{prop:entropy_bounds_vc} together with \cite[Lemma 2.1]{cck2016_empirical_process_coupling} yield that
\begin{align*}
\rho (Z_n, \tilde{Z}) &= \rho (Z_{n}, \tilde{Z}_{n,\gamma})  \le  \sup_{t \in \R} \Prob(|\tilde{Z}_{n,\gamma} - t| \le  C \varpi_n) + C' (\gamma + n^{-1}) \\
& =\sup_{t \in \R} \Prob ( | \tilde{Z} - t | \le  C \varpi_n) + C' (\gamma + n^{-1}).
\end{align*}
Now, the anti-concentration inequality (see Lemma \ref{lem:AC} in Appendix \ref{app:supporting_lemmas})  yields 
\begin{equation}
\sup_{t \in \R} \Prob (|\tilde{Z} - t | \le  C \varpi_n )  \lesssim \varpi_n \left \{ \E[\tilde{Z}] + \sqrt{1 \vee \log(\underline{\sigma}_{\frakg} / (C \varpi_n))} \right \}.
\label{eq:AC}
\end{equation}
Since $\calG$ is VC type with characteristics $4\sqrt{A}$ and $2v$ for envelope $G$ (Lemma \ref{lem:uniform_entropy_numbers_between_GH}), by Lemma \ref{lem:approximation}, we have  $N(\calG, \| \cdot \|_{P,2}, \tau) \le (16\sqrt{A} \| G \|_{P,2}/\tau)^{2v}$ for all $0 < \varepsilon \le 1$.
Hence, 
 Dudley's entropy integral bound \cite[Theorem 2.3.7]{ginenickl2016} yields   $\E[\tilde{Z}] \lesssim (\overline{\sigma}_{\frakg} \vee (n^{-1/2} b_{\frakg})) K_n^{1/2} \lesssim K_{n}^{1/2}$where the last inequality follows from the assumption that  $b_{\frakg} \le n^{1/2}$. Since $\sqrt{1 \vee \log(\underline{\sigma}_{\frakg} / (C \varpi_n))} \lesssim (K_n \vee \log(\gamma^{-1}) )^{1/2}$, we conclude that
\[
\rho (Z_n, \tilde{Z} ) \lesssim   (K_n \vee \log(\gamma^{-1}) )^{1/2} \varpi_n (\gamma)  + \gamma + n^{-1}.
\]
The desired result follows from balancing $K_{n}^{1/2} \varpi_{n}(\gamma)$ and $\gamma$. 
\end{proof}

\subsection{Proofs for Section \ref{sec:bootstrap}}

\begin{proof}[Proof of Theorem \ref{thm:coupling_gaussian_mulitiplier_bootstrap}]
In this proof we will assume that each $h \in \calH$ is $P^{r}$-centered, i.e., $P^{r}h = 0$ for the rotational convenience. 
Recall that $\Prob_{\mid X_{1}^{n}}$ and $\E_{\mid X_{1}^{n}}$ denote the conditional probability and expectation given $X_{1}^{n}$, respectively. 
In view of  the conditional version of the Strassen-Dudley theorem (see Theorem \ref{thm:conditional_strassen-dudley}), it suffices to find constants $C,C'$ depending only on $r$, and an event $E \in \sigma(X_1^n)$ with $\Prob(E) \ge  1- \gamma - n^{-1}$ on which 
\[
\Prob_{\mid X_{1}^{n}} (Z_n^\sharp \in B) \le  \Prob (\tilde{Z} \in B^{C \varpi_n^\sharp} ) + C' (\gamma + n^{-1}) \quad \forall B \in \calB(\R). 
\]
The proof of Theorem \ref{thm:coupling_gaussian_mulitiplier_bootstrap} is involved and  divided into six steps. In what follows, let $C$ denote a generic positive constant depending only on $r$; the value of $C$ may change from place to place. 

\uline{Step 1: Discretization}.  For $0 < \varepsilon \le 1$ to be determined later, let $N := N(\varepsilon) := N(\calG, \|\cdot \|_{P,2}, \varepsilon \| G \|_{P,2})$. Since $\|G\|_{P,2} \le  b_{\frakg}$, there exists an  $\varepsilon b_{\frakg}$-net $\{ g_{k} \}_{k=1}^{N}$ for $(\calG, \| \cdot \|_{P,2})$.
By the definition of $\calG$, each $g_{k}$ corresponds to a kernel $h_{k} \in \calH$ such that $g_{k} = P^{r-1}h_{k}$. 
The Gaussian process $W_P$ extends to the linear hull of $\calG$ in such a way that $W_{P}$ has linear sample paths \citep[e.g., see][Theorem 3.7.28]{ginenickl2016}. Now, observe that
\[
0 \le   \sup_{g \in \calG} W_P(g) - \max_{1 \le  j \le  N} W_P(g_j) \le  \|W_P\|_{\calG_{\varepsilon}}, \ 0 \le   \sup_{h \in \calH} \U_n^\sharp(h) - \max_{1 \le  j \le  N} \U_n^\sharp(h_j) \le  \|\U_n^\sharp\|_{\calH_{\varepsilon}},
\]
where $\calG_{\varepsilon} = \{g-g' : g, g' \in \calG, \|g-g'\|_{P,2} < 2  \varepsilon b_{\frakg} \}$ and 
$\calH_{\varepsilon} =  \{h-h' : h, h' \in \calH, \| P^{r-1} h- P^{r-1} h' \|_{P,2} < 2  \varepsilon b_{\frakg} \}$.

\uline{Step 2: Construction of a high-probability event $E \in \sigma(X_1^n)$}. We divide this step into several sub-steps.

(i). 
For a $P$-integrable function $g$ on $S$, we will use the notation
\[
\G_{n} (g) :=\frac{1}{\sqrt{n}}\sum_{i=1}^{n} \{ g(X_{i}) - Pg \}.
\]
Consider the function class $\breve{\calG} \cdot \breve{\calG} = \{ gg' : g,g' \in \breve{\calG} \}$ with $\breve{\calG} = \{ g, g - Pg : g \in \calG \}$. 
Recall that $\calG$ with envelope $G$  is VC type with characteristics $(4\sqrt{A},2v)$. 
The function class $\{ g -Pg : g \in \calG \}$ with envelope $\breve{G} := G + PG$ is VC type with characteristics $(4\sqrt{2A},2v+1)$ from a simple calculation. Conclude that $\breve{\calG}$ with envelope $\breve{G}$ is VC type with characteristics $(8\sqrt{2A},2v+1)$, and by Lemma \ref{lem:pointwise_product}, $\breve{\calG} \cdot \breve{\calG}$ with envelope $\breve{G}^{2}$ is VC type with characteristics $(16\sqrt{2A},4v+2)$. For $g,g' \in \calG$, $P(gg')^2 \le \sqrt{Pg^4} \sqrt{P(g')^4} \le \overline{\sigma}_{\frakg}^{2}b_{\frakg}^{2}$ by Condition (MT).
Likewise, 
\[
\begin{split}
&P(g-Pg)^{2}(g'-Pg')^{2} \le \sqrt{P(g-Pg)^4} \sqrt{P(g'-Pg')^4} \le 8 \sqrt{Pg^{4}+(Pg)^{4}}\sqrt{P(g')^{4} + (Pg')^{4}} \\
&\quad \le 16 \sqrt{Pg^{4}} \sqrt{P(g')^{4}} \le 16 \overline{\sigma}_{\frakg}^{2}b_{\frakg}^{2}.
\end{split}
\]
We also note that $\| \breve{G} \|_{P,q} \le 2 \| G \|_{P,q} \le 2b_{\frakg}$. 
Hence, applying Corollary  \ref{cor:local_max_ineq_dengenerate_uprocess_VCtype} with $\calF=\breve{\calG} \cdot \breve{\calG}, r=k=1$, and $q= q/2$ yields 
\[
n^{-1/2} \E[ \| \G_{n} \|_{\breve{\calG} \cdot \breve{\calG}} ] \le  C\left (\overline{\sigma}_{\frakg}b_{\frakg} K_{n}^{1/2} n^{-1/2}+ b_{\frakg}^{2} K_{n} n^{-1+2/q} \right ),
\]
so that with probability at least $1 - \gamma/3$,
\begin{equation}
\label{eqn:event_E_prod_G}
n^{-1/2} \|\G_n\|_{\breve{\calG} \cdot \breve{\calG}} \le  C\gamma^{-1} \left (   \overline{\sigma}_{\frakg} b_{\frakg}K_n^{1/2} n^{-1/2} + b_{\frakg}^2 K_n n^{-1+2/q} \right)
\end{equation}
by Markov's inequality. 

(ii). Define
\begin{equation}
\Upsilon_{n} :=  \left \| \frac{1}{n} \sum_{i=1}^{n} \{ U_{n-1,-i}^{(r-1)} (\delta_{X_{i}}h)  - P^{r-1}h(X_{i}) \}^{2} \right \|_{\calH}.
\label{eq:Upsilon}
\end{equation}
We will show that
\begin{equation}
\label{eq:upsilon_bound}
\E[\Upsilon_{n}] 
\le C \left \{   \sigma_{\frakh}^{2}  K_{n}n^{-1}+ \nu_{\frakh}^{2} K_{n}^{2}n^{-3/2+2/q}  + \sigma_{\frakh} b_{\frakh} K_{n}^{3/2} n^{-3/2}  +  b_{\frakh}^{2} K_{n}^{3}n^{-2+2/q} + \chi_{n}^{2} \right \}.
\end{equation}
Together with Markov's inequality, we have that with probability at least $1-\gamma/3$, 
\begin{equation}
\label{eqn:event_E_prod_F}
\Upsilon_{n} \le  C\gamma^{-1}\left \{   \sigma_{\frakh}^{2} K_{n}n^{-1}+ \nu_{\frakh}^{2} K_{n}^{2}n^{-3/2+2/q}  + \sigma_{\frakh} b_{\frakh} K_{n}^{3/2} n^{-3/2}  +  b_{\frakh}^{2} K_{n}^{3}n^{-2+2/q} +  \chi_{n}^{2} \right \}.
\end{equation}
The proof of the inequality (\ref{eq:upsilon_bound}) is lengthy and deferred after the proof of the theorem.

(iii). 
We shall bound $\E[ \| U_{n}(h) - P^{r}h \|_{\calH}^{2}]$. 
Applying Corollary \ref{thm:alternative_max_ineq} to $\calH$ for  $k=2,\dots,r$ yields
\[
\sum_{k=2}^{r} \E[ \| U_{n}^{(k)}(\pi_{k}h) \|_{\calH}^{2}] \le C \left ( b_{\frakh}^{2}K_{n}^{2}n^{-2}  + n^{-1} \chi_{n}^{2} \right).
\]
Next, since $U_{n}^{(1)}(\pi_{1}h), h \in \calH$ is an empirical process, we may apply the Hoffmann-J\o rgensen inequality \cite[Proposition A.1.6]{vandervaartwellner1996} to deduce that 
\begin{align*}
\E[ \| U_{n}^{(1)}(\pi_{1}h) \|_{\calH}^{2}] &\le C\left \{ (\E[\| U_{n}^{(1)}(\pi_{1}h) \|_{\calH}])^{2} + b_{\frakg}^{2} n^{-2+2/q}\right \} \\
&\le C \left ( \overline{\sigma}_{\frakg}^{2} K_{n} n^{-1} + b_{\frakg}^{2} K_{n}^{2}n^{-2+2/q} + b_{\frakg}^{2} n^{-2+2/q} \right)\\
&\le C\left ( \overline{\sigma}_{\frakg}^{2} K_{n} n^{-1} + b_{\frakg}^{2} K_{n}^{2}n^{-2+2/q} \right ),
\end{align*}
where the second inequality follows from Corollary \ref{cor:local_max_ineq_dengenerate_uprocess_VCtype}. Since $\overline{\sigma}_{\frakg} \le \sigma_{\frakh}$ and $b_{\frakg} \le b_{\frakh}$, 
\[
\E[ \| U_{n}(h) - P^{r}h \|_{\calH}^{2}] 
\le  C \left ( \sigma_{\frakh}^{2} K_{n} n^{-1} + b_{\frakh}^{2} K_{n}^{2} n^{-2+2/q} + n^{-1} \chi_{n}^{2} \right ),
\]
so that by Markov's inequality, with probability at least $1-\gamma/3$, 
\begin{equation}
\label{eqn:event_E_U2_n}
\| U_{n}(h) - P^{r}h \|_{\calH}^{2} \le  C \gamma^{-1}\left (\sigma_{\frakh}^{2} K_{n} n^{-1} + b_{\frakh}^{2} K_{n}^{2} n^{-2+2/q} + n^{-1} \chi_{n}^{2}\right). 
\end{equation}

(iv). Let $\Prob_{I_{n,r}} = |I_{n,r}|^{-1} \sum_{(i_1,\dots,i_r) \in I_{n,r}} \delta_{(X_{i_1},\dots,X_{i_r})}$ denote the empirical distribution on all possible $r$-tuples of $X_1^n$. Then Markov's inequality yields that with probability at least $1-n^{-1}$,
\begin{equation}
\label{eqn:event_E_envelope_H}
\|H\|_{\Prob_{I_{n,r}},2} \le  n^{1/2} \|H\|_{P^{r},2}.
\end{equation}

Now, define the event $E$ by the  the intersection of the events  (\ref{eqn:event_E_prod_G}), (\ref{eqn:event_E_prod_F}), (\ref{eqn:event_E_U2_n}), and (\ref{eqn:event_E_envelope_H}). Then, $E \in \sigma(X_1^n)$ and $\Prob(E) \ge  1 - \gamma - n^{-1}$.

\uline{Step 3: Bounding the discretization error for  $W_P$}. By the Borell-Sudakov-Tsirel'son inequality \citep[cf.][Theorem 2.5.8]{ginenickl2016}, we have 
\[
\Prob \left (\|W_P\|_{\calG_{\varepsilon}} \ge  \E [\|W_P\|_{\calG_{\varepsilon}}] + 2\varepsilon b_{\frakg} \sqrt{2 \log{n}} \right) \le   n^{-1}.
\]
From a standard calculation,  $N(\calG_{\varepsilon}, \| \cdot \|_{P,2}, \tau) \le  N^2(\calG, \| \cdot \|_{P,2}, \tau/2)$. 
Since $\calG$ is VC type with characteristics $4\sqrt{A}$ and $2v$ for envelope $G$, 
by Lemma \ref{lem:approximation},  we have  $N(\calG, \| \cdot \|_{P,2}, \tau \|G\|_{P,2}) \le  C (16 \sqrt{A} / \tau)^{2v}$, so that $
N(\calG_{\varepsilon}, \| \cdot \|_{P,2}, \tau) \le  (32 \sqrt{A} b_{\frakg} / \tau)^{4v}$. 
Now, Dudley's entropy integral bound \cite[Corollary 2.2.8]{vandervaartwellner1996} yields 
\[
\E [\|W_P\|_{\calG_{\varepsilon}} ] \le  C (\varepsilon b_{\frakg}) \sqrt{v \log (A / \varepsilon) }.
\]
Choosing $\varepsilon = 1/n^{1/2}$, we have
\[
\E [\|W_P\|_{\calG_{\varepsilon}}] \le  C b_{\frakg} n^{-1/2} \sqrt{v \log(A n^{1/2})} \le  C b_{\frakg}K_{n}^{1/2}n^{-1/2}.
\]
Since $\log n \le K_n$, we conclude that 
\[
\Prob \left ( \|W_P\|_{\calG_{\varepsilon}} \ge  C b_{\frakg}K_{n}^{1/2}n^{-1/2} \right ) \le   n^{-1}.
\]

\uline{Step 4: Bounding the discretization error for  $\U_n^\sharp$}.  Since $\{ \U_{n}^{\sharp} (h) : h \in \calH \}$ is a centered Gaussian process conditionally on $X_{1}^{n}$, applying the Borell-Sudakov-Tsirel'son inequality conditionally on $X_1^n$, we have
\[
\Prob_{\mid X_{1}^{n}} \left ( \|\U_n^\sharp\|_{\calH_{\varepsilon}} \ge  \E_{\mid X_{1}^{n}}[\|\U_n^\sharp\|_{\calH_{\varepsilon}} ]+ \sqrt{2\Sigma_{n}  \log{n}} \right) \le   n^{-1},
\]
where $\Sigma_{n} := \| n^{-1} \sum_{i=1}^{n}\{ U_{n-1,-i}^{(r-1)} (\delta_{X_{i}}h) - U_{n}(h) \}^{2} \|_{\calH_{\varepsilon}}$ with $\varepsilon = 1/n^{1/2}$.

We begin with bounding $\Sigma_{n}$. 
For any $h \in \calH_{\varepsilon}$, $n^{-1}\sum_{i=1}^{n}\{ U_{n-1,-i}^{(r-1)} (\delta_{X_{i}}h)- U_{n}(h) \}^{2}$ is bounded by $n^{-1}\sum_{i=1}^{n}\{ U_{n-1,-i}^{(r-1)} (\delta_{X_{i}}h)\}^{2}$ since the average of $U_{n-1,-i}^{(r-1)} (\delta_{X_{i}}h), i=1,\dots,n$ is $U_{n}(h)$ and the variance is bounded by the second moment. Further, the term $n^{-1}\sum_{i=1}^{n}\{ U_{n-1,-i}^{(r-1)} (\delta_{X_{i}}h)\}^{2}$ is bounded by 
\begin{equation}
\begin{split}
&\frac{2}{n} \sum_{i=1}^{n} \{ U_{n-1,-i}^{(r-1)} (\delta_{X_{i}}h)- P^{r-1}h (X_{i})  \}^{2}   \\
&\quad +\frac{2}{n} \sum_{i=1}^{n} \{ (P^{r-1}h(X_{i}))^{2} - P(P^{r-1}h)^{2} \} + 2P(P^{r-1}h)^{2}.
\end{split}
\label{eq:decomposition}
\end{equation}
The last term on the right hand side of (\ref{eq:decomposition}) is bounded by $8(\varepsilon b_{\frakg})^{2}$. 
The supremum of the first term on $\calH_{\varepsilon}$ is bounded by $8\Upsilon_{n}$ since $\calH_{\varepsilon} \subset  \{ h-h' : h,h' \in \calH \}$ (the notation $\Upsilon_{n}$ is defined in (\ref{eq:Upsilon})). 
For the second term, observe that $\{ (P^{r-1}h)^{2} : h \in \calH_{\varepsilon} \} \subset \{ (g-g')^{2} : g,g' \in \calG \}, (g-g')^{2} - P(g-g')^{2} = (g^{2} - Pg^{2}) + 2(gg' - Pgg') + ((g')^{2} - P(g')^{2})$, and $\{ g^{2} : g \in \calG\} \subset \breve{\calG} \cdot \breve{\calG}$, so that the supremum of the second term on the right hand side of (\ref{eq:decomposition}) is bounded by $8n^{-1/2} \| \G_{n} \|_{\breve{\calG} \cdot \breve{\calG}}$.
Therefore, recalling that we have chosen $\varepsilon=1/n^{1/2}$, we conclude that
\begin{align*}
\Sigma_{n} &\le 8(\varepsilon b_{\frakg})^{2} + 8n^{-1/2} \| \G_{n} \|_{\breve{\calG} \cdot \breve{\calG}} +8 \Upsilon_{n} \\
&\le  C\gamma^{-1} \Bigg \{ \overline{\sigma}_{\frakg} b_{\frakg} K_n^{1/2}  n^{-1/2} + b_{\frakg}^2 K_n  n^{-1+2/q}  + \sigma_{\frakh}^{2} K_{n}n^{-1} \\
&\qquad \qquad \qquad + \nu_{\frakh}^{2} K_{n}^{2}n^{-3/2+2/q}  + \sigma_{\frakh} b_{\frakh} K_{n}^{3/2} n^{-3/2}  +  b_{\frakh}^{2} K_{n}^{3}n^{-2+2/q}+ \chi_{n}^{2}  \Bigg \}
\end{align*}
on the event $E$.

Next, we shall bound $\E_{\mid X_{1}^{n}} [\|\U_n^\sharp\|_{\calH_{\varepsilon}}]$ on the event $E$. Since $\calH$ is VC type with characteristics $(A,v)$, we have
\[
N(\calH_{\varepsilon}, \| \cdot \|_{\Prob_{I_{n,r}},2}, 2\tau \|H\|_{\Prob_{I_{n,r}},2}) \le  N^{2}(\calH, \| \cdot \|_{\Prob_{I_{n,r}},2}, \tau \|H\|_{\Prob_{I_{n,r}},2}) \le  (A/\tau)^{2v}.
\]
In addition, since
\begin{align*}
&d^{2}(h,h') := \E_{\mid X_{1}^{n}} [ \{ \U_n^\sharp(h) - \U_n^\sharp(h') \}^{2} ] \\
&\quad =\frac{1}{n} \sum_{i=1}^{n}\{ U_{n-1,-i}^{(r-1)} (\delta_{X_{i}}h) - U_{n}(h) - U_{n-1,-i}^{(r-1)} (\delta_{X_{i}}h') + U_{n}(h') \}^{2} \\
&\quad \le \frac{1}{n} \sum_{i=1}^{n}\{ U_{n-1,-i}^{(r-1)} (\delta_{X_{i}}h) -  U_{n-1,-i}^{(r-1)} (\delta_{X_{i}}h')  \}^{2} \le \| h - h ' \|_{\Prob_{I_{n,r}},2}^{2},
\end{align*}
where the last inequality follows from Jensen's inequality, and since a weaker pseudometric induces a smaller covering number, we have
\[
N(\calH_{\varepsilon}, d, 2\tau \|H\|_{\Prob_{I_{n,r}},2}) \le N(\calH_{\varepsilon}, \| \cdot \|_{\Prob_{I_{n,r}},2}, 2\tau \|H\|_{\Prob_{I_{n,r}},2}) \le  (A/\tau)^{2v}.
\]
Hence, using $2\left [ (n^{-(r-1)/2}\| H \|_{P^{r},2}) \vee \Sigma_{n}^{1/2} \right ]$ as a bound on the $d$-diameter of $\calH_{\varepsilon}$, we have by Dudley's entropy integral bound 
\begin{align*}
\E_{\mid X_{1}^{n}} [\|\U_n^\sharp\|_{\calH_{\varepsilon}}]  
&\le C \int_{0}^{(n^{-(r-1)/2}\| H \|_{P^{r},2}) \vee \Sigma_{n}^{1/2}} \sqrt{v \log(A \| H \|_{\Prob_{I_{n,r},2}}/\tau)} d\tau \\
&\le  C\left (  (n^{-(r-1)/2}\| H \|_{P^{r},2}) \vee \Sigma_{n}^{1/2} \right )\sqrt{v\log (A\| H \|_{\Prob_{I_{n,r},2}}/(n^{-(r-1)/2}\| H \|_{P^{r},2}))} \\
&\le  C\left (  (n^{-(r-1)/2}\| H \|_{P^{r},2}) \vee \Sigma_{n}^{1/2} \right )\sqrt{v\log (An^{r/2})} 
\end{align*}
on the event $E$ (we have used $\| H \|_{\Prob_{I_{n,k},2}} \le n^{1/2} \| H \|_{P^{r},2}$ on $E$). Since $n^{-(r-1)/2}\| H \|_{P^{r},2} \le \chi_{n}$, we have 
\begin{align*}
\E_{\mid X_{1}^{n}}[\|\U_n^\sharp\|_{\calH_{\varepsilon}}] &\le  C (\chi_{n} \vee \Sigma_n^{1/2}) K_n^{1/2} \\
&\le  C \gamma^{-1/2} \Bigg \{ (\overline{\sigma}_{\frakg} b_{\frakg} K_n^{3/2})^{1/2}  n^{-1/4} + b_{\frakg} K_n n^{-1/2+1/q}  + \sigma_{\frakh}  K_{n}n^{-1/2} \\
&\qquad \qquad + \nu_{\frakh} K_{n}^{3/2} n^{-3/4+1/q}  + (\sigma_{\frakh} b_{\frakh})^{1/2} K_{n}^{5/4} n^{-3/4}  +  b_{\frakh} K_{n}^{2}n^{-1+1/q}+ \chi_{n}K_{n}^{1/2} \Bigg \}
\end{align*}
on the event $E$. 
Hence, we conclude that
$$
\Prob_{\mid X_{1}^{n}} (\|\U_n^\sharp\|_{\calH_{\varepsilon}} \ge  C \delta_n^{(1)}) \le   n^{-1}
$$
on the event $E$, where
\[
\begin{split}
\delta_n^{(1)} &= \frac{1}{\gamma^{1/2}} \Bigg \{  {(\overline{\sigma}_{\frakg} b_{\frakg}K_n^{3/2})^{1/2} \over n^{1/4}} + {b_{\frakg} K_n \over  n^{1/2-1/q}} + {\sigma_{\frakh}  K_n \over  n^{1/2}} + \frac{\nu_{\frakh}K_{n}^{3/2}}{n^{3/4-1/q}} \\
&\qquad \qquad + {(\sigma_{\frakh} b_{\frakh})^{1/2} K_{n}^{5/4}\over n^{3/4}} + \frac{b_{\frakh} K_{n}^{2}}{n^{1-1/q}} + \chi_{n} K_{n}^{1/2} \Bigg \}.
\end{split}
\]

\uline{Step 5: Gaussian comparison}.  Let $Z_n^{\sharp,\varepsilon} := \max_{1 \le  j \le  N} \U_n^\sharp(h_j)$ and $\tilde{Z}^{\varepsilon} := \max_{1 \le  j \le  N} W_P(g_j)$. 
 Observe that the covariance between $\U_{n}^{\sharp}(h_{k})$ and $\U_{n}^{\sharp}(h_{\ell})$ conditionally on $X_{1}^{n}$ is 
\begin{align*}
\hat{C}_{k,\ell} &:= \frac{1}{n} \sum_{i=1}^{n} \{ U_{n-1,-i}^{(r-1)} (\delta_{X_{i}}h_{k})- U_{n}(h_{k}) \} \{ U_{n-1,-i}^{(r-1)} (\delta_{X_{i}}h_{\ell})- U_{n}(h_{\ell}) \} \\
&=\frac{1}{n}  \sum_{i=1}^{n}  U_{n-1,-i}^{(r-1)} (\delta_{X_{i}}h_{k})U_{n-1,-i}^{(r-1)} (\delta_{X_{i}}h_{\ell}) - U_{n} (h_{k}) U_{n}(h_{\ell}) \\
&= \frac{1}{n}  \sum_{i=1}^{n} \{ U_{n-1,-i}^{(r-1)} (\delta_{X_{i}}h_{k})- P^{r-1}h_{k}(X_{i}) \}\{ U_{n-1,-i}^{(r-1)} (\delta_{X_{i}}h_{\ell}) - P^{r-1}h_{\ell}(X_{i}) \} \\
&\quad + \frac{1}{n}  \sum_{i=1}^{n} \{ U_{n-1,-i}^{(r-1)} (\delta_{X_{i}}h_{k})- P^{r-1}h_{k}(X_{i}) \} P^{r-1}h_{\ell}(X_{i}) \\
&\quad + \frac{1}{n}  \sum_{i=1}^{n} \{ U_{n-1,-i}^{(r-1)} (\delta_{X_{i}}h_{\ell})- P^{r-1}h_{\ell}(X_{i}) \}P^{r-1}h_{k}(X_{i}) \\
&\quad + \frac{1}{n} \sum_{i=1}^{n} (P^{r-1}h_{k}(X_{i})) (P^{r-1}h_{\ell}(X_{i})) - U_{n} (h_{k}) U_{n}(h_{\ell}).
\end{align*}
Recall that $g_{k} = P^{r-1}h_{k}$ for each $k$. Replacing $h_{k}$ by $h_{k} - P^{r}h_{k}$ in the above expansion, we have 
\begin{align*}
&\left |\hat{C}_{k,\ell} - P(g_{k}-Pg_{k}) (g_{\ell} - Pg_{\ell})\right | \\
&\le \left [ \frac{1}{n}  \sum_{i=1}^{n} \{ U_{n-1,-i}^{(r-1)} (\delta_{X_{i}}h_{k})- P^{r-1}h_{k}(X_{i}) \}^{2} \right]^{1/2}\left [ \frac{1}{n}  \sum_{i=1}^{n} \{ U_{n-1,-i}^{(r-1)} (\delta_{X_{i}}h_{\ell})- P^{r-1}h_{\ell}(X_{i}) \}^{2} \right]^{1/2} \\
&\quad + \left [ \frac{1}{n}  \sum_{i=1}^{n} \{ U_{n-1,-i}^{(r-1)} (\delta_{X_{i}}h_{k})- P^{r-1}h_{k}(X_{i}) \}^{2} \right]^{1/2}\left [ \frac{1}{n} \sum_{i=1}^{n}\{  g_{\ell} (X_{i}) - Pg_{\ell} \}^{2}  \right ]^{1/2} \\
&\quad +\left [ \frac{1}{n}  \sum_{i=1}^{n} \{ U_{n-1,-i}^{(r-1)} (\delta_{X_{i}}h_{\ell})- P^{r-1}h_{\ell}(X_{i}) \}^{2} \right]^{1/2} \left [ \frac{1}{n} \sum_{i=1}^{n} \{ g_{k} (X_{i}) - Pg_{k} \}^{2} \right ]^{1/2} \\
&\quad + n^{-1/2} |\G_{n}\left ((g_{k}- Pg_{k})(g_{\ell}-Pg_{\ell}) \right )| + | (U_{n}(h_{k}) - P^{r}h_{k})(U_{n}(h_{\ell})- P^{r}h_{\ell}) |,
\end{align*}
where we have used the Cauchy-Schwarz inequality.
Since $n^{-1} \sum_{i=1}^{n} \{ g (X_{i}) - Pg \}^{2}$ is decomposed as $P(g - Pg)^{2} + n^{-1/2} \G_{n}((g-Pg)^{2})$ and the supremum of the latter on $\calG$ is bounded by $\overline{\sigma}_{\frakg}^{2}+n^{-1/2} \| \G_{n} \|_{\breve{\calG} \cdot \breve{\calG}}$, we have 
\begin{align*}
\Delta_{n} &:= \max_{1 \le k,\ell \le N}\left |\hat{C}_{k,\ell} - P(g_{k}-Pg_{k}) (g_{\ell} - Pg_{\ell})\right | \\
&\le \Upsilon_{n} + 2\overline{\sigma}_{\frakg}\Upsilon_{n}^{1/2}  + 2 n^{-1/4} \Upsilon_{n}^{1/2} \| \G_{n} \|^{1/2}_{\breve{\calG} \cdot \breve{\calG}} + n^{-1/2} \| \G_{n} \|_{\breve{\calG} \cdot \breve{\calG}} + \| U_{n}(h) - P^{r}h \|_{\calH}^{2} \\
&\le 2\Upsilon_{n} + 2\overline{\sigma}_{\frakg}\Upsilon_{n}^{1/2}  + 2n^{-1/2} \| \G_{n} \|_{\breve{\calG} \cdot \breve{\calG}} + \| U_{n}(h) - P^{r}h \|_{\calH}^{2},
\end{align*}
where the second inequality follows from the inequality $2ab \le a^{2} + b^{2}$ for $a,b \in \R$. 
Now, Condition (\ref{eq:growthcondition}) ensures that 
\[
\begin{split}
&\Upsilon_{n} \bigvee (\overline{\sigma}_{\frakg} \Upsilon_{n}^{1/2}) \bigvee  \| U_{n}(h) - P^{r}h \|_{\calH}^{2} \\
&\quad \le  C \gamma^{-1}\overline{\sigma}_{\frakg} \left \{  \sigma_{\frakh}  K_{n}^{1/2}n^{-1/2} + \nu_{\frakh} K_{n} n^{-3/4+1/q} + (\sigma_{\frakh} b_{\frakh})^{1/2} K_{n}^{3/4} n^{-3/4}  +  b_{\frakh} K_{n}^{3/2}n^{-1+1/q}+ \chi_{n} \right  \}
\end{split}
\]
on the event $E$, so that 
\begin{align*}
\Delta_{n} \le  C\gamma^{-1} \Bigg [ &(b_{\frakg} \vee \sigma_{\frakh})\overline{\sigma}_{\frakg}K_n^{1/2} n^{-1/2}  + 
b_{\frakg}^{2} K_{n}n^{-1+2/q} \\
&\qquad+ \overline{\sigma}_{\frakg} \left \{  \nu_{\frakh} K_{n} n^{-3/4+1/q}  + (\sigma_{\frakh} b_{\frakh})^{1/2} K_{n}^{3/4} n^{-3/4}  +  b_{\frakh} K_{n}^{3/2}n^{-1+1/q}+ \chi_{n} \right \}\Bigg ] =: \overline{\Delta}_{n}.
\end{align*}
Therefore, the Gaussian comparison inequality of \cite[Theorem 3.2]{cck2016_empirical_process_coupling} yields  that on the event  $E$, 
\[
\Prob_{\mid X_{1}^{n}} (Z_n^{\sharp, \varepsilon} \in B) \le  \Prob(\tilde{Z}^{\varepsilon} \in B^\eta) + C \eta^{-1} \overline{\Delta}_{n}^{1/2} K_{n}^{1/2} \quad \forall B \in \calB(\R), \ \forall \eta > 0. 
\]

\uline{Step 6: Conclusion}. Let
\[
\begin{split}
\delta_n^{(2)} := \frac{1}{\gamma^{1/2}} \Bigg \{ &\frac{ \{  (b_{\frakg} \vee \sigma_{\frakh} ) \overline{\sigma}_{\frakg}K_n^{3/2} \}^{1/2}}{n^{1/4}} + {b_{\frakg} K_n \over n^{1/2-1/q}} +  {(\overline{\sigma}_{\frakg} \nu_{\frakh})^{1/2}K_{n}  \over n^{3/8-1/(2q)}} \\
&\qquad + \frac{\overline{\sigma}_{\frakg}^{1/2}(\sigma_{\frakh} b_{\frakh})^{1/4}K_{n}^{7/8}}{n^{3/8}} 
+ \frac{(\overline{\sigma}_{\frakg}b_{\frakh})^{1/2}K_{n}^{5/4}}{n^{1/2-1/(2q)}} + \overline{\sigma}_{\frakg}^{1/2} \chi_{n}^{1/2}K_{n}^{1/2} \Bigg \}.
\end{split}
\]
Then, from Steps 1--5,  we have for every $B \in \calB(\R)$ and $\eta > 0$, 
\begin{align*}
\Prob_{\mid X_{1}^{n}}(Z_n^\sharp \in B) &\le  \Prob_{\mid X_{1}^{n}}(Z_n^{\sharp,\varepsilon} \in B^{C \delta_n^{(1)}}) + n^{-1} \\
&\le  \Prob(\tilde{Z}^{\varepsilon} \in B^{C \delta_n^{(1)} + \eta}) + C \eta^{-1} \delta_n^{(2)} + n^{-1} \\
&\le  \Prob(\tilde{Z} \in B^{C \delta_n^{(1)} + \eta + C b_{\frakg}K_{n}^{1/2}n^{-1/2}}) + C \eta^{-1} \delta_n^{(2)} + 2 n^{-1}.
\end{align*}
Choosing $\eta = \gamma^{-1} \delta_{n}^{(2)}$ leads to the conclusion of the theorem. 
\end{proof}

It remains to prove the inequality (\ref{eq:upsilon_bound}). 

\begin{proof}[Proof of the inequality (\ref{eq:upsilon_bound})]

For a $P^{r-1}$-integrable symmetric function $f$ on $S^{r-1}$, $U_{n-1,-i}^{(r-1)} (f) $ is a $U$-statistic of order $r-1$ and its first projection term is 
\[
\frac{r-1}{n-1} \sum_{j =1, \neq i}^{n} \{ P^{r-2} f (X_{j})-P^{r-1}f \} =:S_{n-1,-i}(f).
\]
Consider the following decomposition:
\begin{equation}
\begin{split}
&\frac{1}{n}\sum_{i=1}^{n}  \{ U_{n-1,-i}^{(r-1)} (\delta_{X_{i}}h)  - P^{r-1}h(X_{i}) \}^{2}  \\
&\le \frac{2}{n} \sum_{i=1}^{n}  \{ S_{n-1,-i}(\delta_{X_{i}}h) \}^{2} +  \frac{2}{n}\sum_{i=1}^{n}\{ U_{n-1,-i}^{(r-1)} (\delta_{X_{i}}h)  - P^{r-1}(\delta_{X_{i}}h) - S_{n-1,-i}(\delta_{X_{i}}h) \}^{2}. 
\end{split}
\label{eq:decomp}
\end{equation}

Consider the second term. By Corollary \ref{cor:entropy_conditional}, for given $x \in S$,  $\delta_{x} \calH = \{ \delta_{x} x : h \in \calH \}$ is VC type  with characteristics $(A,v)$ for envelope $\delta_{x}H$.
Hence, we apply Corollary \ref{thm:alternative_max_ineq} conditionally on $X_{i}$ and deduce that
\begin{align*}
&\E \left [ \E \left [  \left \| U_{n-1,-i}^{(r-1)} (\delta_{X_{i}}h)  - P^{r-1}(\delta_{X_{i}}h) - S_{n-1,-i}(\delta_{X_{i}}h) \right \|_{\calH}^{2}  \ \Big | \ X_{i} \right ] \right ]\\
&\quad\le  C\sum_{k=2}^{r-1} n^{-k} \E\left [ \| P^{r-k-1} (\delta_{x}H) \|_{P^{k},2}^{2}|_{x=X_{i}} \right ] K_{n}^{k}
=C\sum_{k=2}^{r-1} n^{-k} \| P^{r-k-1}H \|_{P^{k+1},2}^{2} K_{n}^{k}.
\end{align*}
Since $\sum_{k=2}^{r-1} n^{-k} \| P^{r-k-1}H \|_{P^{k+1},2}^{2} K_{n}^{k} = \sum_{k=3}^{r}n^{-(k-1)} \| P^{r-k}H \|_{P^{k},2}^{2} K_{n}^{k-1} \le  C\chi_{n}^{2}$, the expectation of the supremum on $\calH$ of the second term on the right hand side of (\ref{eq:decomp})  is at most $ C \chi_{n}^{2}$. 

For the first term, observe that 
\[
\begin{split}
&n^{-1} \sum_{i=1}^{n} \{ S_{n-1,-i}(\delta_{X_{i}}h) \}^{2} \\
&= \frac{(r-1)^{2}}{n(n-1)^{2}} \sum_{i=1}^{n} \sum_{j \neq i} \sum_{k \neq i} \Bigg \{ (P^{r-2}h)(X_{i},X_{j}) (P^{r-2}h)(X_{i},X_{k}) 
-  (P^{r-2}h)(X_{i},X_{j}) (P^{r-1}h)(X_{i}) \\
&\quad - (P^{r-2}h)(X_{i},X_{k}) (P^{r-1}h)(X_{i}) + (P^{r-1}h)^{2}(X_{i})\Bigg \}.
\end{split}
\]
Let $\calF = \{ P^{r-2} h : h \in \calH \}$ and $F=P^{r-2}H$, and  observe that for $f \in \calF$,
\begin{align*}
&\sum_{i=1}^{n} \sum_{j \neq i} \sum_{k \neq i} \left \{ f(X_{i},X_{j}) f(X_{i},X_{k}) 
- f(X_{i},X_{j}) (Pf)(X_{i}) - f(X_{i},X_{k}) (Pf)(X_{i}) + (Pf)^{2}(X_{i})\right \} \\
&=n(n-1)\{ P^{2}f^{2} - P (Pf)^{2} \}  \\
&\quad +\sum_{(i,j) \in I_{n,2}} \left \{ f^{2}(X_{i},X_{j}) - 2 f(X_{i},X_{j}) (Pf)(X_{i}) + (Pf)^{2}(X_{i}) - P^{2}f^{2} + P(Pf)^{2} \right \}\\
&\quad + \sum_{(i,j,k) \in I_{n,3}}  \left \{ f(X_{i},X_{j}) f(X_{i},X_{k}) 
- f(X_{i},X_{j}) (Pf)(X_{i}) - f(X_{i},X_{k}) (Pf)(X_{i}) + (Pf)^{2}(X_{i})\right \}.
\end{align*}
Since $P^{2}f^{2} - P (Pf)^{2} \le \sigma_{\frakh}^{2}$, we focus on bounding the suprema of the last two terms. The second term is proportional to a non-degenerate $U$-statistic of order $2$, and the third term is proportional to a degenerate $U$-statistic of order $3$. 
Define the function classes
\[
\calF_{1} := \left \{ (x_{1},x_{2}) \mapsto f^{2}(x_{1},x_{2}) - 2f(x_{1},x_{2}) (Pf)(x_{1}) + (Pf)^{2}(x_{1}) : f \in \calF \right \},
\]
\[
\calF^{0}_{2} := \left \{  (x_{1},x_{2},x_{3}) \mapsto 
\left \{ 
\begin{split}
&f(x_{1},x_{2})f(x_{1},x_{3}) - f(x_{1},x_{2}) (Pf)(x_{1}) \\
& -f(x_{1},x_{3})(Pf)(x_{1})+  (Pf)^{2}(x_{1})
\end{split}
\right \}
: f \in \calF \right \}, \quad 
\]
\[
\begin{split}
\calF_{2} &:= \left \{ (x_{2},x_{3}) \mapsto \E[ f(X_{1},x_{2},x_{3})] : f \in \calF_{2}^{0} \right \}, \\
\calF_{3} &:= \left \{ (x_{1},x_{2},x_{3}) \mapsto f(x_{1},x_{2},x_{3}) - \E[ f(X_{1},x_{2},x_{3})] : f \in \calF_{2}^{0} \right \},
\end{split}
\]
together with their envelopes 
\begin{align*}
&F_{1}(x_{1},x_{2}) := F^{2} (x_{1},x_{2}) +  2F(x_{1},x_{2}) (PF)(x_{1}) + (PF)^{2}(x_{1}), \\
&F_{2}^{0}(x_{1},x_{2},x_{3}) := F(x_{1},x_{2})F(x_{1},x_{3}) + F(x_{1},x_{2}) (PF)(x_{1}) +  F(x_{1},x_{3}) (PF)(x_{1}) +  (PF)^{2}(x_{1}), \\
&F_{2}(x_{2},x_{3}) := \E[ F_{2}^{0}(X_{1},x_{2},x_{3})], \\
&F_{3} (x_{1},x_{2},x_{3}) := F_{2}^{0}(x_{1},x_{2},x_{3}) + F_{2}(x_{2},x_{3}),
\end{align*}
respectively.
Lemma \ref{lem:uniform_entropy_numbers_between_GH} yields that $\calF$ is VC type with characteristics $(4\sqrt{A}, 2v)$ for envelope $F$, and Corollary A.1 (i) in \cite{cck2014_empirical_process} together with Lemma \ref{lem:uniform_entropy_numbers_between_GH} yield that $\calF_{1},\calF_{2},\calF_{3}$ are VC type with characteristics bounded by $CA,Cv$ for envelopes $F_{1},F_{2},F_{3}$, respectively. Functions in $\calF_{1}$ are not symmetric, but after symmetrization we may apply Corollaries \ref{cor:local_max_ineq_dengenerate_uprocess_VCtype} and \ref{thm:alternative_max_ineq}  for $k=1$ and $k=2$, respectively. Together with the Jensen and Cauchy-Schwarz inequalities, we deduce that
\begin{align*}
\E [ \| U_{n}^{(2)}(f) - P^{2}f \|_{\calF_{1}} ] 
&\le  C \left \{ \sup_{f \in \calF} \| f^{2} \|_{P^{2},2} K_{n}^{1/2} n^{-1/2} + \| F^{2} \|_{P^{2},q/2} K_{n}n^{-1+2/q} + \| F^{2} \|_{P^{2},2} K_{n}n^{-1} \right \} \\
&\le  C \left ( \sigma_{\frakh} b_{\frakh} K_{n}^{1/2} n^{-1/2} + b_{\frakh}^{2} K_{n}n^{-1+2/q} \right ),
\end{align*}
where we have used  $\| P^{r-2}h \|_{P^{2},4}^{4} \le \sigma_{\frakh}^{2} b_{\frakh}^{2}$ for $h \in \calH$ by Condition (MT). 

Next, observe that $\| U_{n}^{(3)}(f) \|_{\calF_{2}^{0}} \le \| U_{n}^{(2)}(f) \|_{\calF_{2}} + \| U_{n}^{(3)}(f) \|_{\calF_{3}}$.
Since  for $f \in \calF_{2}^{0}$, $\E[ f(x_{1},X_{2},X_{3})] = \E[ f(X_{1},x_{2},X_{3}) ] = \E[f(X_{1},X_{2},x_{3})] =\E[ f(x_{1},X_{2},x_{3}) ] = \E[ f(x_{1},x_{2},X_{3})] = 0$ for all $x_{1},x_{2},x_{3} \in S$, both $U_{n}^{(2)}(f), f \in \calF_{2}$ and $U_{n}^{(3)}(f), f \in \calF_{3}$ are completely degenerate.
So,  applying Corollary \ref{cor:local_max_ineq_dengenerate_uprocess_VCtype} to $\calF_{2}$ and $\calF_{3}$ after symmetrization, combined  with the Jensen and Cauchy-Schwarz inequalities, we deduce that 
\begin{align*}
\E[\| U_{n}^{(3)}(f) \|_{\calF_{2}^{0}}]   &\le  C \Bigg \{ \sup_{f \in \calF} \| f^{\odot 2} \|_{P^{2},2} K_{n} n^{-1}+\| F^{\odot 2}\|_{P^{2},q/2}  K_{n}^{2} n^{-3/2+2/q} \\
&\qquad \qquad + \sup_{f \in \calF} \| f^{2} \|_{P^{2},2} K_{n}^{3/2} n^{-3/2} + \| F^{2} \|_{P^{2},q/2} K_{n}^{3} n^{-2+2/q}  \Bigg \} \\
&\quad \le  C \Bigg \{ \sup_{f \in \calF} \| f^{\odot 2} \|_{P^{2},2} K_{n} n^{-1}+  \| F^{\odot 2}\|_{P^{2},q/2} K_{n}^{2} n^{-3/2+2/q} \\
&\qquad \qquad + \sigma_{\frakh} b_{\frakh} K_{n}^{3/2} n^{-3/2} +  b_{\frakh}^{2} K_{n}^{3} n^{-2+2/q}  \Bigg \} 
\end{align*}
where recall that $f^{\odot 2} (x_{1},x_{2}) := f^{\odot 2}_{P}(x_{1},x_{2}) := \int f(x_{1},x) f(x,x_{2}) dP(x)$ for a symmetric measurable function $f$ on $S^{2}$.
For $f \in \calF$, observe that by the Cauchy-Schwarz inequality,
\begin{align*}
\| f^{\odot 2} \|_{P^{2},2}^{2} &= \iint \left ( \int f(x_{1},x) f(x,x_{2}) dP(x) \right )^{2} dP(x_{1})dP(x_{2}) \\
&\le \left  (\iint f^{2}(x_{1},x_{2}) dP(x_{1}) dP(x_{2}) \right )^{2} = \| f \|_{P^{2},2}^{4} \le \sigma_{\frakh}^{4}. 
\end{align*}
On the other hand, $\| F^{\odot 2}\|_{P^{2},q/2} = \nu_{\frakh}^{2}$ by the definition of $\nu_{\frakh}$. Therefore, we  conclude that 
\begin{align*}
&\E \left [ \left \| n^{-1} \sum_{i=1}^{n} \{ S_{n-1,-i}(\delta_{X_{i}}h) \}^{2} \right \|_{\calH} \right ] \\
&\quad \le  C \left \{   \sigma_{\frakh}^{2} K_{n}n^{-1}+ \nu_{\frakh}^{2} K_{n}^{2}n^{-3/2+2/q}  + \sigma_{\frakh} b_{\frakh} K_{n}^{3/2} n^{-3/2}  +  b_{\frakh}^{2} K_{n}^{3}n^{-2+2/q} + \chi_{n}^{2} \right \}.
\end{align*}
This completes the proof. 
\end{proof}

\begin{proof}[Proof of Corollary \ref{cor:bootstrap_validity_vc}]
This follows from the discussion before Theorem \ref{thm:coupling_gaussian_mulitiplier_bootstrap} combined with the anti-concentration inequality (Lemma \ref{lem:AC}), and optimization with respect to $\gamma$. It is without loss of generality to assume that $\eta_{n} \le \overline{\sigma}_{\frakg}^{1/2}$ since otherwise the result is trivial by taking $C$ or  $C'$ large enough, and hence Condition (\ref{eq:growthcondition}) is automatically satisfied.
\end{proof}

\section*{Acknowledgments}
The authors would like to thank the anonymous referees and an Associate Editor for their constructive comments that improve the quality of this paper.

\appendix

\section{Supporting lemmas}
\label{app:supporting_lemmas}

This appendix collects some supporting lemmas that are repeatedly used in the main text. 

\begin{lem}[An anti-concentration inequality for the Gaussian supremum]\label{lem:AC}
Let $(S,\calS,P)$ be a probability space, and let $\calG \subset L^{2}( P )$ be a $P$-pre-Gaussian class of functions.
Denote by $W_{P}$ a tight Gaussian random variable in $\ell^{\infty}(\calG)$ with mean zero and covariance function $\E[ W_{P}(g) W_{P}(g') ] = \Cov_{P}(g,g')$ for all $g,g' \in \calG$ where $\Cov_{P}(\cdot,\cdot)$ denotes the covariance under $P$. Suppose that there exist constants  $\underline{\sigma}, \overline{\sigma}>0$ such that $\underline{\sigma}^{2} \le \Var_{P}(g) \le \overline{\sigma}^{2}$ for all $g \in \calG$.
Then for every $\varepsilon > 0$,
\[
\sup_{t \in\mathbb{R}}\mathbb{P}\left\{\left|\sup_{g \in \calG} W_P(g)-t\right|\le \varepsilon\right\}\le  C_{\sigma}\varepsilon\left\{\E \left[\sup_{g\in\mathcal{G}}W_P (g)\right]+\sqrt{1\vee \log(\underline{\sigma}/\varepsilon)}\right\},
\]
where $C_{\sigma}$ is a constant depending only on $\underline{\sigma}$ and $\overline{\sigma}$.
\end{lem}

\begin{proof}
See Lemma A.1 in \cite{cck2014_empirical_process}.
\end{proof}

\begin{lem}
\label{lem:approximation}
Let $\calF$ be a class of real-valued  measurable functions on a measurable space $(\calX,\calA)$ with finite measurable envelope $F$. Then for any probability measure $R$ on $(\calX,\calA)$ such that $RF^{2} < \infty$, we have 
\[
N(\calF,\| \cdot \|_{R,2}, 4\varepsilon \| F \|_{R,2}) \le \sup_{Q}N(\calF,\| \cdot \|_{Q,2},\varepsilon \| F \|_{Q,2})
\]
for every $0 < \varepsilon \le 1$, where $\sup_{Q}$ is taken over all finitely discrete distributions on $\calX$.
\end{lem}

\begin{proof}
This follows from approximating $R$ by a finitely discrete distribution. See Problem 2.5.1 in \cite{vandervaartwellner1996}. 
\end{proof}

\begin{lem}
\label{lem:entropy_conditional}
Let $(\calX,\calA), (\calY,\calC)$ be measurable spaces and let $\calF$ be a class of real-valued jointly measurable functions on $\calX \times \calY$  with finite measurable envelope $F$. Let $R$ be a probability measure on $(\calY,\calC)$ and for a jointly measurable function $f: \calX \times \calY \to \R$, define $\overline{f}: \calX \to \R$ by $\overline{f}(x) := \int f(x,y) dR(y)$
whenever the latter integral is defined and finite for every $x \in \calX$. Suppose that $\overline{F}$ is everywhere finite and let $\overline{\calF} = \{ \overline{f} : f \in \calF \}$. Then, for every $r,s \in [1,\infty)$, 
\[
\sup_{Q} N(\overline{\calF},\| \cdot \|_{Q,r},2\varepsilon \| \overline{F} \|_{Q,r}) \le \sup_{Q'} N(\calF, \| \cdot \|_{Q',s},\varepsilon^{r} \| F \|_{Q',s}/4) 
\]
where $\sup_{Q}$ and $\sup_{Q'}$ are taken over all finitely discrete distributions on $\calX$ and $\calX \times \calY$, respectively. 
\end{lem}

\begin{proof}
This follows from Lemma A.2 in \cite{ghosalsenvandervaart2000} combined with Lemma \ref{lem:approximation}. 
\end{proof}

If $R=\delta_{y}$ for some $y \in \calY$, then $\| \delta_{y} f \|_{Q,r}^{r} = \| f \|_{Q \times \delta_{y},r}^{r}$ (with $\delta_{y} f(x) = f(x,y)$) and $Q \times \delta_{y}$ is finitely discrete if $Q$ is so. Hence, we have the following corollary. 

\begin{cor}
\label{cor:entropy_conditional}
Under the setting of Lemma \ref{lem:entropy_conditional}, for every $y \in \calY$ and $r \in [1,\infty)$, 
\[
\sup_{Q} N(\delta_{y}F,\| \cdot \|_{Q,r},\varepsilon \| \delta_{y}F \|_{Q,r}) \le \sup_{Q'} N(\calF, \| \cdot \|_{Q',r},\varepsilon \| F \|_{Q',r}).
\]
\end{cor}

\begin{lem}
\label{lem:pointwise_product}
Let $\calF$ and $\calG$ be function classes on a set $\calX$ with finite envelopes $F$ and $G$, respectively. 
If  $\calF \cdot \calG$ stands for the class of pointwise products of functions from $\calF$ and $\calG$, then for any $r \in [1,\infty)$,
\[
\sup_{Q} N(\calF \cdot \calG, \| \cdot \|_{Q,r},2 \varepsilon \| FG \|_{Q,r}) \le \sup_{Q} N(\calF, \| \cdot \|_{Q,r}, \varepsilon \| F\|_{Q,r}) \sup_{Q}N(\calG, \| \cdot \|_{Q,r}, \varepsilon \| G \|_{Q,r}),
\]
where $\sup_{Q}$ is taken over all finitely discrete distributions on $\calX$. 
\end{lem}

\begin{proof}
See Lemma A.1 in \cite{ghosalsenvandervaart2000} or \cite[Section 2.10.3]{vandervaartwellner1996}. 
\end{proof}

\section{Strassen-Dudley theorem and its conditional version}
\label{app:strassen-dudley}

In this appendix, we state the Strassen-Dudley theorem together with its conditional version due to \cite{monradphilipp1991}.
These results play fundamental roles in the proofs of Proposition \ref{prop:entropy_bounds_vc} and Theorem \ref{thm:coupling_gaussian_mulitiplier_bootstrap}. 
In what follows, let $(S,d)$ be a Polish metric space equipped with  its Borel $\sigma$-field $\calB (S)$. 
For any set $A \subset S$ and $\delta > 0$, let $A^{\delta} = \{ x \in S : \inf_{y \in A}  d(x,y) \le \delta \}$. 
We first state the Strassen-Dudley theorem. 

\begin{thm}[Strassen-Dudley]
\label{thm:strassen-dudley}
Let $X$ be an $S$-valued random variable defined on a probability space $(\Omega,\calA,\Prob)$ which admits a uniform random variable on $(0,1)$ independent of $X$. 
Let $\alpha, \beta >0$ be given constants, and let $G$ be a Borel probability measure on $S$ such that $\Prob(X \in A) \le G(A^{\alpha})+ \beta$ for all $A \in \calB(S)$. Then there exists an $S$-valued random variable $Y$  such that $\mathcal{L}(Y) (:= \Prob \circ Y^{-1}) = G$ and $\Prob(d(X,Y) > \alpha) \le \beta$. 
\end{thm}

For a proof of the Strassen-Dudley theorem, we refer to \cite{dudley2002}. Next, we state a conditional version of the Strassen-Dudley theorem due to \cite[Theorem 4]{monradphilipp1991}. 

\begin{thm}[Conditional version of Strassen-Dudley]
\label{thm:conditional_strassen-dudley}
Let $X$ be an $S$-valued random variable defined on a probability space $(\Omega,\calA,\Prob)$, and let $\calG$ be a countably generated sub $\sigma$-field of $\calA$. 
Suppose that there is a uniform random variable on $(0,1)$ independent of $\calG \vee \sigma (X)$, and let $\Omega \times \calB(S) \ni (\omega,A) \mapsto G(A \mid \calG) (\omega)$ be a regular conditional distribution given $\calG$, i.e., for each fixed $A \in \calB(S)$, $G(A \mid \calG)$ is measurable with respect to $\calG$ and for each fixed $\omega \in \Omega$, $G(\cdot \mid \calG)(\omega)$ is a probability measure on $\calB(S)$. If 
\[
\E^{*} \left [ \sup_{A \in \calB(S)} \{ \Prob(X \in A \mid \calG) - G(A^{\alpha} \mid \calG) \} \right ] \le \beta,
\]
then there exists an $S$-valued random variable $Y$ such that the conditional distribution of $Y$ given $\calG$ is identical to $G(\cdot \mid \calG)$, and 
$\Prob( d(X,Y) > \alpha) \le \beta$.
\end{thm}

\begin{rmk}
(i) The map $(\omega,A) \mapsto \Prob(X \in A \mid \calG)(\omega)$ should be understood as a regular conditional distribution (which is guaranteed to exist since $X$ takes values in a Polish space). (ii) $\E^{*}$ denotes the outer expectation. 
\end{rmk}

For completeness, we provide a self-contained proof of Theorem \ref{thm:conditional_strassen-dudley}, since \cite{monradphilipp1991} do not provide its direct proof. 

\begin{proof}[Proof of Theorem \ref{thm:conditional_strassen-dudley}]
Since $\calG$ is countably generated, there exists a real-valued random variable $W$ such that $\calG = \sigma (W)$. For $n=1,2,\dots$ and $k \in \mathbb{Z}$, let $D_{n,k} = \{ k/2^{n} \le W < (k+1)/2^{n} \}$. 
For each $n$, $\{ D_{n,k} : k \in \mathbb{Z} \}$ forms a partition of $\Omega$. Pick any $D$ from $\{ D_{n,k} : n =1,2,\dots; k \in \mathbb{Z} \}$; let $\Prob_{D} = \Prob(\cdot \mid D)$ and $G(\cdot \mid D) = \int G(\cdot \mid \calG) d\Prob_{D}$. 
Then, the Strassen-Dudley theorem yields that there exists an $S$-valued random variable $Y_{D}$ such that $\Prob_{D} \circ Y_{D}^{-1} = G(\cdot \mid D)$ and $\Prob_{D}(d(X,Y_{D}) > \alpha) \le \varepsilon (D) := \sup_{A \in \calB(S)} \{ \Prob_{D}(X \in A) - G(A^{\alpha} \mid D) \}$. 

For each $n=1,2,\dots$, let  $Y_{n} = \sum_{k \in \mathbb{Z}} Y_{D_{n,k}} 1_{D_{n,k}}$, and observe that
\[
\Prob(d(X,Y_{n}) > \alpha) = \sum_{k} \Prob_{D_{n,k}} (d(X,Y_{D_{n,k}}) > \alpha) \Prob(D_{n,k}) \le \sum_{k} \varepsilon (D_{n,k}) \Prob(D_{n,k}).
\]
Let $M$ be any (proper) random variable such that $M \ge \sup_{A \in \calB(S)} \{ \Prob(X \in A \mid \calG) - G(A^{\alpha} \mid \calG) \}$, and observe that 
\[
\Prob_{D}(X \in A) -G(A^{\alpha} \mid D) = \E^{\Prob_{D}} [ \Prob(X \in A \mid \calG) - G(A^{\alpha} \mid \calG) ] \le \E^{\Prob_{D}}[M],
\]
where the notation $\E^{\Prob_{D}}$ denotes the expectation under $\Prob_{D}$. 
So, 
\[
\sum_{k} \varepsilon (D_{n,k}) \Prob(D_{n,k}) \le \sum_{k} \E^{\Prob_{D_{n,k}}} [M] \Prob(D_{n,k}) = \E[M],
\]
and taking infimum with respect to $M$ yields that the left hand side is bounded by $\beta$. 

Next, we shall verify that $\{ \mathcal{L}(Y_{n}) : n \ge 1 \}$ is uniformly tight. In fact, 
\begin{align*}
&\Prob (Y_{n} \in A) = \sum_{k} \Prob(\{ Y_{D_{n,k}} \in A \} \cap D_{n,k}) = \sum_{k} \Prob_{D_{n,k}} (Y_{D_{n,k}} \in A) \Prob(D_{n,k}) \\
&\quad = \sum_{k} G(A \mid D_{n,k}) \Prob(D_{n,k}) = \E[G(A \mid \calG)],
\end{align*}
and since any Borel probability measure on a Polish space is tight by Ulam's theorem, $\{ \mathcal{L}(Y_{n}) : n \ge 1 \}$ is uniformly tight.
This implies that the family of joint laws $\{ \mathcal{L}(X,W,Y_{n}) : n \ge 1 \}$ is uniformly tight and hence has a weakly convergent subsequence by Prohorov's theorem. Let $\mathcal{L}(X,W,Y_{n'}) \stackrel{w}{\to} Q$ (the notation $\stackrel{w}{\to}$ denotes weak convergence), and observe that the marginal law of $Q$ on the ``first two'' coordinates, $S \times \R$, is identical to $\mathcal{L}(X,W)$. 

We shall  verify that there exists an $S$-valued random variable $Y$ such that $\mathcal{L}(X,W,Y) =Q$. 
Since $S$ is polish, there exists a unique regular conditional distribution, $\calB(S) \times (S \times \R) \ni (A,(x,w)) \mapsto Q_{x,w}(A) \in [0,1]$,  for $Q$ given the first two coordinates. By the Borel isomorphism theorem \citep[][Theorem 13.1.1]{dudley2002}, there exists a bijective map $\pi$ from $S$ onto a Borel subset of $\R$ such that  $\pi$ and $\pi^{-1}$ are Borel measurable. Pick and fix any $(x,w) \in S \times \R$, and observe that $Q_{x,w} \circ \pi^{-1}$ extends to a  Borel probability measure on $\R$. 
Denote by $F_{x,w}$ the distribution function of $Q_{x,w} \circ \pi^{-1}$, and let $F_{x,w}^{-1}$ denotes its quantile function. Let $U$ be a uniform random variable on $(0,1)$ (defined on $(\Omega,\calA,\Prob)$) independent of $(X,W)$. Then $F_{x,w}^{-1} (U)$ has law $Q_{x,w} \circ \pi^{-1}$, and hence
$Y = \pi^{-1} \circ F_{X,W}^{-1} (U)$ is the desired random variable.

Now, for any bounded continuous function $f$ on $S$, observe that, whenever $N \ge  n$, $\E[ f(Y_{N})1_{D_{n,k}} ]  = \int_{D_{n,k}} \int f(y) G(dy \mid \calG) d\Prob$, 
which implies that the conditional distribution of $Y$ given $\calG$ is identical to $G( \cdot \mid \calG)$. 
Finally, the Portmanteau theorem yields $\Prob(d(X,Y) > \alpha) \le \liminf_{n'} \Prob(d(X,Y_{n'}) > \alpha) \le \beta$. 
This completes the proof. 
\end{proof}

\section{Additional proofs for the main text}

\subsection{Proof of Lemma \ref{lem:empirical_process}}
\label{app:complement}

We begin with noting that $\calG$ is VC type with characteristics $4\sqrt{A}$ and $2v$ for envelope $G$. 
The rest of the proof is almost the same as that of Theorem 2.1 in \cite{cck2016_empirical_process_coupling} with $B(f) \equiv 0$ (up to adjustments of the notation), but we now allow $q=\infty$. To avoid repetitions, we only point out required modifications. 
In what follows, we will freely use the notation in the proof of \cite[Theorem 2.1]{cck2016_empirical_process_coupling}, but modify $K_{n}$ to $K_{n} = v \log (A \vee n)$, and $C$ refers to a universal constant whose value may vary from place to place. 
In Step 1, change $\varepsilon$ to $\varepsilon =1/n^{1/2}$. For this choice, $\log N(\calF,e_{P},\varepsilon b) \le  C \log (Ab/(\varepsilon b)) = C\log (A/\varepsilon) \le  CK_{n}$, and Dudley's entropy integral bound yields that 
$\E[ \| G_{P} \|_{\calF_{\varepsilon}}] \le  C\varepsilon b \sqrt{\log (Ab/(\varepsilon b))} \le  Cb\sqrt{K_{n}/n}$ (there is a slip in the estimate of $\E[\| G_{P}\|_{\calF_{\varepsilon}}]$ in \cite{cck2016_empirical_process_coupling}, namely, ``$Ab/\varepsilon$'' inside the log should read ``$Ab/(\varepsilon b)$'', which of course does not affect the proof under their definition of $K_{n}$). 
Combining the Borell-Sudakov-Tsirel'son inequality yields that $\Prob \{ \| G_{P}\|_{\calF_{\varepsilon}} > C b\sqrt{K_{n}/n} \} \le 2n^{-1}$.
In Step 3, Corollary \ref{cor:local_max_ineq_dengenerate_uprocess_VCtype} in the present paper (with $r=k=1$) yields that $\E[ \| \G_{n} \|_{\calF_{\varepsilon}}] \le  C(b\sqrt{K_{n}/n} + bK_{n}/n^{1/2-1/q}) \le  CbK_{n}/n^{1/2-1/q}$, which is valid even when $q=\infty$. Then, instead of applying their Lemma 6.1, we apply Markov's inequality to deduce that 
\[
\Prob \left \{  \| \G_{n} \|_{\calF_{\varepsilon}} > CbK_{n}/(\gamma n^{1/2-1/q}) \right \} \le \gamma.
\]
In Step 4, instead of their equation (14), we have
\[
\Prob (Z^{\varepsilon} \in B) \le \Prob (\tilde{Z}^{\varepsilon} \in B^{C_{7}\delta}) + C \left ( \frac{b\sigma^{2}K_{n}^{2}}{\delta^{3}\sqrt{n}} + \frac{M_{n,X}(\delta)K_{n}^{2}}{\delta^{3}\sqrt{n}} + \frac{1}{n} \right ) \quad \forall B \in \calB(\R)
\]
whenever $\delta \ge 2c\sigma^{-1/2}(\log N)^{3/2} \cdot (\log n)$ for some universal constant $c$ ($C_{7}$ comes from their Theorem 3.1 and is universal). 
Finally, in Step 5, take
\[
\delta = C' \left \{ \frac{(b\sigma^{2}K_{n}^{2})^{1/3}}{\gamma^{1/3}n^{1/6}} + \frac{2bK_{n}}{\gamma n^{1/2-1/q}} \right \}
\]
for some large but universal constant $C' > 1$. Under the assumption that $K_{n}^{3} \le  n$, this choice  ensures that $\delta \ge 2c\sigma^{-1/2}(\log N)^{3/2} \cdot (\log n)$, and 
\[
\frac{b\sigma^{2}K_{n}^{2}}{\delta^{3}\sqrt{n}} \le \frac{1}{(C')^{3}n}. 
\]
It remains to bound $M_{n,X}(\delta)$. For finite $q$, their Step 4 shows that 
\[
 \frac{M_{n,X}(\delta)K_{n}^{2}}{\delta^{3}\sqrt{n}}  \le \frac{2^{q}b^{q}K_{n}^{2}(\log N)^{q-3}}{\delta^{q}n^{q/2-1}}.
\]
Since $\log N \le  C''K_{n}$ for some universal constant $C''$, the right hand side is bounded by
\[
\frac{\gamma^{q}(C'')^{q-3}}{(C')^{q}K_{n}}.
\]
Since $K_{n}$ is bounded from below by a universal positive constant (by assumption), and $\gamma \in (0,1)$, by taking $C' > C''$, 
the above term is bounded by $\gamma$ up to a universal constant. 

Now, consider the $q=\infty$ case. 
In that case, $\max_{1 \le j \le N}| \tilde{X}_{1j} | \le 2b$ almost surely and 
$\delta \sqrt{n}/\log N \ge 2C'b/(C''\gamma) > 2b$ provided that $C' > C''$. Hence $M_{n,X}(\delta) =0$ in that case. 
These modifications lead to the desired conclusion. \qed

\subsection{Proofs for Section \ref{sec:monotonicity_testing}}

We first prove Theorem \ref{thm:app_local_uproc_bootstrap_kolmogorov_distance} and Corollary \ref{cor:app_local_uproc_bootstrap}, and then prove Lemma \ref{lem:rate_normalizing_constant_jackknife_estimate} and Theorem \ref{thm:app_local_uproc_bootstrap_kolmogorov_distance_uniform_bandwidth}.

\begin{proof}[Proof of Theorem \ref{thm:app_local_uproc_bootstrap_kolmogorov_distance}]

In what follows, the notation $\lesssim$ signifies that the left hand side is bounded by the right hand side up to a constant that depends only on $r,m,\zeta,c_1,c_2,C_1,L$. We also write $a \simeq b$ if $a \lesssim b$ and $b \lesssim a$. 
In addition, let $c,C,C'$ denote generic constants depending only on $r, m,\zeta, c_{1},c_{2}, C_{1}, L$; their values may vary from place to place. We divide the rest of the proof into three steps.

\uline{Step 1}. Let 
\[
S_{n}^{\sharp} := \sup_{\vartheta \in \Theta} \frac{b_{n}^{m/2}}{c_{n}(\vartheta)\sqrt{n}}\sum_{i=1}^n \xi_{i} \left[ U_{n-1,-i}^{(r-1)} (\delta_{D_{i}} h_{n,\vartheta})- U_n(h_{n,\vartheta}) \right].
\]
In this step, we shall show that the result (\ref{eqn:app_local_uproc_bootstrap_kolmogorov_distance}) holds with $\hat{S}_{n}$ and $\hat{S}_{n}^{\sharp}$ replaced by $S_{n}$ and $S_{n}^{\sharp}$, respectively. 

We first verify Conditions (PM), (VC), (MT), and~\eqref{eqn:lower_bound_variance_condition} for the function class
\[
\calH_{n} = \left \{ b_{n}^{m/2} c_{n}(\vartheta)^{-1} h_{n,\vartheta} : \vartheta \in \Theta \right \}
\]
with a symmetric envelope 
\[
H_{n}(d_{1:r}) = b_{n}^{-(r-1/2)m} c_{1}^{-1} \| L \|_{\R^{m}}^{r} \overline{\varphi}(v_{1:r}) \prod_{i=1}^{r} 1_{\calX^{\zeta/2}}(x_{i}) \prod_{1 \le  i < j \le  r} 1_{[-2,2]^m}(b_{n}^{-1}(x_{i}-x_{j})).
\]
Condition (PM) follows from our assumption. For Condition (VC), that $\calH_{n}$ is VC type with characteristics $(A', v')$ satisfying  $\log A' \lesssim \log n$ and $v' \lesssim 1$ follows from a slight modification of the proof of Lemma 3.1 in \cite{ghosalsenvandervaart2000}. The latter part follows from our assumption.
Condition (VC) guarantees the existence of a tight Gaussian random variable $\mathcal{W}_{P,n}(g), g \in P^{r-1}\calH_{n} =: \calG_{n}$ in $\ell^{\infty}(\calG_{n})$ with mean zero and covariance function $\E[\mathcal{W}_{P,n}(g)\mathcal{W}_{P,n}(g')] = \Cov_{P}(g,g')$ for $g,g' \in \calG_{n}$. Let $W_{P,n} (\vartheta) = \mathcal{W}_{P,n}(g_{n,\vartheta})$ for $\vartheta \in \Theta$ where $g_{n,\vartheta} = b_{n}^{m/2} c_{n}(\vartheta)^{-1} P^{r-1}h_{n,\vartheta}$. It is seen that $W_{P,n}(\vartheta), \vartheta \in \Theta$ is a tight Gaussian random variable in $\ell^{\infty}(\Theta)$ with mean zero and covariance function (\ref{eq:app_covariance_function}). 

Next, we determine the values of parameters $\underline{\sigma}_{\frakg}, \overline{\sigma}_{\frakg}, b_{\frakg}, \sigma_{\frakh}, b_{\frakh}, \chi_{n},\nu_{\frakh}$ for the function class $\calH_n$. We will show in Step 3 that we may choose
\begin{equation}
\label{eq:moment_choice}
\underline{\sigma}_{\frakg} \simeq 1, \ \overline{\sigma}_{\frakg} \simeq 1, \ b_{\frakg} \simeq b_{n}^{-m/2}, \ \sigma_{\frakh} \simeq b_{n}^{-m/2}, \ b_{\frakh} \simeq b_{n}^{-3m/2},
\end{equation}
and bound $\nu_{\frakh}$ and $\chi_{n}$ as 
\begin{equation}
\label{eq:moment_bound}
\nu_{\frakh} \lesssim b_{n}^{-m(1-1/q)}, \ \chi_{n} \lesssim (\log n)^{3/2}/(nb_{n}^{3m/2}).
\end{equation}
Given these choices and bounds,  Corollaries \ref{cor:kolmogorov_distance_gaussian_coupling_vc} and \ref{cor:bootstrap_validity_vc} yield that 
\begin{equation}
\label{eq:intermediate_bound}
\begin{split}
&\sup_{t \in \R} \left | \Prob (S_{n} \le t) - \Prob (\tilde{S}_{n} \le t) \right | \le Cn^{-c} \ \text{and} \\
&\Prob \left \{ \sup_{t \in \R} \left | \Prob_{\mid D_{1}^{n}} (S_{n}^{\sharp} \le t) - \Prob (\tilde{S}_{n} \le t) \right | > Cn^{-c}\right \} \le Cn^{-c}.
\end{split}
\end{equation}

\uline{Step 2}. Observe that 
\begin{equation}
\label{eq:intermediate_bound3}
| \hat{S}_{n} - S_{n} | \le \sup_{\vartheta \in \Theta} \left | \frac{c_{n}(\vartheta)}{\hat{c}_{n}(\vartheta)} - 1 \right | \| \sqrt{n}U_{n} \|_{\calH_{n}} \quad \text{and} \quad
| \hat{S}_{n}^{\sharp}- S_{n}^{\sharp} | \le  \sup_{\vartheta \in \Theta} \left | \frac{c_{n}(\vartheta)}{\hat{c}_{n}(\vartheta)} - 1 \right | \| \U_{n}^{\sharp} \|_{\calH_{n}}.
\end{equation}
We shall bound $\sup_{\vartheta \in \Theta} | c_{n}(\vartheta)/\hat{c}_{n}(\vartheta) - 1|, \|\sqrt{n}U_{n} \|_{\calH_{n}}$, and $\| \U_{n}^{\sharp} \|_{\calH_{n}}$. 

Choose $n_{0}$ by the smallest $n$ such that $C_{1}n^{-c_{2}} \le 1/2$; it is clear that $n_{0}$ depends only on $c_{2}$ and $C_{1}$. It suffices to prove (\ref{eqn:app_local_uproc_bootstrap_kolmogorov_distance}) for  $n \ge n_{0}$, since for $n < n_{0}$, the result (\ref{eqn:app_local_uproc_bootstrap_kolmogorov_distance}) becomes trivial by taking $C$ sufficiently large. 
So let $n \ge n_{0}$. Then Condition (T8) ensures that with probability at least $1-C_{1}n^{-c_{2}}$, $\inf_{\vartheta \in \Theta} \hat{c}_{n}(\vartheta)/c_{n}(\vartheta) \ge 1/2$. 
Since $| a^{-1} - 1 | \le 2 | a - 1 |$ for $a \ge 1/2$, Condition (T8) also ensures that 
\begin{equation}
\label{eq:intermediate_bound0}
\Prob \left \{ \sup_{\vartheta \in \Theta} \left | \frac{c_{n}(\vartheta)}{\hat{c}_{n}(\vartheta)} - 1 \right | > Cn^{-c} \right \} \le Cn^{-c}.
\end{equation}

Next, we shall bound $\|\sqrt{n}U_{n} \|_{\calH_{n}}$ and $\| \U_{n}^{\sharp} \|_{\calH_{n}}$.
Given (\ref{eq:moment_choice}) and (\ref{eq:moment_bound}), and in view of the fact that the covering number of $\calH_{n} \cup (-\calH_{n}) := \{ h,-h : h \in \calH_{n} \}$ is at most twice that of $\calH_{n}$,  applying Corollaries \ref{cor:kolmogorov_distance_gaussian_coupling_vc} and \ref{cor:bootstrap_validity_vc} to the function class $\calH_{n} \cup (-\calH_{n})$, we deduce that
\begin{align*}
&\sup_{t \in \R} \left | \Prob (\| \sqrt{n}U_{n} \|_{\calH_{n}} \le t) - \Prob (\| \mathcal{W}_{P,n} \|_{\calG_{n}} \le t) \right | \le Cn^{-c} \ \text{and} \\
&\Prob \left \{ \sup_{t \in \R} \left | \Prob_{\mid D_{1}^{n}} (\| \U_{n}^{\sharp} \|_{\calH_{n}} \le t) -\Prob (\| \mathcal{W}_{P,n} \|_{\calG_{n}} \le t) \right | > Cn^{-c}\right \} \le Cn^{-c}.
\end{align*}
(Theorem 3.7.28 in \cite{ginenickl2016} ensures that the Gaussian process $\mathcal{W}_{P,n}$ extends to the symmetric convex hull of $\calG_{n}$ in such a way that $\mathcal{W}_{P,n}$ has linear, bounded, and uniformly continuous (with respect to the intrinsic pseudometric) sample paths; in particular, $\{ \mathcal{W}_{P,n}(g) : g \in \calG_{n} \cup (-\calG_{n}) \}$ is a tight Gaussian random variable in $\ell^{\infty}(\calG_{n} \cup (-\calG_{n}))$ with mean zero and covariance function $\E[\mathcal{W}_{P,n}(g)\mathcal{W}_{P,n}(g')] = \Cov_{P}(g,g')$ for $g,g' \in \calG_{n} \cup (-\calG_{n})$ and $\sup_{g \in \calG_{n} \cup (-\calG_{n})} \mathcal{W}_{n}(g) = \| \mathcal{W}_{P,n} \|_{\calG_{n}}$.)
Dudley's entropy integral bound and the Borell-Sudakov-Tsirel'son inequality yield that 
$\Prob \{ \| \mathcal{W}_{P,n} \|_{\calG_{n}} > C(\log n)^{1/2} \} \le 2n^{-1}$, 
so that 
\begin{equation}
\label{eq:intermediate_bound2}
\begin{split}
&\Prob \{ \| \sqrt{n}U_{n} \|_{\calH_{n}} > C(\log n)^{1/2} \} \le Cn^{-c} \ \text{and} \\
&\Prob \left \{  \Prob_{\mid D_{1}^{n}} \{ \| \U_{n}^{\sharp} \|_{\calH_{n}} > C (\log n)^{1/2} \} > Cn^{-c}\right \} \le Cn^{-c}.
\end{split}
\end{equation}

Now, the desired result (\ref{eqn:app_local_uproc_bootstrap_kolmogorov_distance})  follows from combining (\ref{eq:intermediate_bound})--(\ref{eq:intermediate_bound2}) and the anti-concentration inequality (Lemma \ref{lem:AC}). In fact, the anti-concentration inequality yields 
\begin{equation}
\label{eq:app_AC}
\sup_{t \in \R} \Prob ( |\tilde{S}_{n} -t| \le Cn^{-c} ) \le C'n^{-c} (\log n)^{1/2}.
\end{equation}
Hence, combining the bounds (\ref{eq:intermediate_bound})--(\ref{eq:intermediate_bound2}) and (\ref{eq:app_AC}), we have  for every $t \in \R$, 
\begin{align*}
\Prob (\hat{S}_{n} \le t ) &\le \Prob (S_{n} \le t + Cn^{-c}) + Cn^{-c} \\
&\le \Prob (\tilde{S}_{n} \le t+Cn^{-c}) + Cn^{-c} \\
&\le \Prob (\tilde{S}_{n} \le t) + Cn^{-c},
\end{align*}
and likewise $\Prob (\hat{S}_{n} \le t) \ge \Prob(\tilde{S}_{n} \le t) - Cn^{-c}$. Similarly, we have  
\[
\Prob \left \{ \sup_{t \in \R} \left | \Prob_{\mid D_{1}^{n}} (\hat{S}_{n}^{\sharp} \le t) - \Prob (\tilde{S}_{n} \le t) \right | > Cn^{-c}\right \} \le Cn^{-c}.
\]

\uline{Step 3}. It remains to verify (\ref{eq:moment_choice}) and (\ref{eq:moment_bound}). 
First, that we may choose $\underline{\sigma}_{\frakg} \simeq 1$ follows from Conditions (T6) and (T7). For $\varphi \in \Phi$ and $k=1,\dots,r-1$, let
\[
\varphi_{[r-k]}(v_{1:k},x_{k+1:r}) = \E[ \varphi (v_{1:k}, V_{k+1:r}) \mid X_{k+1:r} = x_{k+1:r}] \prod_{j=k+1}^{r}p(x_{j}),
\]
and define $\overline{\varphi}_{[r-k]}$ similarly. 
Then, for $k=1,\dots,r$, 
\[
(P^{r-k}h_{n,\vartheta}) (d_{1:k}) =\left ( \prod_{j=1}^{k} L_{b_{n}}(x-x_{j}) \right )  \int_{[-1,1]^{m(r-k)}}\varphi_{[r-k]}(v_{1:k},x-b_{n} x_{k+1:r}) \left ( \prod_{j=k+1}^{r}L(x_{j}) \right ) dx_{k+1:r},
\]
where $x-b_{n}x_{k+1:r} = (x-b_{n}x_{k+1},\dots,x-b_{n}x_{r})$. Likewise, we have 
\begin{align*}
(P^{r-k}H_{n}) (d_{1:k})
&\lesssim  b_{n}^{-(k-1/2)m}  \left ( \prod_{i=1}^{k} 1_{\calX^{\zeta/2}}(x_{i}) \right )\left ( \prod_{1 \le  i < j \le k} 1_{[-2,2]^m}(b_{n}^{-1}(x_{i}-x_{j})) \right ) \\
&\qquad \times \int_{[-2,2]^{m (r-k)}} \overline{\varphi}_{[r-k]} (v_{1:k},x_{1}-b_{n}x_{k+1:r}) dx_{k+1:r}.
\end{align*}
Suppose first that $q$ is finite and let $\ell \in [2,q]$. Observe that by Jensen's inequality, 
\[
\begin{split}
\| P^{r-k}h_{n,\vartheta} \|_{P^k,\ell}^{\ell} &\le C^{\ell} b_n^{-(\ell-1)mk} \int_{[-1,1]^{mr}} \E \left[\overline{\varphi}^{\ell}(V_{1:r}) \mid X_{1:r} = x-b_{n}x_{1:r} \right] \left( \prod_{j=1}^k p(x-b_n x_{j}) \right) d {x_{1:r}} \\
&\le C^{\ell} b_{n}^{-(\ell-1)mk}  \int_{[-1,1]^{mr}} \E \left[\overline{\varphi}^{\ell}(V_{1:r}) \mid X_{1:r}=x-b_{n}x_{1:r} \right] dx_{1:r}
\le C^{\ell} b_{n}^{-(\ell-1)mk},
\end{split}
\]
so that $\sup_{h \in \calH_n} \| P^{r-k}h \|_{P^k,\ell} \lesssim b_n^{-m[(k-1/2)-k/\ell]}$. Hence, we may choose  $\overline{\sigma}_\frakg \simeq 1$ and $\sigma_\frakh \simeq b_n^{-m/2}$. Similarly, Jensen's inequality and the symmetry of $\overline{\varphi}$ yield that 
\begin{align*}
&\| P^{r-k} H_n \|_{P^k,\ell}^\ell  \le C^{\ell} b_n^{-(k-1/2)m\ell+m(k-1)} \\
&\quad \times \int_{\calX^{\zeta/2} \times [-2,2]^{m(r-1)}} \E\left[ \overline{\varphi}^{\ell}(V_{1:r}) \mid X_1 = x_1, X_{2:r} = x_1-b_n x_{2:j} \right] p(x_1) \prod_{j=2}^{k} p(x_1 - b_n x_j) d x_{1:r} \\
&\quad \le C^{\ell}  b_n^{-(k-1/2)m\ell+m(k-1)}\int_{\calX^{\zeta/2} \times [-2,2]^{m(r-1)}} \E\left[ \overline{\varphi}^{\ell}(V_{1:r}) \mid X_1 = x_1, X_{2:r} = x_1-b_n x_{2:j} \right] d x_{1:r} \\
&\quad \le C^{\ell} b_n^{-(k-1/2)m\ell+m(k-1)},
\end{align*}
so that $\| P^{r-k} H_n \|_{P^k,\ell} \lesssim b_n^{-m[(1-1/\ell)k - (1/2-1/\ell)]}$. Hence, we may choose $b_\frakg \simeq b_n^{-m/2}$, $b_\frakh \simeq b_n^{-3m/2}$, and bound $\chi_{n}$ as 
\[
\chi_n \lesssim \sum_{k=3}^r  n^{-(k-1)/2} (\log{n})^{k/2} b_n^{-mk/2} \lesssim {(\log n)^{3/2}\over n b_n^{3m/2}}.
\]
Similar calculations yield that
\begin{align*}
\| (P^{r-2}H_{n})^{\odot 2} \|^{q/2}_{P^{2},q/2} &\le C^{q} b_n^{-m(q-1)} \int_{\calX^{\zeta/2} \times [-2,2]^{m(r-1)}} \E\left[ \overline{\varphi}^{q}(V_{1:r}) \mid X_1 = x_1, X_{2:r} = x_1-b_n x_{2:j} \right] d x_{1:r} \\
&\le C^{q} b_{n}^{-m(q-1)}. 
\end{align*}
Hence,  $\nu_{\frakh} \lesssim b_n^{-m(1-1/q)}$. 

It is not difficult to verify that  (\ref{eq:moment_choice}) and (\ref{eq:moment_bound}) hold in the $q=\infty$ case as well under the convention that $1/q=0$ for $q=\infty$. 
This completes the proof. 
\end{proof}

\begin{proof}[Proof of Corollary \ref{cor:app_local_uproc_bootstrap}]
Let $\eta_{n} := Cn^{-c}$ where the constants $c,C$ are those given in Theorem \ref{thm:app_local_uproc_bootstrap_kolmogorov_distance}. 
Denote by  $q_{\tilde{S}_{n}}(\alpha)$ the $\alpha$-quantile of $\tilde{S}_{n}$. Define the event
\[
 \mathcal{E}_{n}: =\left \{ \sup_{t \in \R} \left | \Prob_{\mid D_{1}^{n}} (\hat{S}_{n}^{\sharp} \le t) - \Prob (\tilde{S}_{n} \le t) \right | \le \eta_{n} \right \},
\]
whose probability is at least $1-\eta_{n}$. 
On this event, 
\begin{align*}
\Prob_{\mid D_{1}^{n}} \left \{ \hat{S}_{n}^{\sharp} \le q_{\tilde{S}_{n}}(\alpha+\eta_{n}) \right \} &\ge \Prob \left \{\tilde{S}_{n}\le q_{\tilde{S}_{n}}(\alpha+\eta_{n}) \right \} - \eta_{n} \\
&= \alpha+\eta_{n} - \eta_{n} = \alpha,
\end{align*}
where the second equality follows from the fact that the distribution function of $\tilde{S}_{n}$ is continuous (cf. Lemma \ref{lem:AC}). This shows that the inequality $q_{\hat{S}_{n}^{\sharp}}(\alpha) \le q_{\tilde{S}_{n}}(\alpha+\eta_{n})$
holds on the event $\mathcal{E}_{n}$, so that 
\begin{align*}
\Prob \left \{ \hat{S}_{n} \le q_{\hat{S}_{n}^{\sharp}}(\alpha) \right \} &\le \Prob \left \{ \hat{S}_{n} \le q_{\tilde{S}_{n}}(\alpha+\eta_{n}) \right \} + \Prob ( \mathcal{E}_{n}^{c}) \\
&\le  \Prob \left \{ \tilde{S}_{n} \le q_{\tilde{S}_{n}}(\alpha+\eta_{n}) \right  \} + 2\eta_{n} \\
&= \alpha + 3\eta_{n}. 
\end{align*}
The above discussion presumes that $\alpha + \eta_{n} < 1$, but if $\alpha + \eta_{n} \ge 1$, then the last inequality is trivial. Likewise, we have $\Prob \left \{ \hat{S}_{n} \le q_{\hat{S}_{n}^{\sharp}}(\alpha) \right \} \ge \alpha-3\eta_{n}$. 
This completes the proof.
\end{proof}

\begin{proof}[Proof of Lemma \ref{lem:rate_normalizing_constant_jackknife_estimate}]
We begin with noting that 
\begin{align*}
\left| \frac{\hat{c}_{n}(\vartheta)}{c_{n}(\vartheta)} - 1 \right | \le \left| \frac{\hat{c}_{n}^2(\vartheta)}{c_{n}^2(\vartheta)} - 1 \right |  &\le \frac{1}{n} \sum_{i=1}^{n} \left [  \{ U_{n-1,-i}^{(r-1)}(\delta_{D_{i}}\breve{h}_{n,\vartheta}) - U_{n}(\breve{h}_{n,\vartheta}) \}^2 - 1 \right ],
\end{align*}
where $\breve{h}_{n,\vartheta} = b_{n}^{m/2}c_{n}(\vartheta)^{-1} h_{n,\vartheta}$. We note that  $\Var_{P}(P^{r-1}\breve{h}_{n,\vartheta}) =1$ by the definition of $c_{n}(\vartheta)$. 
Recall from the proof of Theorem \ref{thm:app_local_uproc_bootstrap_kolmogorov_distance} that the function class $\calH_{n} =\{ \breve{h}_{n,\vartheta} : \vartheta \in \Theta \}$ is VC type with characteristics  $(A', v')$ satisfying $\log A' \lesssim \log n$ and $v' \lesssim 1$ for envelope $H_{n}$.
Now, from Step 5 in the proof of Theorem \ref{thm:coupling_gaussian_mulitiplier_bootstrap} applied with $\calH = \calH_{n}$, we have for every $\gamma \in (0,1)$, with probability at least $1-\gamma-n^{-1}$, 
\begin{align*}
&\left \| \frac{1}{n} \sum_{i=1}^{n} \left [  \{ U_{n-1,-i}^{(r-1)}(\delta_{D_{i}}h) - U_{n}(h) \}^2 - 1 \right ] \right \|_{\calH_{n}}  \\
&\le  C\gamma^{-1} \Bigg [ (b_{\frakg} \vee \sigma_{\frakh})\overline{\sigma}_{\frakg}K_n^{1/2} n^{-1/2}  + 
b_{\frakg}^{2} K_{n}n^{-1+2/q} \\
&\qquad+ \overline{\sigma}_{\frakg} \left \{  \nu_{\frakh} K_{n} n^{-3/4+1/q}  + (\sigma_{\frakh} b_{\frakh})^{1/2} K_{n}^{3/4} n^{-3/4}  +  b_{\frakh} K_{n}^{3/2}n^{-1+1/q}+ \chi_{n} \right \}\Bigg ]
\end{align*}
for some constant $C$ depending only on $r$. The desired result follows from the choices of parameters $\overline{\sigma}_{\frakg}, b_{\frakg}, \sigma_{\frakh}, b_{\frakh}, \chi_{n}$, and $\nu_{\frakh}$ given in the proof of Theorem \ref{thm:app_local_uproc_bootstrap_kolmogorov_distance} together with choosing $\gamma = n^{-c}$ for some  constant $c$ sufficiently small but depending only on $r, m, \zeta, c_{1},c_{2}, C_{1}, L$. 
\end{proof}

\begin{proof}[Proof of Theorem \ref{thm:app_local_uproc_bootstrap_kolmogorov_distance_uniform_bandwidth}]
The proof follows from similar arguments to those in the proof of Theorem \ref{thm:app_local_uproc_bootstrap_kolmogorov_distance}, so we only highlight the differences. Define the function class
\[
\calH_{n} = \left \{ b^{m/2} c_{n}(\vartheta,b)^{-1} h_{\vartheta,b} : \vartheta \in \Theta, b \in \calB_{n} \right \}
\]
with a symmetric envelope 
\[
H_{n}(d_{1:r}) = \ub_{n}^{-(r-1/2)m} c_{1}^{-1} \| L \|_{\R^{m}}^{r} \overline{\varphi}(v_{1:r}) \prod_{i=1}^{r} 1_{\calX^{\zeta/2}}(x_{i}) \prod_{1 \le  i < j \le  r} 1_{[-2,2]^m}(\overline{b}_{n}^{-1}(x_{i}-x_{j})).
\]
Recall that we assume $q=\infty$ in this theorem. 
In view of the calculations in the proof of Theorem \ref{thm:app_local_uproc_bootstrap_kolmogorov_distance}, we may choose 
\[
\underline{\sigma}_{\frakg} \simeq 1, \ \overline{\sigma}_{\frakg} \simeq 1, \ b_{\frakg} \simeq \kappa_{n}^{m(r-1)} \ub_{n}^{-m/2}, \ \sigma_{\frakh} \simeq \ub_{n}^{-m/2}, \ b_{\frakh} \simeq \kappa_{n}^{m(r-2)} \ub_{n}^{-3m/2},
\]
and bound $\nu_{\frakh}$ and $\chi_{n}$ as 
\[
\nu_{\frakh} \lesssim \kappa_{n}^{m/2} \ub_{n}^{-m}, \ \chi_{n} \lesssim {\kappa_{n}^{m(r-2)} (\log n)^{3/2} \over n \ub_{n}^{3m/2}}.
\]
Given these choices and bounds, the conclusion of the theorem follows from repeating the proof of Theorem \ref{thm:app_local_uproc_bootstrap_kolmogorov_distance}. 
\end{proof}

\section{Conditional UCLT for JMB}
\label{sec:conditional_UCLT_JMB}

In this section we prove the conditional UCLT for the JMB when the function class $\calH$ and the distribution $P$ are independent of $n$ under a metric entropy condition. We obey the notation used in Sections \ref{sec:gaussian_approx} and \ref{sec:bootstrap} but since we consider a limit theorem we assume that the probability space is $(\Omega,\calA,\Prob) = (S^{\mathbb{N}},\calS^{\mathbb{N}},P^{\mathbb{N}}) \times (\Xi, \calC, R)$ and $X_{1},X_{2},\dots$ are the coordinate projections of $(S^{\mathbb{N}},\calS^{\mathbb{N}},P^{\mathbb{N}})$. 
To formulate the conditional UCLT, recall that weak convergence in $\ell^{\infty}(\calH)$ is ``metrized'' by the bounded Lipschitz distance: for arbitrary maps $\mathbb{X}_{n}: \Omega \to \ell^{\infty}(\calH)$ and a tight Borel measurable map $\mathbb{X}: \Omega \to \ell^{\infty}(\calH)$, $\mathbb{X}_{n}$ converge weakly to $\mathbb{X}$ if and only if 
\[
d_{BL}(\mathbb{X}_{n},\mathbb{X}) := \sup_{f \in BL_{1}} | \E^{*}[f(\mathbb{X}_{n})] - \E[f(\mathbb{X})]| \to 0,
\]
where $BL_{1} = \{ f : \ell^{\infty}(\calH) \to \R: |f| \le 1, |f(x)-f(y)| \le \| x-y \|_{\calH} \ \forall x,y \in \ell^{\infty}(\calH) \}$; see \cite[][p.73]{vandervaartwellner1996}. If the function class $\calG = P^{r-1} \calH = \{ P^{r-1} h : h \in \calH \}$ is $P$-pre-Gaussian, then there exists a tight Gaussian random variable $W_{P}$ in $\ell^{\infty}(\calG)$ with mean zero and covariance function $\E[W_{P}(g)W_{P}(g')] = \Cov_{P} (g,g')$. Set $\mathbb{W}_{P} (h) = W_{P} \circ P^{r-1} (h)$, which is a tight Gaussian random variable in $\ell^{\infty}(\calH)$ with mean zero and covariance function $\E[\mathbb{W}_{P} (h)\mathbb{W}_{P}(h')] = \Cov_{P}(P^{r-1}h,P^{r-1}h')$. 
We will show that conditionally on $X_{1}^{\infty} = \{ X_{1},X_{2},\dots \}$, $\U_{n}^{\sharp}$ converges weakly to $\mathbb{W}_{P}$ in probability in the sense that 
\[
d_{BL \mid X_{1}^{\infty}} (\U_{n}^{\sharp}, \mathbb{W}_{P}):= \sup_{f \in BL_{1}} | \E_{\mid X_{1}^{\infty}} [f(\U^{\sharp}_{n})] - \E[f(\mathbb{W}_{P})]| 
\]
converges to zero in outer probability under regularity conditions ($\E_{\mid X_{1}^{\infty}}$ denotes the conditional expectation given $X_{1}^{\infty}$). Since the map $(\xi_{1},\dots,\xi_{n}) \mapsto n^{-1/2} \sum_{i=1}^n \xi_{i}[  U_{n-1,-i}^{(r-1)} (\delta_{X_{i}}\cdot) - U_n(\cdot) ]$ is continuous from $\R^{n}$ into $\ell^{\infty}(\calH)$, the multiplier process $\U_{n}^{\sharp}$ induces a Borel measurable map into $\ell^{\infty}(\calH)$ for fixed $X_{1}^{\infty}$. 
For an arbitrary map $Y: \Omega \to \R$, let $Y^{*}$ denote the measurable cover \cite[lemma 1.2.1]{vandervaartwellner1996}. 

\begin{thm}[Conditional UCLT for JMB]
\label{thm:bootstrap_UCLT}
Let  $\calH$ be a fixed pointwise measurable class of symmetric measurable functions on $S^{r}$  with  symmetric  envelope  $H \in L^{2}(P^{r})$ such that $\int_{0}^{1} \sqrt{\lambda (\varepsilon)} d\varepsilon < \infty$ with $\lambda (\varepsilon) = \sup_{Q} \log N(\calH,\| \cdot \|_{Q,2},\varepsilon \| H \|_{Q,2})$. Then $\calG = P^{r-1}\calH = \{ P^{r-1} h : h \in \calH \}$ is $P$-pre-Gaussian, $d_{BL}(\U_{n}/r,\mathbb{W}_{P}) \to 0$,  and $d_{BL \mid X_{1}^{\infty}}(\U_{n}^{\sharp},\mathbb{W}_{P})^{*} \stackrel{\Prob}{\to} 0$ as $n \to \infty$. 
\end{thm}

Theorem \ref{thm:bootstrap_UCLT} should be compared with Theorem 2.1 in \cite{arconesgine1994} that establishes a conditional UCLT for the empirical bootstrap for a non-degenerate $U$-process under the same metric entropy condition. Interestingly, however, our moment condition on the envelope $H$ is weaker than their condition (2.3), which, if $r=2$, requires $\E[H(X_1,X_1)]<\infty$ in addition to $\E[H^{2}(X_1,X_2)] < \infty$. This comes from the difference in how to estimate the Haj\'{e}k projection; our JMB estimates the Haj\'{e}k projection by a jackknife $U$-statistic, while the empirical bootstrap estimates it by  a $V$-statistic (see Remark \ref{rem:other_bootstraps}). 

If we are interested in $\sup_{h \in \calH} \U_{n}(h)/r$, then the result of Theorem \ref{thm:bootstrap_UCLT} implies that 
\[
\begin{split}
&\sup_{t \in \R} \left| \Prob\left (\sup_{h \in \calH} \U_{n}(h)/r \le t \right) - \Prob\left(\sup_{g \in \calG} W_{P} (g) \le t \right) \right| \to 0 \quad \text{and} \\
&\sup_{t \in \R} \left| \Prob_{\mid X_{1}^{\infty}}\left (\sup_{h \in \calH} \U_n^{\sharp} (h) \le t \right) - \Prob\left(\sup_{g \in \calG} W_{P} (g)\le t\right ) \right| \stackrel{\Prob}{\to} 0
\end{split}
\]
as long as the distribution function of $\sup_{g \in \calG} W_{P}(g)$ is continuous, which is true if $\inf_{g \in \calG} \Var_{P}(g) > 0$ (cf. Lemma \ref{lem:AC}).  When the function class $\calH$ is centrally symmetric (i.e., $-h \in \calH$ whenever $h \in \calH$) so that $\sup_{h \in \calH}\U_{n}(h) = \| \U_{n} \|_{\calH}$, $\sup_{g \in \calG}W_{P}(g) = \| W_{P} \|_{\calG}$, and $\sup_{h \in \calH}\U_{n}^{\sharp}(h) = \| \U_{n}^{\sharp} \|_{\calH}$, then the distribution function of $\| W_{P} \|_{\calG}$ is continuous under a much less restrictive assumption that $\Var_{P}(g) > 0$ for some $g \in \calG$. 
Indeed, from Theorem 11.1 in \cite{Davydov1998}, the distribution of $\| W_{P} \|_{\calG}$ is (absolutely) continuous on $(\ell_{0},\infty)$ with $\ell_{0} \ge 0$ being the left endpoint of the support of $\| W_{P} \|_{\calG}$, but from \cite[p.57-58]{ledouxtalagrand1991}, $\ell_{0} = 0$. This implies that, unless $\| W_{P} \|_{\calG} = 0$ almost surely, the distribution function of $\| W_{P} \|_{\calG}$ does not have a jump at $\ell_{0} = 0$ (as $\Prob (\| W_{P} \|_{\calG} = 0) = 0$) and so is everywhere continuous on $\R$.

\begin{proof}[Proof of Theorem~\ref{thm:bootstrap_UCLT}]
The first two results are essentially implied by the proof of Theorem 4.9 in \cite{arconesgine1993} but we include their proofs for completeness. By changing $H$ to $H \vee 1$ if necessary, we may assume $\| G \|_{P,2} > 0$ (recall $G=P^{r-1}H$), which implies $\| H \|_{P,2} > 0$.  By Jensen's inequality, $\| P^{r-1}h \|_{P,2} \le \| h \|_{P^{r},2}$ and so we have 
\[
N(\calG,\| \cdot \|_{P,2},\tau \| H \|_{P^{r},2}) \le N(\calH, \| \cdot \|_{P^{r},2}, \tau \| H \|_{P^{r},2}).
\]
The right hand side is bounded by $\sup_{Q}N(\calH,\| \cdot \|_{Q,2},\tau \| H \|_{Q,2}/4)$ by Lemma \ref{lem:approximation}. Conclude that 
\[
\int_{0}^{1} \sqrt{\log N(\calG,\| \cdot \|_{P,2},\tau \| H \|_{P^{r},2})} d\tau < \infty,
\]
which implies by Dudley's criterion for sample continuity that $\calG$ is $P$-pre-Gaussian (to be precise we have to verify $\int_{0}^{1} \sqrt{\log N(\{ g-Pg : g \in \calG \},\| \cdot \|_{P,2},\tau)} d\tau < \infty$ but this is immediate). The convergence of marginals of $\U_{n}/r$ to $\mathbb{W}_{P}$ follows from the multidimensional CLT for $U$-statistics. To conclude $d_{BL}(\U_{n}/r,\mathbb{W}_{P}) \to 0$, it suffices to show the asymptotic equicontinuity condition
\begin{equation}
\label{eq:AEC}
\lim_{\delta \downarrow 0} \limsup_{n \to \infty} \Prob \left ( \sup_{\| h-h' \|_{P^{r},2} < \delta \| H \|_{P^{r},2}} | \U_{n} (h-h') | > \eta \right ) = 0
\end{equation}
holds for every $\eta > 0$. We defer the proof of (\ref{eq:AEC}) after the proof of the theorem. 

 To prove the last result of the theorem, let $e_{P} (h,h') = \| P^{r-1}(h-h') \|_{P,2}$ and for given $\delta > 0$ let $\{ h_{1},\dots,h_{N(\delta)} \}$ be a $(\delta \| G \|_{P,2})$-net of $(\calH,e_{P})$.
Let $\pi_{\delta}: \calH \to \{ h_{1},\dots,h_{N(\delta)} \}$ be a map such that for each $h \in \calH$, $e_{P} (h,\pi_{\delta}(h)) \le \delta \| G \|_{P,2}$. 
Define $\U_{n,\delta}^{\sharp} := \U_{n}^{\sharp} \circ \pi_{\delta}$ and $\mathbb{W}_{P,\delta} := \mathbb{W}_{P} \circ \pi_{\delta}$. For any $f \in BL_{1}$, we have 
\begin{equation}
\begin{split}
| \E_{\mid X_{1}^{\infty}}[f(\U_{n}^{\sharp})] - \E[f(\mathbb{W}_{P})] | &\le | \E_{\mid X_{1}^{\infty}}[f(\U_{n}^{\sharp})] - \E_{\mid X_{1}^{\infty}} [f(\U_{n,\delta}^{\sharp})]| \\
&\quad + |\E_{\mid X_{1}^{\infty}}[f(\U_{n,\delta}^{\sharp})] - \E[f(\mathbb{W}_{P,\delta})]| \\
&\quad + | \E[f(\mathbb{W}_{P,\delta})] - \E[f(\mathbb{W}_{P})]|. 
\end{split}
\label{eq:BL_distance}
\end{equation}
The third term on the right hand side of (\ref{eq:BL_distance}) is bounded by $\E[2 \wedge \| \mathbb{W}_{P,\delta} - \mathbb{W}_{P} \|_{\calH}]$ and by construction $\mathbb{W}_{P}$ has sample paths almost surely uniformly  $e_{P}$-continuous, so that $\E[2 \wedge \| \mathbb{W}_{P,\delta} - \mathbb{W}_{P} \|_{\calH}] \to 0$ as $\delta \downarrow 0$ by the dominated convergence theorem. Since $\U_{n,\delta}^{\sharp}$ can be identified with a Gaussian vector of dimension $N(\delta)$ conditionally on $X_{1}^{\infty}$, by Lemma 3.7.46 in \cite{ginenickl2016}, 
the second term on the right hand side of (\ref{eq:BL_distance}) is bounded by
\[
c(\delta) \max_{1 \le j,k \le N(\delta)} | \hat{C}_{j,k} -\Cov_{P}(P^{r-1}h_{j},P^{r-1}h_{k}) |^{1/3}
\]
for some constant $c(\delta)$ that depends only on $\delta$, where 
\[
\hat{C}_{j,k} = n^{-1}\sum_{i=1}^{n}\{ U_{n-1,-i}^{(r-1)}(\delta_{X_{i}}h_{j}) - U_{n}(h_{j}) \}\{ U_{n-1,-i}^{(r-1)}(\delta_{X_{i}}h_{k}) - U_{n}(h_{k}) \}.
\]
From Step 5 of the proof of Theorem \ref{thm:coupling_gaussian_mulitiplier_bootstrap} and using the notation in the proof, we have 
\[
\max_{1 \le j,k \le N(\delta)} | \hat{C}_{j,k} - \Cov_{P}(P^{r-1}h_{j},P^{r-1}h_{k}) | \le 2 \Upsilon_{n}  + 2\| G \|_{P,2} \Upsilon_{n}^{1/2} + 2n^{-1/2} \| \mathbb{G}_{n} \|_{\breve{\calG} \cdot \breve{\calG}} + \| U_{n}(h) - P^{r}h \|_{\calH}^{2}.
\]
From the UCLT for the $U$-process established in the first paragraph, the last term on the right hand side is $o_{\Prob}(1)$. 
The function class $\breve{\calG} \cdot \breve{\calG}$ is weak $P$-Glivenko-Cantelli by Lemmas \ref{lem:entropy_conditional} and \ref{lem:pointwise_product} together with Theorem 2.4.3 in \cite{vandervaartwellner1996}, which implies that $n^{-1/2} \| \mathbb{G}_{n} \|_{\breve{\calG} \cdot \breve{\calG}} = o_{\Prob}(1)$. From Lemma \ref{lem:Upsilon_small} below, we also have $\Upsilon_{n} = o_{\Prob}(1)$. 

Finally, the first term on the right hand side of (\ref{eq:BL_distance}) is bounded by 
\[
\varepsilon + 2\Prob_{\mid X_{1}^{\infty}} (\| \U_{n}^{\sharp} \|_{\calH_{\delta}} > \varepsilon)
\]
for any $\varepsilon > 0$, where $\calH_{\delta} = \{ h-h' : h,h' \in \calH, e_{P}(h,h') < 2\delta \| G \|_{P,2} \}$. Let $\Sigma_{n,\delta} := \| n^{-1} \sum_{i=1}^{n}\{ U_{n-1,-i}^{(r-1)} (\delta_{X_{i}}h) - U_{n}(h) \}^{2} \|_{\calH_{\delta}}$. 
By Markov's inequality, 
\[
 \Prob_{\mid X_{1}^{\infty}} (\| \U_{n}^{\sharp} \|_{\calH_{\delta}} > \varepsilon) \le \frac{\E_{\mid X_{1}^{\infty}}[\| \U_{n}^{\sharp} \|_{\calH_{\delta}}]}{\varepsilon}. 
 \]
From Step 5 of the proof of Theorem \ref{thm:coupling_gaussian_mulitiplier_bootstrap}, 
\[
N(\calH_{\delta},d,2\tau \| H \|_{\Prob_{I_{n,r},2}}) \le N^{2}(\calH,\| \cdot \|_{\Prob_{I_{n,r},2}}, \tau \| H \|_{\Prob_{I_{n,r},2}})
\]
with $d(h,h') = \{ \E_{\mid X_{1}^{\infty}} [\{ \U_{n}^{\sharp} (h) - \U_{n}^{\sharp} (h') \}^{2}]\}^{1/2}$. 
Hence by  Dudley's entropy integral bound, we have 
 \[
 \E_{\mid X_{1}^{\infty}}[\| \U_{n}^{\sharp} \|_{\calH_{\delta}}] \lesssim \int_{0}^{\Sigma_{n,\delta}^{1/2}} \sqrt{1 + \lambda (\tau/\| H \|_{\Prob_{I_{n,r},2}})} d\tau
 \]
 up to a constant independent of $n$ and $\delta$, and $\| H \|_{\Prob_{I_{n,r},2}}^{2} = |I_{n,r}|^{-1}\sum_{I_{n,r}} H^{2}(X_{i_{1}},\dots,X_{i_{r}}) = \| H \|_{P^{r},2}^{2} + o_{\Prob}(1)$ by the law of large numbers for $U$-statistics \cite[Theorem 4.1.4]{delaPenaGine1999}. 
 From Step 4 of the proof of Theorem \ref{thm:coupling_gaussian_mulitiplier_bootstrap},
 \[
 \Sigma_{n,\delta} \le 8(\delta \| G \|_{P,2})^{2} + 8n^{-1/2} \| \G_{n} \|_{\breve{\calG} \cdot \breve{\calG}} +8 \Upsilon_{n},
 \]
 and the last two terms on the right hand side are $o_{\Prob}(1)$ while the first term can be arbitrarily small by taking $\delta$ sufficiently small. This implies that for any $\eta > 0$,
 \[
\lim_{\delta \downarrow 0} \limsup_{n \to \infty}\Prob \left ( \Prob_{\mid X_{1}^{\infty}} (\| \U_{n}^{\sharp} \|_{\calH_{\delta}} > \varepsilon)  > \eta \right ) = 0.
\]
Putting everything together, we conclude $d_{BL \mid X_{1}^{\infty}}(\U_{n}^{\sharp},\mathbb{W}_{P})^{*} \stackrel{\Prob}{\to} 0$, completing the proof.
\end{proof}

\begin{lem}
\label{lem:AEC}
Under the assumption of Theorem \ref{thm:bootstrap_UCLT}, the asymptotic equicontinuity condition (\ref{eq:AEC}) holds. 
\end{lem}

\begin{proof}[Proof of Lemma \ref{lem:AEC}]
For $\delta \in (0,1]$, let $\calH_{\delta}' = \{ h -h' : \| h - h' \|_{P^{r},2} < \delta \| H \|_{P^{r},2} \}$. By Markov's inequality, it suffices to show that 
\[
\lim_{\delta \downarrow 0} \limsup_{n \to \infty} \E[ \| \U_{n} \|_{\calH_{\delta}'}] = 0.
\]
We use Hoeffding's averaging \cite[Section 5.1.6]{serfling1980} to bound the expectation.  Let 
\[
S_{f}(x_{1},\dots,x_{n}) = \frac{1}{m} \sum_{i=1}^{m} f(x_{(i-1)r+1},\dots,x_{ir}) \ \text{with} \ m=\lfloor n/r \rfloor.
\]
Then we have 
\[
U_{n}(h) = \frac{1}{n!} \sum_{j_{1},\dots,j_{n}} S_{h}(X_{j_{1}},\dots,X_{j_{n}}),
\]
where $\sum_{j_{1},\dots,j_{n}}$ are taken over all permutations $j_{1},\dots,j_{n}$ of $1,\dots,n$. By Jensen's inequality, $\E[ \| \U_{n} \|_{\calH_{\delta}'}]$ is bounded by $\sqrt{n}\E[\| S_{h}(X_{1},\dots,X_{n}) - P^{r}h \|_{\calH_{\delta}'}]$. Since 
\[
S_{h}(X_{1},\dots,X_{n}) - P^{r}h = \frac{1}{m} \sum_{i=1}^{m} (h(X_{(i-1)r+1},\dots,X_{ir}) - P^{r}h)
\]
and since $(X_{(i-1)r+1},\dots,X_{ir}) , i=1,\dots,m$ are i.i.d., we can apply Theorem 5.2 in \cite{cck2014_empirical_process} to conclude that
\[
\E[ \| \U_{n} \|_{\calH_{\delta}'}] \lesssim \| H \|_{P^{r},2}J(\delta,\calH_{\delta}', 2H) + \frac{\| M_{r} \|_{\Prob,2}J^{2}(\delta,\calH_{\delta}',2H)}{\delta^{2} \sqrt{m}}
\]
up to a constant that depends only on $r$, where $M_{r} = \max_{1 \le i \le m} H(X_{(i-1)r+1},\dots,X_{ir})$ and the $J$ function is defined in \cite{cck2014_empirical_process}. From a standard calculation, $J(\delta,\calH_{\delta}', 2H) \lesssim J(\delta,\calH,H) = \int_{0}^{\delta}\sqrt{1+\lambda(\tau)} d\tau$ up to a universal constant and $\| M_{r} \|_{\Prob,2} = o(\sqrt{m})$ by $H \in L^{2}(P^{r})$ \citep[][Problem 2.3.4]{vandervaartwellner1996}. Hence we conclude 
\[
\limsup_{n \to \infty} \E[ \| \U_{n} \|_{\calH_{\delta}'}]  \lesssim  \| H \|_{P^{r},2}J(\delta,\calH,H)
\]
up to a constant that depends only on $r$, and by the dominated convergence theorem the right hand side is $o(1)$ as $\delta \downarrow 0$. This completes the proof. 
\end{proof}

\begin{lem}
\label{lem:Upsilon_small}
Under the assumption of Theorem \ref{thm:bootstrap_UCLT}, we have $\E[\Upsilon_{n}]= O(n^{-1})$ where $\Upsilon_{n}$ is defined in (\ref{eq:Upsilon}). 
\end{lem}

\begin{proof}[Proof of Lemma~\ref{lem:Upsilon_small}]
We begin with noting that 
\[
\E[\Upsilon_{n}] \le \E\left [ \E \left [  \left \| U_{n-1,-n}^{(r-1)} (\delta_{X_{n}}h)  - P^{r-1}(\delta_{X_{n}}h)  \right \|_{\calH}^{2}  \ \Big | \ X_{n} \right ] \right ].
\]
By Hoeffding's averaging \cite[Section 5.1.6]{serfling1980}, 
\[
U_{n-1,-n}^{(r-1)} (f) = \frac{1}{(n-1)!} \sum_{j_{1},\dots,j_{n-1}} T_{f}(X_{j_{1}},\dots,X_{j_{n-1}}),
\]
where $\sum_{j_{1},\dots,j_{n-1}}$ is taken over all permutations $j_{1},\dots,j_{n-1}$ of $1,\dots,n-1$, and 
\[
T_{f}(x_{1},\dots,x_{n-1}) = \frac{1}{m} \sum_{i=1}^{m} f(x_{(i-1)(r-1)+1},\dots,x_{i(r-1)}) \ \text{with} \ m=\lfloor (n-1)/(r-1) \rfloor.
\]
By Jensen's inequality,
\[
 \E \left [  \left \| U_{n-1,-n}^{(r-1)} (\delta_{X_{n}}h)  - P^{r-1}(\delta_{X_{n}}h)  \right \|_{\calH}^{2}  \ \Big | \ X_{n} \right ]  \le \E \left [  \left \| T_{\delta_{X_{n}h}} (X_{1},\dots,X_{n-1})  - P^{r-1}(\delta_{X_{n}}h)  \right \|_{\calH}^{2}  \ \Big | \ X_{n} \right ]. 
\]
By Corollary \ref{cor:entropy_conditional} and the condition of Theorem \ref{thm:bootstrap_UCLT}, for given $x \in S$,  
\[
\int_{0}^{1} \sqrt{ \sup_{Q} \log N(\delta_{x} \calH, \| \cdot \|_{Q,2},\tau \| \delta_{x} H \|_{Q,2})} \le \int_{0}^{1} \sqrt{\lambda (\tau)} d\tau < \infty.
\]
Hence, applying Theorem 2.14.1 in \cite{vandervaartwellner1996} conditionally on $X_{n}$, we have 
\begin{align*}
&\E \left [  \left \| T_{\delta_{X_{n}}h} (X_{1},\dots,X_{n-1})  - P^{r-1}(\delta_{X_{n}}h)  \right \|_{\calH}^{2}  \ \Big | \ X_{n} \right ] \lesssim n^{-1} \| \delta_{X_{n}} H \|_{P^{r-1},2}^{2}
\end{align*}
up to a constant independent of $n$. Since $\E[\| \delta_{X_{n}} H \|_{P^{r-1},2}^{2}] = \| H \|_{P^{r},2}^{2}$, we obtain the desired conclusion by Fubini's theorem. 
\end{proof}

\section{Gaussian approximation for suprema of $U$-processes indexed by general function classes}
\label{sec:Gaussian_approx_general}

In this section we derive Gaussian approximation error bounds for the $U$-process supremum indexed by general function classes. We obey the notation used in Sections \ref{sec:gaussian_approx}, \ref{sec:bootstrap} and~\ref{sec:local_maximal_inequalities}. We make the following assumptions on the function class $\calH$ and the distribution $P$. 

\begin{enumerate}
\item[(A1)] The function class $\calH$ is pointwise measurable. 

\item[(A2)] The envelope $H$ satisfies that $H \in L^{3}(P^{r})$. 

\item[(A3)] The class $\calG = P^{r-1} \calH = \{ P^{r-1} h : h \in \calH \}$ is $P$-pre-Gaussian, i.e., there exists a tight Gaussian random variable $W_{p}$ in $\ell^{\infty}(\calG)$ with mean zero and covariance function $\E[W_{P}(g) W_{P}(g')] = \Cov(g(X_{1}), g'(X_{1}))$ for all $g,g' \in \calG$. 
\end{enumerate}

Conditions (A1)--(A3) are parallel with the corresponding conditions in \cite{cck2014_empirical_process}.
Condition (A1) is the same as Condition (PM) in Section~\ref{sec:gaussian_approx}. 
Condition (A3) is a high-level assumption that is implied by Condition (VC) in Section~\ref{sec:gaussian_approx}.

For $\varepsilon > 0$, define  $\mathcal{N}_{n}(\varepsilon) = \log(N(\calG, \|\cdot\|_{P,2}, \varepsilon \| G \|_{P,2}) \vee n)$
with $G= P^{r-1}H$. 
Under Condition (A3), $\calG$ is totally bounded for the intrinsic pseudometric induced by $\|\cdot\|_{P,2}$ and  $\mathcal{N}_{n}(\varepsilon)$ is finite for every $\varepsilon \in (0,1]$. In addition, the Gaussian process $W_{P}$ extends to the linear hull of $\calG$ in such a way that $W_{P}$ has linear sample paths (see e.g., Theorem 3.7.28 in \cite{ginenickl2016}).
For $\varepsilon \in (0,1], \gamma \in (0,1)$, and $\kappa > 0$, define  
\begin{align*}
\Delta_n(\varepsilon, \gamma, \kappa) :=& \gamma^{-1} \E[\|\G_{n}\|_{\calG_{\varepsilon}}] + \E[\|W_{P}\|_{\calG_{\varepsilon}}] + \sqrt{\log(1/\gamma)} \varepsilon \| G \|_{P,2} + n^{-1/6} \gamma^{-1/3} \kappa \mathcal{N}_{n}^{2/3}(\varepsilon) \\
& \qquad + n^{-1/4} \gamma^{-1/2} (\E\|\G_{n}\|_{\breve{\calG} \cdot \breve{\calG}})^{1/2} \mathcal{N}_{n}^{1/2}(\varepsilon) + n^{1/2} \gamma^{-1} \sum_{k=2}^{r} \E[\|U_{n}^{(k)}(\pi_{k} h)\|_{\calH}], \\
\delta_{n}(\varepsilon, \gamma, \kappa) :=& {1 \over 5} P \left [ (\breve{G}/\kappa)^{3} {1}(\breve{G}/\kappa > c \gamma^{-1/3} n^{1/3} \mathcal{N}_{n}(\varepsilon)^{-1/3}) \right ], 
\end{align*}
where $\calG_{\varepsilon} = \{g-g' : g, g' \in \calG, \|g-g'\|_{P,2} < 2\varepsilon \|G\|_{P,2}\}$, $\breve{\calG} \cdot \breve{\calG} = \{gg' : g, g' \in \breve{\calG}\}$, $\breve{\calG} = \{g, g-Pg : g \in \calG\}$, and $\breve{G} = G + PG$. Here $c > 0$ is some universal constant. 
Below is an abstract (yet general) version of the Gaussian coupling bound. 

\begin{prop}[Abstract Gaussian coupling bound]
\label{prop:coupling_bounds_general_function_classes}
Let $Z_{n} = \sup_{h \in \calH} \U_{n}(h)/r$. Suppose that Conditions (A1)--(A3) hold. Let $\kappa > 0$ be any positive constant such that $\kappa^{3} \ge \E[\|n^{-1}\sum_{i=1}^{n}|g(X_{i}) - P g|^{3}\|_{\calG}]$. Then, for every $n \ge r+1$, $\varepsilon \in (0,1]$, and $\gamma \in (0,1)$, one can construct a random variable $\tilde{Z}_{n} = \tilde{Z}_{n,\varepsilon,\gamma,\kappa}$ such that $\mathcal{L}(\tilde{Z}_{n}) = \mathcal{L}(\sup_{g \in \calG} W_P(g))$ and
\[
\Prob \left( |Z_{n} - \tilde{Z}_{n}| >  C_{1} \Delta_n(\varepsilon, \gamma,\kappa) \right) \le  \gamma \{1 + \delta_{n}(\varepsilon, \gamma,\kappa)\} + {C_{2} \log{n} \over n},
\]
where $C_{1} = C_{1,r}$ is a constant depending only on $r$ and $C_{2}$ is a universal constant. 
\end{prop}

The proposition should be considered as an extension of Theorem 2.1 in \cite{cck2014_empirical_process} to the $U$-process. 
To apply the above proposition, we need to derive bounds on 
\begin{equation} 
\label{eqn:quantities_in_coupling_bounds_general_function_classes}
\begin{split}
&\E[\|\G_{n}\|_{\calG_{\varepsilon}}], \; \E[\|W_{P}\|_{\calG_{\varepsilon}}], \; \E\left [\left\|n^{-1}\sum_{i=1}^{n}|g(X_{i})-Pg|^{3}\right\|_{\calG}\right], \\
&\E[\|\G_{n}\|_{\breve{\calG} \cdot \breve{\calG}}], \; \mbox{and } \E[\|U_{n}^{(k)}(\pi_{k} h)\|_{\calH}, k = 2,\dots,r,
\end{split}
\end{equation}
which can be derived under some moment conditions on $H$ and by using the uniform entropy integrals $J_{k}(\delta), k=1,\dots,r$ defined in~\eqref{eqn:uniform_entropy_integral} (cf. Lemma 2.2 in \cite{cck2014_empirical_process} and our Theorem~\ref{lem:local_max_ineq_dengenerate_uprocess}), where the latter can be simplified in terms of the VC characteristics $(A, v)$ for a VC type function class (cf. the proof of Corollary~\ref{lem:local_max_ineq_dengenerate_uprocess_VCtype}).

\begin{proof}[Proof of Proposition~\ref{prop:coupling_bounds_general_function_classes}]
The proof is based on a modification to that of Theorem 2.1 in \cite{cck2014_empirical_process}.  
In this proof $C$ denotes a generic universal constant; the value of $C$ may change from place to place.  Let $\{g_{k}\}_{k=1}^{N}$ be a minimal $\varepsilon \|G\|_{P,2}$-net of $(\calG, \|\cdot\|_{P,2})$ with $N := N(\calG, \|\cdot\|_{P,2}, \varepsilon\|G\|_{P,2})$. By the definition of $\calG$, each $g_{k}$ corresponds to a kernel $h_{k} \in \calH$ such that $g_{k}=P^{r-1}h_{k}$. Recall the Hoeffding decomposition $\U_{n}(h) = r \G_{n}(P^{r-1}h) + \sqrt{n} \sum_{k=2}^{r} {r \choose k} U_{n}^{(k)}(\pi_{k}h)$, 
where $\G_{n}(P^{r-1} h) = n^{-1/2} \sum_{i=1}^{n} (P^{r-1}h (X_{i}) - P^{r}h)$. Let $L_{n}=\sup_{g \in \calG} \G_{n}(g)$ and $R_{n}=\| r^{-1} \sqrt{n} \sum_{k=2}^{r} {r \choose k} U_{n}^{(k)}(\pi_{k}h)\|_{\calH}$. Then $|Z_{n}-L_{n}| \le R_{n}$. Define 
\[
L_{n}^{\varepsilon} = \max_{1 \le j \le N} \G_{n}(g_{j}), \; \tilde{Z} = \sup_{g \in \calG} W_{P}(g), \; \tilde{Z}^{\varepsilon} = \max_{1 \le j \le N} W_{P}(g_{j}). 
\]
We note that $|L_{n}-L_{n}^{\varepsilon}| \le \|\G_{n}\|_{\calG_{\varepsilon}}$ and $|\tilde{Z}-\tilde{Z}^{\varepsilon}| \le \|W_{P}\|_{\calG_{\varepsilon}}$. By Corollary 4.1 in \cite{cck2014_empirical_process}, we have for every  $B \in \calB(\R)$ and $\delta > 0$, 
\[
\Prob(L_{n}^{\varepsilon} \in B) - \Prob(\tilde{Z}^{\varepsilon} \in B^{16\delta}) \le C \delta^{-2} \{T_{1}+\delta^{-1}(T_{2}+T_{3}) \mathcal{N}_{n}(\varepsilon)\} \mathcal{N}_{n}(\varepsilon) + C n^{-1} \log{n},
\]
where 
\begin{align*}
T_{1} =& n^{-1} \E \left[ \max_{1 \le j,k \le N} \left| \sum_{i=1}^{n} (g_{j}(X_{i})-P g_{j}) (g_{k}(X_{i})-P g_{k}) - P(g_{j}-P g_{j}) (g_{k}-P g_{k})\right| \right], \\
T_{2} =& n^{-3/2} \E \left[ \max_{1 \le j \le N} \sum_{i=1}^{n} |g_{j}(X_{i}) - P g_{j}|^{3} \right], \\
T_{3} =& n^{-1/2} \E \left[ \max_{1 \le j \le N} |g_{j}(X_{1})-Pg_{j}|^{3}  \cdot 1\left( \max_{1 \le j \le N} |g_{j}(X_{1})-Pg_{j}| > \delta \sqrt{n} \mathcal{N}_{n}(\varepsilon)^{-1} \right) \right].
\end{align*}
Observe that $T_{1} \le n^{-1/2} \E[\|\G_{n}\|_{\breve{\calG} \cdot \breve{\calG}}]$, $T_{2} \le n^{-1/2} \kappa^{3}$, and $T_{3} \le n^{-1/2} P[\breve{G}^{3} 1(\breve{G}>\delta \sqrt{n} \mathcal{N}_{n}(\varepsilon)^{-1})]$. Thus choosing 
\[
\delta \ge C \max \left\{ \gamma^{-1/2} n^{-1/4} (\E[\|\G_{n}\|_{\breve{\calG} \cdot \breve{\calG}}])^{1/2} \mathcal{N}_{n}^{1/2}(\varepsilon), \; \gamma^{-1/3} n^{-1/6} \kappa \mathcal{N}_{n}^{2/3}(\varepsilon) \right\}, 
\]
we have 
\[
\Prob(L_{n}^{\varepsilon} \in B) \le \Prob(\tilde{Z}^{\varepsilon} \in B^{16\delta}) + {2 \gamma \over 5} + {\gamma \over 5} \kappa^{-3} P[\breve{G}^{3} 1(\breve{G}>\delta \sqrt{n} \mathcal{N}_{n}(\varepsilon)^{-1})] + {C \log{n} \over n}. 
\]
Since $\delta \ge c \gamma^{-1/3} n^{-1/6} \kappa \mathcal{N}_{n}^{2/3}(\varepsilon)$, we have 
\[
P[\breve{G}^{3} 1(\breve{G}>\delta \sqrt{n} \mathcal{N}_{n}(\varepsilon)^{-1})] \le P[\breve{G}^{3} 1(\breve{G}/\kappa>c \gamma^{-1/3} n^{1/3} \mathcal{N}_{n}(\varepsilon)^{-1/3})].
\]
Conclude that with $\eta_{n} = (\gamma / 5) P[(\breve{G}/\kappa)^{3} 1(\breve{G}/\kappa>c \gamma^{-1/3} n^{1/3} \mathcal{N}_{n}(\varepsilon)^{-1/3})]$,
\[
\Prob(L_{n}^{\varepsilon} \in B) \le \Prob(\tilde{Z}^{\varepsilon} \in B^{16\delta}) + {2\gamma \over 5} + \eta_{n} + {C \log{n} \over n}. 
\]
Next, we will bound $\|\G_{n}\|_{\calG_{\varepsilon}}$ and $\|W_{P}\|_{\calG_{\varepsilon}}$. 
By Markov's inequality, with probability at least $1-\gamma/5$, 
\[
\|\G_{n}\|_{\calG_{\varepsilon}} \le 5\gamma^{-1}\E[\|\G_{n}\|_{\calG_{\varepsilon}}] =: a. 
\]
Further, by the Borell-Sudakov-Tsirel'son inequality (see Theorem 2.5.8 in \cite{ginenickl2016}), with probability at least $1-\gamma/5$, we have 
\[
\|W_{P}\|_{\calG_{\varepsilon}} \le \E[\|W_{P}\|_{\calG_{\varepsilon}}] + 2 \varepsilon \| G \|_{P,2} \sqrt{2\log(5/\gamma)} =: b.
\]
Therefore, for every $B \in \calB(\R)$, 
\begin{align*}
\Prob(Z_{n} \in B) \le& \Prob(L_{n} \in B^{5\gamma^{-1} \E[R_{n}]}) + {\gamma \over 5} \le \Prob(L_{n}^{\varepsilon} \in B^{a+5\gamma^{-1} \E[R_{n}]}) + {2\gamma \over 5} \\
\le& \Prob(\tilde{Z}^{\varepsilon} \in B^{a+16\delta+5\gamma^{-1} \E[R_{n}]}) + {4\gamma \over 5} + \eta_{n} + {C \log{n} \over n}\\
\le& \Prob(\tilde{Z} \in B^{a+b+16\delta+5\gamma^{-1} \E[R_{n}]}) + \gamma + \eta_{n} + {C \log{n} \over n}.
\end{align*}
The conclusion of the proposition follows from the Strassen-Dudley theorem (see Theorem~\ref{thm:strassen-dudley}). 
\end{proof}


\section{Alternative tests for concavity/convexity and monotonicity of regression functions}
\label{sec:comments_on_alternative_tests}

We will obey the setting of Example \ref{exmp:test_curv}.

\subsection{Alternative tests for concavity/convexity of regression function $f$}
Instead of the original localized simplex statistic (\ref{eq:localized_simplex_statistic}) proposed in \cite{abrevaya-wei2005_JBES}, we may consider the following modified version:
\[
\tilde{U}_{n}(x) = {1 \over |I_{n,m+2}|} \sum_{(i_{1},\dots,i_{m+2}) \in I_{n,m+2}} \tilde{\varphi}(V_{i_{1}},\dots,V_{i_{m+2}}) \prod_{k=1}^{m+2} L_{b_n}(x-X_{i_{k}}),
\]
where $\tilde{\varphi} (v_{1},\dots,v_{m+2}) = 1\{ (x_{1},\dots,x_{m+2}) \in \calD \} w(v_{1},\dots,v_{m+2})$,
and test concavity or convexity of $f$ if the scaled supremum or infimum of $\tilde{U}_{n}$ is large or small, respectively. These alternative tests will work without the symmetry assumption on the conditional distribution of $\varepsilon$, which is maintained in \cite{abrevaya-wei2005_JBES}. Our results below also cover these alternative tests. 

\subsection{Alternative tests for monotonicity of regression function $f$}
\cite{chetverikov2012} considers testing monotonicity of the regression function $f$ without the assumption that the error term  $\varepsilon$ is independent of $X$. \cite{chetverikov2012} studies, e.g.,  $U$-statistics given by replacing $\sign (Y_{j}-Y_{i})$ in (\ref{eq:gsv_statistic}) by $Y_{j}-Y_{i}$, and the test statistic defined by taking the maximum of such $U$-statistics over a discrete set of design points and bandwidths whose cardinality may grow with the sample size (indeed, the cardinality can be much larger than the sample size).  His analysis is conditional on $X_{i}$'s, and he cleverly avoids  $U$-process machineries and applies directly high-dimensional Gaussian and bootstrap approximation theorems developed in \cite{cck2013}. It should be noted that \cite{chetverikov2012} considers more general test statistics and studies multi-step procedures to improve on powers of his tests. 

Another related test for regression monotonicity is based on the local linear rank statistics \cite{Dumbgen2002_JNS}. Let $R_{mk}(i) = \sum_{j=m+1}^{k} 1(Y_{j} \le Y_{i})$ be the local rank of $Y_{i}$ among $Y_{m+1},\dots,Y_{k}$. In \cite{Dumbgen2002_JNS}, D\"umbgen considers a test for monotone trend of $f$ (with fixed design points $X_{1},\dots,X_{n}$) via the local linear rank statistics 
\[
T_{mk} = \sum_{i=m+1}^{k} \beta \left( {i-m \over k-m+1} \right) q \left( {R_{mk}(i) \over k-m+1} \right), \quad 0 \le m < k \le n, 
\]
where $\beta$ and $q$ are functions on $(0,1)$ such that: 1) $\beta(1-u)=-\beta(u)$ and $q(1-u)=-q(u)$ for $u \in (0,1)$; 2) $\beta(\cdot)$ and $q(\cdot)$ are nondecreasing on $(0,1)$. Then \cite{Dumbgen2002_JNS} proposes the multiscale test statistic 
\[
T = \max_{0 \le m < k \le n} (s_{k-m} |T_{mk}| - c_{k-m}),
\]
where $s_{i}$ and $c_{i}$ are properly chosen nonnegative numbers. For the special case of the Wilcoxon score function $q(u) = 2u-1$ and $\beta(u) = q(u)$, one can write 
\[
T_{mk} = {2 \over (k-m+1)^{2}} \sum_{m < i < j \le k} (j-i) \sign(Y_{j}-Y_{i}). 
\]
The statistic $T_{mk}$ is related to our test statistic $\check{U}_{n}(x)$ with $L(u) = 1(u \in [-1,1])$, namely $T_{mk}$ and $\check{U}_{n}(x)$ are (local) $U$-statistics with kernels $(j-i) \sign(Y_{j}-Y_{i})$ and $\sign(X_{i}-X_{j}) \sign(Y_{j}-Y_{i})$, respectively. Thus for a given sequence of bandwidths $b_{n}$, our monotonicity test based on the $U$-process $\check{U}_{n}(x)$ can be viewed as a single-scale test $T_{mk}$ with $(k-m)/n = 2 b_{n}$ in D\"umbgen's sense. In particular, both $T_{0n}$ and $\check{U}_{n}(x)$ with $b_{n} = 1$ quantify the monotonicity on the global scale. In addition, the ``uniform-in-bandwidth" type results for our $U$-process approach in Section~\ref{subsec:uniformly_valid_JMB_bandwidth} can be viewed as the multiscale analog $T$ of $T_{mk}$ with the Wilcoxon score function. Nevertheless, since \cite{Dumbgen2002_JNS} considers the fixed design points, $T_{mk}$ is a local $U$-statistic on $Y_{i}$'s and $\check{U}_{n}(x)$ is a local $U$-statistic on $(X_{i}, Y_{i})$'s. Our analysis (which requires a Lebesgue density on $X$) is not directly applicable for the local linear rank statistics of \cite{Dumbgen2002_JNS}.

\bibliographystyle{plain}
\bibliography{uprocess}

\end{document}